\documentclass[11pt]{amsart}
\usepackage{graphicx}

\usepackage[english]{babel}
\usepackage[latin1]{inputenc}
\usepackage{amsfonts}
\usepackage{amsthm}
\usepackage{amsmath}
\usepackage{amssymb}
\usepackage{amscd}
\usepackage{verbatim}
\usepackage{multicol}
\usepackage{tabularx}
\usepackage{amssymb}
\usepackage[all]{xy}
\usepackage{tikz}
\usepackage[parfill]{parskip}

\theoremstyle{definition}
\newtheorem{thm}{Theorem}[section]
\newtheorem{prop}[thm]{Proposition}
\newtheorem{cor}[thm]{Corollary}

\newtheorem{dfn}[thm]{Definition}

\DeclareMathOperator{\Ad}{\mathrm{Ad}}

\DeclareMathOperator{\ad}{\mathrm{ad}}
\DeclareMathOperator{\Cent}{\mathrm{Cent}}
\DeclareMathOperator{\vspan}{\mathrm{span}}
\DeclareMathOperator{\diag}{\mathrm{diag}}

\def\ord#1^#2{#1$^{\text{#2}}$}

\def\lie#1{\mathfrak{#1}}

\def\hlie#1{\hat{\mathfrak{#1}}}

\def\uqr#1^#2{\text{$U_q^{#2}(\lie #1)$}}

\def\uqhr#1^#2{\text{$U_q^{#2}(\hlie #1)$}}
\def\us#1^#2{\text{$U_{\xi}^{#2}(\lie #1)$}}
\def\ush#1^#2{\text{$U_{\xi}^{#2}(\hlie #1)$}}
\def\dus#1^#2{\text{$\dot{U}_{\xi}^{#2}(\lie #1)$}}
\def\dush#1^#2{\text{$\dot{U}_{\xi}^{#2}(\hlie #1)$}}

\def\opl_#1^#2{\text{\scriptsize$\bigoplus\limits_{\text{\footnotesize$#1$}}^{\text{\footnotesize$#2$}}$}}
\def\otm_#1^#2{\text{\scriptsize$\bigotimes\limits_{\text{\footnotesize$#1$}}^{\text{\footnotesize$#2$}}$}}

\renewcommand{\thefootnote}

\begin{document}


\title[]{
{\Large Geodesic orbit spaces in real flag manifolds}
}

\author[]{Brian Grajales, Lino Grama and Caio J. C. Negreiros
}

\address{IMECC-Unicamp, Departamento de Matem\'{a}tica. Rua S\'{e}rgio Buarque de Holanda,
651, Cidade Universit\'{a}ria Zeferino Vaz. 13083-859, Campinas - SP, Brazil.} 

\dedicatory{Dedicated to Professor Karen Uhlenbeck on the occasion of her 75th birthday.}

\begin{abstract}
We describe the invariant metrics on real flag manifolds and classify those with the following property: every geodesic is the orbit of a one-parameter subgroup.  Such a metric is called g.o. (geodesic orbit). In contrast to the complex case, on real flag manifolds the isotropy representation can have equivalent submodules, which makes invariant metrics depend on more parameters and allows us to find more cases in which non-trivial g.o. metrics exist.
\end{abstract}

\maketitle
\tableofcontents

\section{Introduction}

$\indent$The theory of harmonic maps is a classic topic in Mathematics. This theory lives in the intersection of several areas of Mathematics like analysis of PDE, calculus of variations, differential geometry and so on. Due to this interdisciplinarity, the theory of harmonic maps archives several deep results in Mathematics and Theoretical Physics. See for instance, \cite{an5}, \cite{sacks-uhle}, \cite{S-U}, \cite{ulem1}.

Based on the article \cite{ulem1}  of Professor Uhlenbeck and a list of problems 
at the end of  the article, the third author of this paper followed her
suggestion to discuss the Question 4 in this list as his thesis problem.
This problem concerns to the study of harmonic maps on flag manifolds
and is based on the interplay between the Geometry of non-necessarily
symmetric spaces and the Analysis of flag manifolds. This study is
linked deeply with Lie Theory, and continuing to motivate the study of other
variational problems. 

In \cite{neg-thesis} and \cite{neg-indiana} was studied harmonic maps into flag manifolds, using holomorphic horizontal maps and by exploring the (almost) complex geometry of the complex flag manifolds (see also \cite{caio-estab}). 

In the theory of harmonic maps it is well know that the geometry of the target space plays a fundamental role. Therefore, understand the geometry (Riemannian and Hermitian) of flag manifolds motivated by the study harmonic maps is a natural step. For instance the understanding of $(1,2)$-symplectic structures on flag manifolds is useful to provide examples of harmonic maps and motivate a question about its classification on complex flag manifolds. The answer of such question culminate with the complete characterization of the invariant Hermitian geometry of full flag manifolds in \cite{an7}.

In this paper we deal with Riemannian geometry of {\em real flag manifolds}. By real flag manifolds we will refer the coset (homogeneous) space given the the split real form of a complex Lie group by some parabolic subgroup (see Section \ref{prelim} for details). We are interested in the problem of classifications of {\em homogeneous geodesics}, namely, geodesics given by an orbit of 1-parameter subgroup of the isometry group. References about the study of homogeneous geodesics on {\em complex} flag manifolds are \cite{DA2}, \cite{caio-lino-nir}, \cite{dP-G}.

An interesting class of riemannian homogeneous spaces are the so called {\em g.o spaces}, namely a homogeneous space such that {\em every} geodesic is a homogeneous geodesic. Examples of g.o. spaces are compact Lie groups equipped with the bi-invariant metric, naturally reductive homogeneous space, the normal metric on homogeneous space of compact Lie groups and so on. The classification of {\em complex} flag manifolds which are {\em g.o. spaces} is given in \cite {DA2}.

In this paper we provide a description of invariant metrics on real flag manifolds. As an application we prove our main result: we give the classification of the real flag manifolds of classical Lie groups admitting invariant riemannian metrics such that every geodesic is an homogeneous geodesic, that is, its become a {\em g.o. space}. 

A remarkable difference between {\em real} and {\em complex} flag manifolds lies in the isotropy representation. Such isotropy representation is essential in order to describe invariant tensors on homogeneous spaces (e.g. riemannian metrics, almost complex structures, differential forms). In the case of complex flag manifolds, the isotropy representation decomposes into irreducible {\em non-equivalent} components and for real flag manifolds one can have {\em equivalent} components of the isotropy representation.  A reference for the description of the isotropy representation of real flag manifolds is given in \cite{PSM}. For recent results concerning to the invariant geometry of real flag manifolds we suggest \cite{dB-SM} and \cite{tese-brian}.

The paper is organized as follows: in Section 2 we review the basics facts about real flag manifolds. In Section 3 we describe the invariant metrics on real flag manifolds associated to classical Lie groups, and in Section 4 we prove our main result: the classification of the real flag manifolds which are {\em g.o. spaces}.

{\bf Acknowledgment:} The third named author of this paper wants to express his deep gratitude to Professor Uhlenbeck for all her immense support through all these years. Professor Uhlenbeck  was the Ph.D. Thesis advisor of prof. Negreiros at University of Chicago (USA) who was the Ph.D. Thesis advisor of prof. Grama at University of Campinas (Brazil) who was the Ph.D. Thesis advisor of B. Grajales also at University of Campinas.

\section{Preliminaries: Real Flag Manifolds}
\label{prelim}
Let $\mathfrak{g}$ be a non-compact, simple real Lie algebra. We consider the case where $\mathfrak{g}$ is a split real form of a complex Lie algebra. A generalized flag manifold of $\mathfrak{g}$ is the homogeneous space $\mathbb{F}_\Theta=G/P_\Theta$ where $G$ is a connected Lie group with Lie algebra $\mathfrak{g}$ and $P_\Theta\subset G$ is a parabolic subgroup. The Lie algebra $\mathfrak{p}_\Theta$ of $P_\Theta$ is a parabolic subalgebra, which is the direct sum of the eigenspaces  associated with the non-negative eigenvalues of $\ad(H_\Theta)$, where $H_\Theta\in\mathfrak{g}$ is an element chosen in an appropriate way. If $K\subset G$ is a maximal compact subgroup and $K_\Theta=K\cap P_\Theta$ then we have and identification $\mathbb{F}_\Theta=K/K_\Theta.$ It is a known fact that flag manifolds are reductive homogeneous spaces, and therefore we have a reductive decomposition of the Lie algebra of $K$, given by 
\begin{center}
$\mathfrak{k}=\mathfrak{k}_\Theta\oplus\mathfrak{m}_\Theta,$
\end{center}
where $\mathfrak{m}_\Theta$ is an $\Ad(K_\Theta)-$invariant complement of $\mathfrak{k}_\Theta.$

For an alternative description of the parabolic subalgebra we consider $\mathfrak{g}=\mathfrak{k}\oplus\mathfrak{s}$ a Cartan decomposition and $\mathfrak{a}\subset\mathfrak{s}$ a maximal abelian subalgebra. Denote by $\Pi$ the associated set of roots and by
\begin{center}
$\mathfrak{g}=\displaystyle\mathfrak{g}_0\oplus \bigoplus\limits_{\alpha\in\Pi}\mathfrak{g}_{\alpha},$
\end{center}
the corresponding root space decomposition. Fixing a set $\Pi^+$ of positive roots let $\Sigma$ be the corresponding set of simple roots. Any $\Theta\subseteq\Sigma$ defines a parabolic subalgebra
\begin{center}
$\mathfrak{p}_\Theta=\displaystyle\mathfrak{g}_0\oplus \bigoplus\limits_{\alpha\in\Pi^+}\mathfrak{g}_{\alpha}\oplus \bigoplus\limits_{\alpha\in\langle\Theta\rangle^-}\mathfrak{g}_{\alpha},$
\end{center}
where $\langle\Theta\rangle^-$ is the set of negative roots generated by $\Theta.$ We say that $H_\Theta\in\mathfrak{a}$ is \textit{characteristic} for $\Theta$ if $\alpha(H_\Theta)\geq 0$ for every $\alpha\in\Theta$ and $\Theta=\{\alpha\in\Sigma:\alpha(H_\Theta)=0\}$. The subalgebra 
\begin{center}
$\mathfrak{n}_\Theta^-=\displaystyle\bigoplus\limits_{\alpha\in\Pi^-\setminus\langle\Theta\rangle^-}\mathfrak{g}_\alpha,$
\end{center}
is identified with the tangent space of $\mathbb{F}_\Theta$ at the origin $eP_\Theta.$ If $\mathfrak{z}_\Theta=\Cent_{\mathfrak{g}}(H_\Theta)$, then the adjoint representation of $\mathfrak{z}_\Theta$ on $\mathfrak{n}_\Theta^-$ is completely reducible and we can decompose 
\begin{center}
$\mathfrak{n}_\Theta^-=\displaystyle\bigoplus\limits_{\sigma}V_\Theta^{\sigma},$
\end{center}
into $\mathfrak{z}_\Theta-$invariant, irreducible and non-equivalent subspaces.

With this notation we have that $K_\Theta=\Cent_K(H_\Theta)$ and $\mathfrak{k}_\Theta=\Cent_{\mathfrak{k}}(H_\Theta).$ The tangent space at $o=eK_\Theta\in\mathbb{F}_\Theta$ is identified with $\mathfrak{m}_\Theta$ and there exists a one-to-one between $G-$invariant tensors on $\mathbb{F}_\Theta$ and tensors on $T_o\mathbb{F}_\Theta\approx\mathfrak{m}_\Theta$ which are invariant with respect to the isotropy representation of $K_\Theta.$ If $H_\alpha$, $\alpha\in\Sigma$ and $X_\alpha\in\mathfrak{g}_\alpha,$ $\alpha\in\Pi$ is Weyl basis for $\mathfrak{g}$, we identify $\mathfrak{n}_\Theta^-$ with $\mathfrak{m}_\Theta$ via
\begin{center}
$X_\alpha\longmapsto X_\alpha-X_{-\alpha},$ $\alpha\in\Pi^-\setminus\langle\Theta\rangle^-.$
\end{center}
The $\mathfrak{z}_\Theta-$invariant subspaces $V^{\sigma}_\Theta$ are $K_\Theta-$invariant but not necessarily $K_\Theta-$irreducible. The $K_\Theta-$inva- riant subspaces of each $V^{\sigma}_\Theta$ and their equivalences by the adjoint representation of $K_\Theta$ are completely described in \cite{PSM}.

\section{Invariant Metrics on Real Flag Manifolds}
It is known that there exists a one-to-one correspondence between $K-$invariant metrics on $\mathbb{F}_\Theta=K/K_\Theta$ and $\Ad(K_\Theta)-$invariant inner products $g$ on $\mathfrak{m}_\Theta$, that is

\begin{equation}
g(\Ad(k)X,\Ad(k)Y)=g(X,Y) \text{ for all } k\in K_\Theta,\ X,Y\in\mathfrak{m}_\Theta.
\end{equation}

We fix an $\Ad(K)-$invariant inner product $(\cdot,\cdot)$ on $\mathfrak{k}$ such that the reductive decomposition $\mathfrak{k}=\mathfrak{k}_\Theta\oplus\mathfrak{m}_\Theta$ is $(\cdot,\cdot)-$orthogonal. Given another $\Ad(K_\Theta)-$invariant inner product $g$, there exists a unique $(\cdot,\cdot)-$self-adjoint, positive operator $A:\mathfrak{m}_\Theta\longrightarrow\mathfrak{m}_\Theta$ commuting with $\Ad(k)$ for all $k\in K_\Theta$ such that

\begin{equation}
g(X,Y)=(AX,Y) \text{ for all } X,Y\in\mathfrak{m}_\Theta.
\end{equation}

Any $\Ad(K_\Theta)-$invariant inner product is determined by such an operator. We call $A$ the \textit{metric operator} corresponding to $g.$ We will make no distinction between $g$ and its metric operator $A.$ Since $K_\Theta$ is compact, the adjoint representation of $K_\Theta$ in $\mathfrak{m}_\Theta$ induces a $(\cdot,\cdot)-$orthogonal splitting

\begin{equation}\label{3}
\mathfrak{m}_\Theta=\bigoplus\limits_{i=1}^s\mathfrak{m}_i
\end{equation}

of $\mathfrak{m}_\Theta$ into $K_\Theta-$invariant, irreducible submodules $\mathfrak{m}_i$, $i=1,...,s.$  We say that the submodules $\mathfrak{m}_i$ and $\mathfrak{m}_j$ are \textit{equivalent} if there exists  an $\Ad(K_\Theta)-$equivariant isomorphism $T:\mathfrak{m}_i\longrightarrow\mathfrak{m}_j$, that is, $\left.\Ad(k)\right|_{\mathfrak{m}_j}\circ T=T\circ \left.\Ad(k)\right|_{\mathfrak{m}_i}$, for all $k\in K_\Theta.$ Evidently, this equivalence relation induces a partition 

\begin{center}
$\{\mathfrak{m}_1,...,\mathfrak{m}_s\}=C_1\cup ...\cup C_S \text{ with } S\leq s,$
\end{center}

therefore, we have a new $(\cdot,\cdot)-$orthogonal splitting

\begin{equation}\label{4}
\mathfrak{m}_\Theta=\bigoplus\limits_{i=1}^SM_i
\end{equation}

where

\begin{center}
$M_i=\displaystyle\bigoplus\limits_{\mathfrak{m}_j\in C_i}\mathfrak{m}_j,$ \ \  $i=1,...,S.$
\end{center}

We call each $M_i$ an \textit{isotypical summand} of the decomposition \eqref{3}. We consider a $(\cdot,\cdot)-$orthogonal ordered basis $\mathcal{B}=\mathcal{B}_1\cup...\cup\mathcal{B}_S$ adapted to the decomposition \eqref{4} such that for every $i$, all vectors in $\mathcal{B}_i$ have the same norm with respect to $(\cdot,\cdot)$. Then, any metric operator $A$ can be written in the basis $\mathcal{B}$ as a block-diagonal matrix of the form 

\begin{equation}\label{5}
[A]_{\mathcal{B}}=\left(\begin{array}{cccc}
[A|_{M_1}]_{\mathcal{B}_1} & 0 & \dots & 0\\
0 & [A|_{M_2}]_{\mathcal{B}_2} & \dots & 0\\
\vdots & \vdots & \ddots & \vdots\\
0 & 0 & \dots & [A|_{M_S}]_{\mathcal{B}_S}\\
\end{array}\right).
\end{equation}

If $M_i=\mathfrak{m}_{i_1}\oplus...\oplus\mathfrak{m}_{i_n}$, then $[A|_{M_i}]_{\mathcal{B}_i}$ has the form

\begin{equation}\label{6}
[A|_{M_i}]_{\mathcal{B}_i}=\left(\begin{array}{cccc}
\mu_1 \text{I}_{\mathfrak{m}_{i_1}} & B_{21}^T & \dots & B_{n1}^T\\
B_{21} & \mu_2\text{I}_{\mathfrak{m}_{i_2}} & \dots & B_{n2}^T\\
\vdots & \vdots & \ddots & \vdots\\
B_{n1} & B_{n2} & \dots & \mu_n\text{I}_{\mathfrak{m}_{i_n}}\\
\end{array}\right)
\end{equation}

where $B_{ab}^T$ represents the transpose of $B_{ab}$ and $\mu_1,...,\mu_n>0.$ We shall use the facts above and the results in \cite{PSM} about the isotropy representation on real flag manifolds to obtain the invariant metrics in these manifolds.\\

\subsection{Flags of $A_l,$ $l\geq 1$} We use the standard realization of $A_l$ where the positive roots are $\alpha_{ij}=\lambda_i-\lambda_{j},$ $1\leq i<j\leq l+1$ and the simple roots are $\alpha_i=\alpha_{i,i+1},$ $i=1,...,l.$ The Lie algebra $\mathfrak{k}$ is the set $\mathfrak{so}(l+1)$ of skew-symmetric real matrices of order $l+1.$ We fix $(\cdot,\cdot)=-\langle\cdot,\cdot\rangle$, where $\langle\cdot,\cdot\rangle$ is the Killing form of $\mathfrak{so}(l+1).$ Let $\mathfrak{m}_{ij}$ be the subspace $\vspan\{w_{ij}=E_{ij}-E_{ji}\},$ where $E_{ij}$ is the real $(l+1)\times(l+1)$ matrix with value equal to 1 in the $(i,j)-$entry and zero elsewhere. The set $\{w_{ij}:1\leq j<i\leq l+1\}$ is an $(\cdot,\cdot)-$orthogonal basis for $\mathfrak{so}(l+1).$ For every $\Theta\subseteq\Sigma$, there exist positive integers $l_1,...,l_r$ such that $l+1=l_1+...+l_r,$ and if we set $\tilde{l_0}=0,$ $\tilde{l_i}=\tilde{l}_{i-1}+l_i,$ $i=1,...,r,$ then $\Theta$ is written as the union of its connected components as

\begin{equation}\label{7}
\displaystyle\Theta=\bigcup\limits_{l_i>1}\left\{\alpha_{\tilde{l}_{i-1}+1},...,\alpha_{\tilde{l}_i-1}\right\}.
\end{equation}

By writing $\Theta$ in this form, we have that $K_\Theta=S(O(l_1)\times...\times O(l_r)).$ The following proposition was proved by Patro and San Martin in \cite{PSM}.\\

\begin{prop}\label{2.1}\text{(\cite{PSM})}
For any flag manifold $\mathbb{F}_\Theta$ of $A_l,$ $l\neq 3,$ the $K_\Theta-$invariant irreducible subspaces are

\begin{equation}\label{8}
M_{mn}=\displaystyle\bigoplus\limits_{\begin{subarray}{c} 1\leq i\leq l_m\\
1\leq j\leq l_n\end{subarray}}\mathfrak{m}_{\tilde{l}_{m-1}+i,\tilde{l}_{n-1}+j}, \textstyle \ 1\leq n<m\leq r.
\end{equation}

Two such subspaces are not equivalent.\hfill $\qed$
\end{prop}

\begin{cor}\label{2.2}
Let $\mathbb{F}_\Theta$ a flag of $A_l$, $l\neq 3.$ For $\Theta$ as in \eqref{7}, every  invariant metric $A$ is determined by $\frac{r(r-1)}{2}$ positive numbers $\mu_{mn},$ $1\leq n<m\leq r,$ such that 

\begin{equation}\label{9}
A|_{M_{mn}}=\mu_{mn}\text{I}_{M_{mn}},\ 1\leq n<m\leq r.
\end{equation}
\begin{flushright}
$\qed$
\end{flushright}
\end{cor}

Next, we describe the $K_\Theta-$invariant, irreducible subspaces and their equivalences for each flag of $A_3$. For more details, see \cite{PSM}.

$\bullet$ $\Theta=\emptyset,$ $\mathbb{F}_{\emptyset}=SO(4)/S(O(1)\times O(1)\times O(1)\times O(1)).$

In this case, $\mathfrak{k}_\emptyset=\{0\}$ and the tangent space $\mathfrak{m}_{\emptyset}$ has a decomposition

\begin{center}
$\mathfrak{m}_{\emptyset}=\mathfrak{m}_{21}\oplus\mathfrak{m}_{43}\oplus\mathfrak{m}_{31}\oplus\mathfrak{m}_{42}\oplus\mathfrak{m}_{32}\oplus\mathfrak{m}_{41}$
\end{center}

where all the $\mathfrak{m}_{ij},$ $1\leq j<i\leq 4$ are $K_\emptyset-$invariant, irreducible and

\begin{center}
$\begin{array}{ccc}
M_1 & = & \mathfrak{m}_{21}\oplus\mathfrak{m}_{43},\\
M_2 & = & \mathfrak{m}_{31}\oplus\mathfrak{m}_{42},\\
M_3 & = & \mathfrak{m}_{32}\oplus\mathfrak{m}_{41}
\end{array}$
\end{center}

are the corresponding isotypical summands.

$\bullet$ $\Theta=\{\alpha_1\},$ $\mathbb{F}_{\{\alpha_1\}}=SO(4)/S(O(2)\times O(1)\times O(1)).$

In this case, $\mathfrak{k}_{\{\alpha_1\}}=\mathfrak{m}_{21}$ and

\begin{center}
$\mathfrak{m}_{\{\alpha_1\}}=\mathfrak{m}_{43}\oplus(\mathfrak{m}_{31}\oplus\mathfrak{m}_{32})\oplus(\mathfrak{m}_{42}\oplus\mathfrak{m}_{41})$
\end{center}

where $\mathfrak{m}_{43}$, $\mathfrak{m}_{31}\oplus\mathfrak{m}_{32}$ and $\mathfrak{m}_{42}\oplus\mathfrak{m}_{41}$ are $K_{\{\alpha_1\}}-$invariant, irreducible and

\begin{center}
$\begin{array}{ccl}
M_1 & = & \mathfrak{m}_{43},\\ 
M_2 & = & (\mathfrak{m}_{31}\oplus\mathfrak{m}_{32})\oplus(\mathfrak{m}_{42}\oplus\mathfrak{m}_{41})
\end{array}$ 
\end{center}

are the corresponding isotypical summands.

$\bullet$ $\Theta=\{\alpha_2\},$ $\mathbb{F}_{\{\alpha_2\}}=SO(4)/S(O(1)\times O(2)\times O(1)).$

We have $\mathfrak{k}_{\{\alpha_2\}}=\mathfrak{m}_{32}$ and

\begin{center}
$\mathfrak{m}_{\{\alpha_2\}}=\mathfrak{m}_{41}\oplus(\mathfrak{m}_{21}\oplus\mathfrak{m}_{31})\oplus(\mathfrak{m}_{43}\oplus\mathfrak{m}_{42})$
\end{center}

where $\mathfrak{m}_{21}$, $\mathfrak{m}_{31}\oplus\mathfrak{m}_{41}$ and $\mathfrak{m}_{42}\oplus\mathfrak{m}_{32}$ are $K_{\{\alpha_2\}}-$invariant, irreducible and

\begin{center}
$\begin{array}{ccl}
M_1 & = & \mathfrak{m}_{41},\\ 
M_2 & = & (\mathfrak{m}_{21}\oplus\mathfrak{m}_{31})\oplus(\mathfrak{m}_{43}\oplus\mathfrak{m}_{42})
\end{array}$ 
\end{center}

are the corresponding isotypical summands.

$\bullet$ $\Theta=\{\alpha_3\},$ $\mathbb{F}_{\{\alpha_3\}}=SO(4)/S(O(1)\times O(1)\times O(2)).$

In this case, $\mathfrak{k}_{\{\alpha_3\}}=\mathfrak{m}_{43}$ and

\begin{center}
$\mathfrak{m}_{\{\alpha_3\}}=\mathfrak{m}_{21}\oplus(\mathfrak{m}_{31}\oplus\mathfrak{m}_{41})\oplus(\mathfrak{m}_{42}\oplus\mathfrak{m}_{32})$
\end{center}

where $\mathfrak{m}_{41}$, $\mathfrak{m}_{21}\oplus\mathfrak{m}_{31}$ and $\mathfrak{m}_{43}\oplus\mathfrak{m}_{42}$ are $K_{\{\alpha_3\}}-$invariant, irreducible and

\begin{center}
$\begin{array}{ccl}
M_1 & = & \mathfrak{m}_{41},\\ 
M_2 & = & (\mathfrak{m}_{21}\oplus\mathfrak{m}_{31})\oplus(\mathfrak{m}_{43}\oplus\mathfrak{m}_{42})
\end{array}$ 
\end{center}

are the corresponding isotypical summands.

$\bullet$ $\Theta=\{\alpha_1,\alpha_2\}$ or $\{\alpha_2,\alpha_3\},$ $\mathbb{F}_\Theta=SO(4)/S(O(3)\times O(1))$ or $SO(4)/S(O(1)\times O(3)),$ respectively.

For these sets, the adjoint representation of $K_\Theta$ on $\mathfrak{m}_\Theta$ is irreducible.

$\bullet$ $\Theta=\{\alpha_1,\alpha_3\},$ $\mathbb{F}_{\{\alpha_1,\alpha_3\}}=SO(4)/S(O(2)\times O(2)).$

We have $\mathfrak{k}_{\{\alpha_1,\alpha_3\}}=\mathfrak{m}_{21}\oplus\mathfrak{m}_{43}$, and $\mathfrak{m}_{\{\alpha_1,\alpha_3\}}=M_1\oplus M_2,$ where

\begin{center}
$M_1=\displaystyle\left\{\left(\begin{array}{cc}0 &-B\\ B & 0\end{array}\right):\ B \text{ has the form } B=\left(\begin{array}{cc}a &b\\ b & -a\end{array}\right),\ a,b\in\mathbb{R}\right\}$
\end{center}
and 
\begin{center}
$M_2=\displaystyle\left\{\left(\begin{array}{cc}0 &-B^T\\ B & 0\end{array}\right):\ B \text{ has the form } B=\left(\begin{array}{cc}a &-b\\ b & a\end{array}\right),\ a,b\in\mathbb{R}\right\}$
\end{center}

are not equivalent $K_{\{\alpha_1,\alpha_3\}}-$invariant, irreducible subspaces. \hfill $\qed$

For each $\Theta\subseteq\Sigma$, we fix an ordered $(\cdot,\cdot)-$orthogonal basis for $\mathfrak{m}_\Theta$:

\begin{equation}\label{10}
\begin{array}{lcl}
\mathcal{B}_{\emptyset} & = & \{w_{21},w_{43},w_{31},w_{42},w_{32},w_{41}\},\\
\mathcal{B}_{\{\alpha_1\}} & = & \{w_{43},w_{31},w_{32},w_{42},w_{41}\},\\
\mathcal{B}_{\{\alpha_2\}} & = & \{w_{41},w_{21},w_{31},w_{43},w_{42}\},\\
\mathcal{B}_{\{\alpha_3\}} & = & \{w_{21},w_{31},w_{41},w_{42},w_{32}\},\\
\mathcal{B}_{\{\alpha_1,\alpha_2\}} & = & \{w_{41},w_{42},w_{43}\},\\
\mathcal{B}_{\{\alpha_2,\alpha_3\}} & = & \{w_{21},w_{31},w_{41}\},\\
\mathcal{B}_{\{\alpha_1,\alpha_3\}} & = & \{w_{31}-w_{42},w_{41}+w_{32},w_{31}+w_{42},w_{41}-w_{32}\}.\\
\end{array}
\end{equation}

Now, we describe the invariant metrics for the flags of $A_3.$

\begin{prop}\label{2.3} Every invariant metric $A$ on a flag of $A_3$ is written in the above basis in the following form:

\ \ \ \ \ \ \ \ \ \ \ \ \ \ \ \ \ \ \ \ \ $[A]_{\mathcal{B}_{\emptyset}}=\left(\begin{array}{cccccc}
\mu^{(1)}_1 & b_1 & 0 & 0 & 0 & 0\\
b_1 & \mu^{(1)}_2 & 0 & 0 & 0 & 0\\
0 & 0 & \mu^{(2)}_1 & b_2 & 0 & 0\\
0 & 0 & b_2 & \mu^{(2)}_2 & 0 & 0\\
0 & 0 & 0 & 0 & \mu^{(3)}_1 & b_3\\
0 & 0 & 0 & 0 & b_3 & \mu^{(3)}_2\\
\end{array}\right),$

\

\ \ \ \ \ \ \ \ \ \ \ \ \ \ \ \ \ \ \ \ \ $[A]_{\mathcal{B}_{\Theta}}=\left(\begin{array}{ccccc}
\mu^{(1)}_1 & 0 & 0 & 0 & 0 \\
0 & \mu^{(2)}_1 & 0 & b & 0 \\
0 & 0 & \mu^{(2)}_1 & 0 & -b \\
0 & b & 0 & \mu^{(2)}_2 & 0 \\
0 & 0 & -b & 0 & \mu^{(2)}_2 \\
\end{array}\right),$ for $\Theta=\{\alpha_1\},\{\alpha_2\}$ or $\{\alpha_3\}$

\

\ \ \ \ \ \ \ \ \ \ \ \ \ \ \ \ \ \ \ \ \ $[A]_{\mathcal{B}_{\Theta}}=\left(\begin{array}{ccc}
\mu & 0 & 0 \\
0 & \mu & 0 \\
0 & 0 & \mu \\
\end{array}\right),$ for $\Theta=\{\alpha_1,\alpha_2\}$ or $\{\alpha_2,\alpha_3\}$

\

\ \ \ \ \ \ \ \ \ \ \ \ \ \ \ \ \ \ \ \ \ $[A]_{\mathcal{B}_{\{\alpha_1,\alpha_3\}}}=\left(\begin{array}{cccc}
\mu_1 & 0 & 0 & 0 \\
0 & \mu_1 & 0 & 0 \\
0 & 0 & \mu_2 & 0 \\
0 & 0 & 0 & \mu_2 \\
\end{array}\right),$

where the numbers $\mu^{(i)}_j,$ $\mu$ and $\mu_i$ are positive.
\end{prop}

\begin{proof}
For $\Theta=\emptyset,$ $\{\alpha_1\},$ $\{\alpha_2\},$ $\{\alpha_3\},$ $\{\alpha_1,\alpha_2\}$ and $\{\alpha_2,\alpha_3\}$ is obvious from the description above of the isotypical summands. For $\Theta=\{\alpha_1\}$, we have $A$ is written in $\mathcal{B}_{\{\alpha_1\}}$ in the form

\begin{center}$[A]_{\mathcal{B}_{\{\alpha_1\}}}=\left(\begin{array}{ccccc}
\mu^{(1)}_1 & 0 & 0 & 0 & 0 \\
0 & \mu^{(2)}_1 & 0 & b & d \\
0 & 0 & \mu^{(2)}_1 & c & e \\
0 & b & c & \mu^{(2)}_2 & 0 \\
0 & d & e & 0 & \mu^{(2)}_2 \\
\end{array}\right).$
\end{center}

Given $k\in K_{\{\alpha_1\}}=S(O(2)\times O(1)\times O(1)),$ we have that $k$ has the form 

\begin{center}
$k=\left(\begin{array}{cccc}
r & s & 0 & 0 \\
t & u & 0 & 0 \\
0 & 0 & v & 0 \\
0 & 0 & 0 & z \\
\end{array}\right),$
\end{center}

where its columns are orthonormal and $\det(k)=1.$ It is easy to verify that 

\begin{center}$[\Ad(k)]_{\mathcal{B}_{\{\alpha_1\}}}=\left(\begin{array}{ccccc}
vz & 0 & 0 & 0 & 0 \\
0 & vr & vs & 0 & 0 \\
0 & vt & vu & 0 & 0 \\
0 & 0 & 0 & zu & zt \\
0 & 0 & 0 & zs & zr \\
\end{array}\right).$
\end{center}

Since $A$ commutes with $\Ad(k)$ for all $k\in K_{\{\alpha_1\}}$, then for $r=t=u=-s=\frac{1}{\sqrt{2}}$ and $v=z=1$ we have

\ \ \ \ \ \ \ \ \ \ \ \ \ \ \ \ \ \ \ \ \ \ \ \ \ \ \ \ \ \ \ \ \ \ \ \ \ \ \ $[A]_{\mathcal{B}_{\{\alpha_1\}}}[\Ad(k)]_{\mathcal{B}_{\{\alpha_1\}}}-[\Ad(k)]_{\mathcal{B}_{\{\alpha_1\}}}[A]_{\mathcal{B}_{\{\alpha_1\}}}=\textbf{0}$\\

\ \ \ \ \ \ \ \ \ \ \ \ \ \ \ \ \ \ \ \ \ \ \ \ \ \ \ \ \ \ \ \ $\Longleftrightarrow$ $\displaystyle\left(\begin{array}{ccccc}
0 & 0 & 0 & 0 & 0 \\
0 & 0 & 0 & -\frac{c-d}{\sqrt{2}} & -\frac{b+e}{\sqrt{2}} \\
0 & 0 & 0 & \frac{b+e}{\sqrt{2}} & -\frac{c-d}{\sqrt{2}} \\
0 & \frac{c-d}{\sqrt{2}} & \frac{b+e}{\sqrt{2}} & 0 & 0 \\
0 & -\frac{b+e}{\sqrt{2}} & \frac{c-d}{\sqrt{2}} & 0 & 0 \\
\end{array}\right)=\textbf{0}$

thus, $b=-e$ and $c=d.$ Taking $r=v=-u=-z=1$ and $t=s=0$ we have

\ \ \ \ \ \ \ \ \ \ \ \ \ \ \ \ \ \ \ \ \ \ \ \ \ \ \ \ \ \ \ \ \ \ \ \ \ \ \ $[A]_{\mathcal{B}_{\{\alpha_1\}}}[\Ad(k)]_{\mathcal{B}_{\{\alpha_1\}}}-[\Ad(k)]_{\mathcal{B}_{\{\alpha_1\}}}[A]_{\mathcal{B}_{\{\alpha_1\}}}=\textbf{0}$\\

\ \ \ \ \ \ \ \ \ \ \ \ \ \ \ \ \ \ \ \ \ \ \ \ \ \ \ \ \ \ \ \ $\Longleftrightarrow$ $\displaystyle\left(\begin{array}{ccccc}
0 & 0 & 0 & 0 & 0 \\
0 & 0 & 0 & 0 & -2d \\
0 & 0 & 0 & 2c & 0 \\
0 & 0 & -2c & 0 & 0 \\
0 & 2d & 0 & 0 & 0 \\
\end{array}\right)=\textbf{0}$

concluding that $c=d=0,$ as we wanted to prove. Analogously for $\Theta=\{\alpha_2\}$ and $\Theta=\{\alpha_3\}.$
\end{proof}

\subsection{Flags of $B_l,$ $l\geq 5$} The set of roots of the Lie algebra of type $B_l$ is described as follows:
\begin{itemize}
\item The long ones $\pm(\lambda_i-\lambda_j),$ $\pm(\lambda_i+\lambda_j),$ $1\leq i<j\leq l$ and
\item the short ones $\pm\lambda_i,$ $1\leq i\leq l,$
\end{itemize}

where

\begin{center}
$\begin{array}{rccc}
\lambda_i: & \left\{H=\left(\begin{array}{ccc}\textbf{0} & \textbf{0} & \textbf{0}\\ \textbf{0} & \Lambda & \textbf{0}\\ \textbf{0} & \textbf{0} & -\Lambda\end{array}\right):\Lambda=\diag(a_1,...,a_l)\right\} & \longrightarrow & \mathbb{R},
\end{array}$ $\lambda_i(H)=a_i,$ $i=1,...,l.$
\end{center}

The simple roots are $\alpha_i=\lambda_i-\lambda_{i+1},$ $1\leq i\leq l-1$ and $\alpha_l=\lambda_l.$ The subalgebra $\mathfrak{k}$ is the set of $(2l+1)\times(2l+1)$ skew-symmetric matrices

\begin{center}
$\left(\begin{array}{ccc}
0 & -a & -a \\
a^T & A & B \\
a^T & B & A
\end{array}\right)$, $A+A^T=B+B^T=\textbf{0}.$
\end{center}
It is isomorphic to $\mathfrak{so}(l+1)\oplus\mathfrak{so}(l).$ The isomorphism is provided by the decomposition
\begin{center}
$\left(\begin{array}{ccc}
0 & -a & -a \\
a^T & A & B \\
a^T & B & A
\end{array}\right)=\left(\begin{array}{ccc}
0 & -a & -a \\
a^T & \frac{(A+B)}{2} & \frac{(A+B)}{2} \\
a^T & \frac{(A+B)}{2} & \frac{(A+B)}{2}
\end{array}\right)+\left(\begin{array}{ccc}
0 & \textbf{0} & \textbf{0} \\
\textbf{0} & \frac{(A-B)}{2} & -\frac{(A-B)}{2} \\
\textbf{0} & -\frac{(A-B)}{2} & \frac{(A-B)}{2}
\end{array}\right).$
\end{center}

We fix the $\Ad(K)-$invariant inner product $(\cdot,\cdot)$ on $\mathfrak{k}$ defined by 

\begin{equation}\label{B1}
\left(\left(\begin{array}{ccc}
0 & -a & -a \\
a^T & A & B \\
a^T & B & A
\end{array}\right),\left(\begin{array}{ccc}
0 & -c & -c \\
c^T & C & D \\
c^T & D & C
\end{array}\right)\right)=\displaystyle ac^T-\frac{1}{2}(Tr(BD)+Tr(AC)).
\end{equation}

The matrices 
\begin{equation}\label{Special1}
\begin{array}{ll}
v_{k}=E_{1+k,1}-E_{1,1+k}+E_{1+l+k,1}-E_{1,1+l+k}, & 1\leq k \leq l, \\
w_{ij}=E_{1+i,1+j}-E_{1+j,1+i}+E_{1+l+i,1+l+j}-E_{1+l+j,1+l+i}, & \\
u_{ij}=E_{1+l+i,1+j}-E_{1+l+j,1+i}+E_{1+i,1+l+j}-E_{1+j,1+l+i}, & 1\leq j<i\leq l,\\
\end{array}
\end{equation}
where $E_{ij}$ is the matrix with value equal to 1 in the $(i,j)-$entry and zero elsewhere, form  a $(\cdot,\cdot)-$orthonormal basis for $\mathfrak{k}.$

We take $r$ positive integers $l_1,...,l_r$ such that $l=l_1+...+l_r$ and if $\tilde{l}_{0}=0,\ \tilde{l}_i=\tilde{l}_{i-1}+l_i$ then 

\begin{equation}\label{B2}
\Theta=\bigcup\limits_{l_i>1}\{\alpha_{\tilde{l}_{i-1}+1},...,\alpha_{\tilde{l}_i-1}\} \text{ or } \bigcup\limits_{l_i>1}\{\alpha_{\tilde{l}_{i-1}+1},...,\alpha_{\tilde{l}_i-1}\}\cup\{\alpha_l\},
\end{equation}

In this case, it is more difficult to explicitly determine the subgroup $K_\Theta.$ Instead we have that

\begin{center}
$X=\left(\begin{array}{ccc}
0 & \textbf{0} & \textbf{0} \\
\textbf{0} & A & -A \\
\textbf{0} & -A & A
\end{array}\right)\in\mathfrak{so}(l)\subseteq\mathfrak{k}$ $\Longrightarrow$ $exp(X)=\left(\begin{array}{ccc}
1 & \textbf{0} & \textbf{0} \\
\textbf{0} & \frac{I+exp(2A)}{2} & \frac{I-exp(2A)}{2} \\
\textbf{0} & \frac{I-exp(2A)}{2} & \frac{I+exp(2A)}{2}
\end{array}\right)\in K$
\end{center}
and
\begin{center}
$Y=\left(\begin{array}{ccc}
0 & \textbf{0} & \textbf{0} \\
\textbf{0} & A & A \\
\textbf{0} & A & A
\end{array}\right)\in\mathfrak{so}(l+1)\subseteq\mathfrak{k}$ $\Longrightarrow$ $exp(Y)=\left(\begin{array}{ccc}
1 & \textbf{0} & \textbf{0} \\
\textbf{0} & \frac{I+exp(2A)}{2} & -\frac{I-exp(2A)}{2} \\
\textbf{0} & -\frac{I-exp(2A)}{2} & \frac{I+exp(2A)}{2}
\end{array}\right)\in K.$
\end{center}

Since $A$ is skew-symmetric, then $exp(2A)\in SO(l).$ Also, every $P\in SO(l)$ can be written as a product $\prod\limits_{i=1}^{n}exp(A_i)$ with $A_i$ skew-symmetric, so 

\begin{center}
$\left\{\left(\begin{array}{ccc}
1 & \textbf{0} & \textbf{0} \\
\textbf{0} & \frac{I+P}{2} & \pm\frac{I-P}{2} \\
\textbf{0} & \pm\frac{I-P}{2} & \frac{I+P}{2}
\end{array}\right):P\in SO(l)\right\}\subseteq K.$
\end{center}
Therefore
\begin{center}
$k=\left(\begin{array}{ccc}
1 & \textbf{0} & \textbf{0} \\
\textbf{0} & \frac{I+P}{2} & \frac{I-P}{2} \\
\textbf{0} & \frac{I-P}{2} & \frac{I+P}{2}
\end{array}\right)\cdot \left(\begin{array}{ccc}
1 & \textbf{0} & \textbf{0} \\
\textbf{0} & \frac{I+P}{2} & -\frac{I-P}{2} \\
\textbf{0} & -\frac{I-P}{2} & \frac{I+P}{2}
\end{array}\right)=\left(\begin{array}{ccc}
1 & \textbf{0} & \textbf{0} \\
\textbf{0} & \frac{P+Q}{2} & \frac{P-Q}{2} \\
\textbf{0} & \frac{P-Q}{2} & \frac{P+Q}{2}
\end{array}\right)\in K$
\end{center}
for $P,Q\in SO(l).$ If 

\begin{center}
$H_\Theta=\left(\begin{array}{ccc}
0 & \textbf{0} & \textbf{0} \\
\textbf{0} & \Lambda_\Theta & \textbf{0}\\
\textbf{0} & \textbf{0} & -\Lambda_\Theta
\end{array}\right)$
\end{center}

is characteristic for $\Theta$, then
\begin{center}
$kH_\Theta=H_\Theta k$ $\Longleftrightarrow$ $P\Lambda_\Theta=\Lambda_\Theta Q.$
\end{center}

In particular, for $P=Q$ we have

\begin{center}
$S(O(l_1)\times...\times O(l_r))\stackrel{\text{dif.}}{\approx}\left\{\left(\begin{array}{ccc}
1 & \textbf{0} & \textbf{0} \\
\textbf{0} & P & \textbf{0}\\
\textbf{0} & \textbf{0} & P
\end{array}\right):P\in SO(l),\ P\Lambda_\Theta=\Lambda_\Theta P\right\}\subseteq K_\Theta.$
\end{center}

\begin{prop}\label{2.4}(\cite{PSM}) Let $\mathbb{F}_\Theta$ be a flag manifold of $B_l,$ with $l\geq 5.$ Then the following subspaces are $K_\Theta-$invariant and irreducible:

$a)$\begin{center}
$V_i=\vspan\{v_{\tilde{l}_{i-1}+s}:1\leq s\leq l_i\},$ 
\end{center}

with $1\leq i\leq r$ if $\alpha_l\notin\Theta.$ All these subspaces are not equivalent.

$b)$\begin{center}
$W_{mn}=\displaystyle\bigoplus\limits_{\begin{subarray}{c} 1\leq i\leq l_m\\
1\leq j\leq l_n\end{subarray}}\vspan\{w_{\tilde{l}_{m-1}+i,\tilde{l}_{n-1}+j}\}$ and $U_{mn}=\displaystyle\bigoplus\limits_{\begin{subarray}{c} 1\leq i\leq l_m\\
1\leq j\leq l_n\end{subarray}}\vspan\{u_{\tilde{l}_{m-1}+i,\tilde{l}_{n-1}+j}\},$
\end{center}

with $1\leq n<m\leq r$ if $\alpha_l\not\in\Theta$ and $1\leq n<m\leq r-1$ if $\alpha_l\in\Theta.$ For each $(m,n)$, $W_{mn}$ and $U_{mn}$ are equivalent. We denote by $M_{mn}=W_{mn}\oplus U_{mn}.$ 

$c)$\begin{center}
$U_i=\vspan\{u_{\tilde{l}_{i-1}+s,\tilde{l}_{i-1}+t}:1\leq t\leq s\leq l_i\}$
\end{center}
for $i$ such that $l_i>1$ and $1\leq i\leq r$ if $\alpha_l\notin\Theta$ and $1\leq i\leq r-1$ if $\alpha_l\in\Theta.$ All these subspaces are not equivalent.

$d)$\begin{center}
$\begin{array}{ccl}
(V_i)_1 & = & \vspan\{w_{\tilde{l}_{r-1}+s,\tilde{l}_{i-1}+t}-u_{\tilde{l}_{r-1}+s,\tilde{l}_{i-1}+t}:1\leq s \leq l_r, 1\leq t\leq l_i\}\\
\\
(V_i)_2 & = & \vspan\{v_{\tilde{l}_{i-1}+s},\ w_{\tilde{l}_{r-1}+s,\tilde{l}_{r-1}+t}+u_{\tilde{l}_{r-1}+s,\tilde{l}_{r-1}+t}:1\leq s \leq l_r, 1\leq t\leq l_i\}\\
\end{array}$
\end{center}
with $1\leq i\leq r-1$ when $\alpha_l\in\Theta.$ All these subspaces are not equivalent. \hfill $\qed$
\end{prop}

\begin{prop}\label{2.5} Let $\Theta\subseteq\Sigma,$ and $l\geq 5.$

$a)$ If $\alpha_l\notin \Theta$ then 
\begin{equation}\label{B3}
\mathcal{B}_\Theta=\displaystyle\left(\bigcup\limits_{i=1}^r\mathcal{B}_{0,i}\right)\cup\left(\bigcup\limits_{1\leq n<m\leq r}\mathcal{B}_{mn}\right)\cup\left(\bigcup\limits_{l_i>1}\mathcal{B}_{i}\right)
\end{equation}

is a  $(\cdot,\cdot)-$orthonormal basis for $\mathfrak{m}_\Theta$ adapted to the subspaces of Proposition \ref{2.4}. Where
\begin{center}
$\mathcal{B}_{0,i}=\displaystyle\left\{v_{\tilde{l}_{i-1}+s}:1\leq s\leq r\right\},$
\end{center}

\begin{center}
$\mathcal{B}_{mn}=\{w_{\tilde{l}_{m-1}+s,\tilde{l}_{n-1}+t}:s=1,...,l_m,\ t=1,...,l_n\}\cup\{u_{\tilde{l}_{m-1}+s,\tilde{l}_{n-1}+t}:s=1,...,l_m,\ t=1,...,l_n\}$
\end{center} 

\begin{center}
$\mathcal{B}_{i}=\left\{u_{\tilde{l}_{i-1}+s,\tilde{l}_{i-1}+t}:1\leq t<s\leq l_i\right\}.$
\end{center} 
$b)$ If $\alpha_l\in\Theta$ then 
\begin{equation}\label{B4}
\mathcal{B}_\Theta=\displaystyle\left(\bigcup\limits_{i=1}^{r-1}(\mathcal{B}_{i})_1\cup(\mathcal{B}_{i})_2\right)\cup\left(\bigcup\limits_{1\leq n<m\leq r-1}\mathcal{B}_{mn}\right)\cup\left(\bigcup\limits_{\begin{subarray}{c}1\leq i\leq r-1\\l_i>1\end{subarray}}\mathcal{B}_{i}\right)
\end{equation}

is a $(\cdot,\cdot)-$orthogonal basis for $\mathfrak{m}_\Theta$ adapted to the subspaces of Proposition \ref{2.4}. Here $\mathcal{B}_{mn}$ and $\mathcal{B}_i$ are as before and 
\begin{center}
$(\mathcal{B}_{i})_1=\displaystyle\left\{w_{\tilde{l}_{r-1}+s,\tilde{l}_{i-1}+t}-u_{\tilde{l}_{r-1}+s,\tilde{l}_{i-1}+t}:1\leq s\leq l_r,\ 1\leq t\leq l_i\right\},$
\end{center}

\begin{center}
$(\mathcal{B}_{i})_2=\displaystyle\left\{v_{\tilde{l}_{i-1}+t},\ w_{\tilde{l}_{r-1}+s,\tilde{l}_{i-1}+t}+u_{\tilde{l}_{r-1}+s,\tilde{l}_{i-1}+t}:1\leq s\leq l_r,\ 1\leq t\leq l_i\right\}.$ 
\end{center}\hfill $\qed$
\end{prop}
Next proposition describes invariant metrics on flags of $B_l,$ $l\geq 5.$
\begin{prop}\label{2.6}
Every invariant metric $A$ on a flag of $B_l,$ $l\geq 5$ is written in the bases above in the following form:

$a)$ If $\alpha_l\notin\Theta$ then

\ \ \ \ \ \ \ \ \ \ \ \ \ \ \ \ \ \ \ \ \ \ \ \ \ \ \ \ \ \ \ \ \ $A|_{V_i}=\mu^{(i)}\text{I}_{V_i},$ $1\leq i\leq r,$

\ \ \ \ \ \ \ \ \ \ \ \ \ \ \ \ \ \ \ \ \ \ \ \ \ \ \ \ \ \ \ \ \ $[A|_{M_{mn}}]_{\mathcal{B}_{mn}}=\left(\begin{array}{cc}
\lambda^{(mn)}_1\text{I}_{W_{mn}} & b_{mn}\text{I} \\
 & \\
b_{mn}\text{I} & \lambda^{(mn)}_2\text{I}_{U_{mn}} \\
\end{array}\right),$ $1\leq n<m\leq r,$

\

\ \ \ \ \ \ \ \ \ \ \ \ \ \ \ \ \ \ \ \ \ \ \ \ \ \ \ \ \ \ \ \ \ $A|_{U_i}=\gamma^{(i)}\text
{I}_{U_i},$\  $1\leq i\leq r$ and $l_i>1.$

$b)$ If $\alpha_l\in\Theta$ then 

\ \ \ \ \ \ \ \ \ \ \ \ \ \ \ \ \ \ \ \ \ \ \ \ \ \ \ \ $A|_{(V_i)_1}=\rho^{(i)}\text{I}_{(V_i)_1},$ $1\leq i\leq r-1,$

\

\ \ \ \ \ \ \ \ \ \ \ \ \ \ \ \ \ \ \ \ \ \ \ \ \ \ \ \ $A|_{(V_i)_2}=\mu^{(i)}\text{I}_{(V_i)_2},$ $1\leq i\leq r-1,$

\

\ \ \ \ \ \ \ \ \ \ \ \ \ \ \ \ \ \ \ \ \ \ \ \ \ \ \ \ $[A|_{M_{mn}}]_{\mathcal{B}_{mn}}=\left(\begin{array}{cc}
\lambda^{(mn)}_1\text{I}_{W_{mn}} & b_{mn}\text{I} \\
 & \\
b_{mn}\text{I} & \lambda^{(mn)}_2\text{I}_{U_{mn}} \\
\end{array}\right),$ $1\leq n<m\leq r-1,$

\

\ \ \ \ \ \ \ \ \ \ \ \ \ \ \ \ \ \ \ \ \ \ \ \ \ \ \ \ $A|_{U_i}=\gamma^{(i)}\text{I}_{U_i},$\  $1\leq i\leq r-1$ and $l_i>1.$
\end{prop}
\begin{proof}

It is enough to prove the result for $A|_{M_{mn}}.$ By equation \eqref{6} we have

\begin{center}
$[A|_{M_{mn}}]_{\mathcal{B}_{mn}}=\left(\begin{array}{cc}
\lambda^{(mn)}_1\text{I}_{W_{mn}} & B^T \\
 & \\
B & \lambda^{(mn)}_2\text{I}_{U_{mn}} \\
\end{array}\right).$
\end{center}

Let us take
\begin{center}
$k=\left(\begin{array}{ccc}
1 & \textbf{0} & \textbf{0}\\
\textbf{0} & P & \textbf{0} \\
\textbf{0} & \textbf{0} & P \\
\end{array}\right)\in K_\Theta$
\end{center}
 
where $\det(P)=1$ and $P$ is a block diagonal matrix 

\begin{center}
$P=\left(\begin{array}{cccc}
P_1 & \textbf{0} & \dots & \textbf{0} \\
\textbf{0} & P_2 & \dots & \textbf{0} \\
\vdots & \vdots & \ddots & \vdots \\
\textbf{0} & \textbf{0} & \dots & P_r \\
\end{array}\right),$ $P_i\in O(l_i)$ for $i=1,...,r.$
\end{center}

If $P_i=\left(p^{i}_{st}\right)_{l_i\times l_i},$ then

\begin{center}
$\begin{array}{ccl}
\Ad(k)w_{\tilde{l}_{m-1}+s,\tilde{l}_{n-1}+t} & = & \displaystyle \sum\limits_{e=1}^{l_m}\sum\limits_{f=1}^{l_n}p_{es}^mp_{ft}^n\ w_{\tilde{l}_{m-1}+e,\tilde{l}_{n-1}+f}
\end{array}$
\end{center}

and

\begin{center}
$\begin{array}{ccl}
\Ad(k)u_{\tilde{l}_{m-1}+s,\tilde{l}_{n-1}+t} & = & \displaystyle \sum\limits_{e=1}^{l_m}\sum\limits_{f=1}^{l_n}p_{es}^mp_{ft}^n\ u_{\tilde{l}_{m-1}+e,\tilde{l}_{n-1}+f},
\end{array}$
\end{center}

Also, we have that for every $(s,t)$ there exist a set of real numbers $\{b^{st}_{ef}:1\leq e\leq l_m,\ 1\leq f\leq l_n\}$ (each $b_{ef}^{st}$ is an entry of the matrix $B$) such that  

\begin{center}
$\begin{array}{ccl}
Aw_{\tilde{l}_{m-1}+s,\tilde{l}_{n-1}+t} & = & \lambda_1^{(mn)}w_{\tilde{l}_{m-1}+s,\tilde{l}_{n-1}+t}+ \displaystyle \sum\limits_{e=1}^{l_m}\sum\limits_{f=1}^{l_n}b^{st}_{ef}\ u_{\tilde{l}_{m-1}+e,\tilde{l}_{n-1}+f},
\end{array}$
\end{center}

So

\begin{center}
$\begin{array}{ccl}
\Ad(k)\circ Aw_{\tilde{l}_{m-1}+s,\tilde{l}_{n-1}+t} & = & \displaystyle \lambda_1^{(mn)}\sum\limits_{e=1}^{l_m}\sum\limits_{f=1}^{l_n}p^{m}_{es}p^{n}_{ft}\ w_{\tilde{l}_{m-1}+e,\tilde{l}_{n-1}+f}\\
 & & \displaystyle+\sum\limits_{\tilde{e},e=1}^{l_m}\sum\limits_{\tilde{f},f=1}^{l_n}b^{st}_{ef}p^{m}_{\tilde{e}e}p^{n}_{\tilde{f}f}\ u_{\tilde{l}_{m-1}+\tilde{e},\tilde{l}_{n-1}+\tilde{f}}
\end{array}$
\end{center}

and

\begin{center}
$\begin{array}{ccl}
A\circ \Ad(k)w_{\tilde{l}_{m-1}+s,\tilde{l}_{n-1}+t} & = & \displaystyle \lambda_1^{(mn)}\sum\limits_{e=1}^{l_m}\sum\limits_{f=1}^{l_n}p^{m}_{es}p^{n}_{ft}\ w_{\tilde{l}_{m-1}+e,\tilde{l}_{n-1}+f}\\
 & & \displaystyle+\sum\limits_{\tilde{e},e=1}^{l_m}\sum\limits_{\tilde{f},f=1}^{l_n}b^{ef}_{\tilde{e}\tilde{f}}p^{m}_{es}p^{n}_{ft}\ u_{\tilde{l}_{m-1}+\tilde{e},\tilde{l}_{n-1}+\tilde{f}}.
\end{array}$
\end{center}

Since $A$ commutes with $\Ad(k)$ then 
\begin{equation}\label{B5}
\displaystyle\sum\limits_{e=1}^{l_m}\sum\limits_{f=1}^{l_n}b^{st}_{ef}p^{m}_{\tilde{e}e}p^{n}_{\tilde{f}f}=\sum\limits_{e=1}^{l_m}\sum\limits_{f=1}^{l_n}b^{ef}_{\tilde{e}\tilde{f}}p^{m}_{es}p^{n}_{ft},
\end{equation}

for $1\leq s,\tilde{e}\leq l_m,$ and $1\leq t,\tilde{f}\leq l_n.$ Fixing $s,t,\tilde{e},\tilde{f},$ we shall show that $b^{st}_{\tilde{e}\tilde{f}}=0$ if $(s,t)\neq(\tilde{e},\tilde{f}).$ First, we suppose $r>2$. By taking $P_m=\diag(1,..,1,-1,1,...,1)$, with $-1$ in the $(\tilde{e},\tilde{e})-$entry, $P_i=\diag(-1,1,...,1)$ for some $i\notin\{m,n\}$ (which exists because $r>2$) and $P_j=\text{I}$, for $j\notin\{m,i\}$ we have

\begin{center}
$-b^{st}_{\tilde{e}\tilde{f}}=b^{st}_{\tilde{e}\tilde{f}}p^m_{ss},$
\end{center}

since $p^m_{ss}=1$ for $s\neq\tilde{e},$ then $b^{st}_{\tilde{e}\tilde{f}}=0$  if $s\neq\tilde{e}.$ By taking $P_n=\diag(1,...,1,-1,1,...,1),$ with $-1$ in the $(\tilde{f},\tilde{f})-$entry, $P_i=\diag(-1,1,...,1)$ for some $i\notin\{m,n\}$ and $P_j=\text{I}$ for $j\notin\{n,i\}$ we have

\begin{center}
$-b^{st}_{\tilde{e}\tilde{f}}=b^{st}_{\tilde{e}\tilde{f}}p^n_{tt},$
\end{center}

again, $p^{n}_{tt}=1$ for $t\neq\tilde{f}$,then $b^{st}_{\tilde{e}\tilde{f}}=0$  if $t\neq\tilde{f}.$ We conclude that $b^{st}_{\tilde{e}\tilde{f}}=0$ if $(s,t)\neq(\tilde{e},\tilde{f}).$ If $r=2$ then $m=2,$ $n=1$ and we have two possibilities:

$\bullet$ $l_1>2:$ In this case, we take $P_2=\diag(1,...,1,-1,1,...,1)$ and $P_1=\diag(1,...,1,-1,1,...,1),$ where $P_2$ has $-1$ in the $(\tilde{e},\tilde{e})-$entry and $P_1$ has $-1$ in the $(j,j)-$entry for some $j\notin\{\tilde{f},t\}$ and we obtain from \eqref{B5} that 
\begin{center}
$-b^{st}_{\tilde{e}\tilde{f}}=b^{st}_{\tilde{e}\tilde{f}}p^2_{ss}.$
\end{center}
Then $s\neq \tilde{e}\Longrightarrow b^{st}_{\tilde{e}\tilde{f}}=0.$ For $s=\tilde{e},$ we take $P_2=\diag(1,...,1,-1,1,...,1)$ with $-1$ in the $(\tilde{e},\tilde{e})-$entry, $P_1=\diag(1,...,1,-1,1,...,1)$ with $-1$ in the $(\tilde{f},\tilde{f})-$entry and therefore 
\begin{center}
$b^{\tilde{e}t}_{\tilde{e}\tilde{f}}=-b^{\tilde{e}t}_{\tilde{e}\tilde{f}}p^1_{tt},$
\end{center}
so $t\neq\tilde{f}\Longrightarrow b^{\tilde{e}t}_{\tilde{e}\tilde{f}}=0.$

$\bullet$ $l_1\leq 2:$ Since $l\geq 5$ and $l_1+l_2=l,$ we have that $l_2>2$ and we can proceed analogously as before. 

We conclude that $b^{st}_{\tilde{e}\tilde{f}}=0$ of $(s,t)\neq(\tilde{e},\tilde{f}),$ thus equation \eqref{B5} becomes
\begin{equation}\label{B6}
b^{st}_{st}p^m_{\tilde{e}s}p^n_{\tilde{f}t}=b^{\tilde{e}\tilde{f}}_{\tilde{e}\tilde{f}}p^m_{\tilde{e}s}p^n_{\tilde{f}t},
\end{equation}
by taking $P_m\in SO(l_m)$ with non-zero $(\tilde{e},s)-$entry and $P_n\in SO(l_n)$ with non-zero $(\tilde{f},t)-$entry, we have $b^{st}_{st}=b^{\tilde{e}\tilde{f}}_{\tilde{e}\tilde{f}}=:b_{mn}$ for all $s,t,\tilde{e},\tilde{f}.$ Then 

\begin{center}
$\begin{array}{ccl}
Aw_{\tilde{l}_{m-1}+s,\tilde{l}_{n-1}+t} & = & \lambda_1^{(mn)}w_{\tilde{l}_{m-1}+s,\tilde{l}_{n-1}+t}+ b_{mn}\ u_{\tilde{l}_{m-1}+s,\tilde{l}_{n-1}+t}.
\end{array}$
\end{center}Hence, $B=b_{mn}\text{I}.$\end{proof}
\subsection{Flags of $C_l,$ $l\geq 3$}
The set of roots of the Lie algebra of type $C_l$ is described as follows:

\begin{itemize}
\item The long ones $\pm 2\lambda_i,$ $1\leq i\leq l,$ and
\item the short ones $\pm(\lambda_i-\lambda_j)$ and $\pm(\lambda_i+\lambda_j),$ $1\leq i <j\leq l$
\end{itemize}

where

\begin{center}
$\begin{array}{rccc}
\lambda_i: & \left\{H=\left(\begin{array}{cc}\Lambda & \textbf{0}\\ \textbf{0} & -\Lambda\end{array}\right):\Lambda=\diag(a_1,...,a_l)\right\} & \longrightarrow & \mathbb{R},
\end{array}$ $\lambda_i(H)=a_i,$ $i=1,...,l.$
\end{center}

The simple roots are $\alpha_i=\lambda_i-\lambda_{i+1},$ $1\leq i\leq l-1$ and $\alpha_l=2\lambda_l.$ The subalgebra $\mathfrak{k}$ is the set of $2l\times2l$ matrices

\begin{center}
$\left(\begin{array}{cc}
A & -B \\
B & A
\end{array}\right)$, $A+A^T=B-B^T=\textbf{0}$
\end{center}

which is isomorphic to $\mathfrak{u}(l),$ where the isomorphism associates the above matrix to $A+\sqrt{-1}B.$ We fix the $\Ad(K)-$invariant inner product $(\cdot,\cdot)$ on $\mathfrak{k}$ defined by

\begin{equation}\label{11}
\left(\left(\begin{array}{cc}
A & -B \\
B & A
\end{array}\right),\left(\begin{array}{cc}
C & -D \\
D & C
\end{array}\right)\right)=\displaystyle\frac{1}{2}(Tr(BD)-Tr(AC)).
\end{equation}

The $2l\times2l$ matrices 
\begin{center}
$\begin{array}{ll}
u_{kk}=E_{l+k,k}-E_{k,l+k}, & 1\leq k \leq l, \\
w_{ij}=E_{ij}-E_{ji}+E_{l+i,l+j}-E_{l+j,l+i}, & \\
u_{ij}=E_{l+i,j}+E_{l+j,i}-E_{i,l+j}-E_{j,l+i}, & 1\leq j<i\leq l\\
\end{array}$
\end{center}

form a $(\cdot,\cdot)-$orthogonal basis for $\mathfrak{k}.$ In what follows, we describe the $K_\Theta-$invariant, irreducible subspaces . Similar to the case $B_l$, we can find positive integers $l_1,...,l_r$ such that $l=l_1+...+l_r$ and $\Theta$ is written as disjoint union of its connected components as

\begin{equation}\label{12}
\Theta=\bigcup\limits_{l_i>1}\{\alpha_{\tilde{l}_{i-1}+1},...,\alpha_{\tilde{l}_i-1}\} \text{ or } \bigcup\limits_{l_i>1}\{\alpha_{\tilde{l}_{i-1}+1},...,\alpha_{\tilde{l}_i-1}\}\cup\{\alpha_l\}
\end{equation}

where $\tilde{l}_0=0$, $\tilde{l}_{i-1}+l_i,$ $i=1,...,r.$ If $\alpha_l\not\in\Theta$, then $K_\Theta\stackrel{\text{dif.}}{\approx}O(l_1)\times...\times O(l_r),$ otherwise, we have $K_\Theta\stackrel{\text{dif.}}{\approx}O(l_1)\times...\times O(l_{r-1})\times U(l_r).$

\begin{prop}\label{2.7}(\cite{PSM}) Let $\mathbb{F}_\Theta$ be a flag manifold of $C_l,$ with $l\neq4.$ The following subspaces are $K_\Theta-$invariant irreducible:

$a)$ 
\begin{center}$V_i=\vspan\{u_{\tilde{l}_{i-1}+1,\tilde{l}_{i-1}+1}+...+u_{\tilde{l}_i,\tilde{l}_i}\},$\end{center}

with $1\leq i\leq r$ if $\alpha_l\not\in\Theta$ and $1\leq i\leq r-1$ if $\alpha_l\in\Theta.$ All these subspaces are equivalent.

$b)$
\begin{center}
$W_{mn}=\displaystyle\bigoplus\limits_{\begin{subarray}{c} 1\leq i\leq l_m\\
1\leq j\leq l_n\end{subarray}}\vspan\{w_{\tilde{l}_{m-1}+i,\tilde{l}_{n-1}+j}\}$ and $U_{mn}=\displaystyle\bigoplus\limits_{\begin{subarray}{c} 1\leq i\leq l_m\\
1\leq j\leq l_n\end{subarray}}\vspan\{u_{\tilde{l}_{m-1}+i,\tilde{l}_{n-1}+j}\},$
\end{center}

with $1\leq n<m\leq r$ if $\alpha_l\not\in\Theta$ and $1\leq n<m\leq r-1$ if $\alpha_l\in\Theta.$ For each $(m,n)$, $W_{mn}$ and $U_{mn}$ are equivalent.

$c)$
\begin{center}
$M_{rn}=\displaystyle\bigoplus\limits_{\begin{subarray}{c} 1\leq i\leq l_r\\
1\leq j\leq l_n\end{subarray}}\vspan\{w_{\tilde{l}_{r-1}+i,\tilde{l}_{n-1}+j}\}\oplus\bigoplus\limits_{\begin{subarray}{c} 1\leq i\leq l_r\\
1\leq j\leq l_n\end{subarray}}\vspan\{u_{\tilde{l}_{r-1}+i,\tilde{l}_{n-1}+j}\},$
\end{center}

with $1\leq n\leq r-1,$ if $\alpha_l\in\Theta.$ All these subspaces are not equivalent.

$d)$
\begin{center}
$U_i=\vspan\{u_{\tilde{l}_{i-1}+s,\tilde{l}_{i-1}+s}-u_{\tilde{l}_{i-1}+s+1,\tilde{l}_{i-1}+s+1}:1\leq s\leq l_i-1\}\cup\{u_{\tilde{l}_{i-1}+s,\tilde{l}_{i-1}+t}:1\leq t<s\leq l_i\},$
\end{center}

for $i$ such that $l_i>1$ and $1\leq i\leq r$ if $\alpha_l\not\in\Theta,$ $1\leq i\leq r-1$ if $\alpha_l\in\Theta.$ All these subspaces are not equivalent.

And any other pair of subspaces are not equivalent.\hfill $\qed$
\end{prop}

\begin{cor}\label{2.8}
Let $\mathbb{F}_\Theta$ be a flag of $C_l$, $l\neq 4$. Then, the isotypical summands corresponding  to the $K_\Theta-$invariant spaces in Proposition \ref{2.7} are

\begin{equation}\label{13}
\begin{array}{lcl}
M_0 & = & \displaystyle\bigoplus\limits_{i}V_i,\\
\\
M_{mn} & = & W_{mn}\oplus U_{mn},\ 1\leq n<m\leq r, \\
\\
M_j & = & U_j,\ l_j>1.
\end{array}
\end{equation}

Where $i$ extends over $\{1,...,r\}$ if $\alpha_l\notin\Theta$ and over $\{1,...,r-1\}$ if $\alpha_l\in\Theta.$ \hfill $\qed$
\end{cor}

\begin{prop}\label{2.9}
Let $\Theta\subseteq\Sigma,$ with $l\neq 4,$ $\alpha_l\notin\Theta$ and $l_1,...,l_r$ as in \eqref{12}. Then
\begin{equation}\label{14}
\mathcal{B}=\displaystyle\mathcal{B}_0\cup\left(\bigcup\limits_{1\leq n<m\leq r}\mathcal{B}_{mn}\right)\cup\left(\bigcup\limits_{l_i>1}\mathcal{B}_{i}\right)
\end{equation}

is a $(\cdot,\cdot)-$orthogonal basis for $\mathfrak{m}_\Theta$ adapted to the decomposition

\begin{center}
$\mathfrak{m}_\Theta=\displaystyle M_0\oplus\left(\bigoplus\limits_{1\leq n<m\leq r}M_{mn}\right)\oplus\left(\bigoplus\limits_{l_i>1}M_{i}\right)$
\end{center}

Where

\begin{center}
$\mathcal{B}_0=\displaystyle\left\{\frac{1}{\sqrt{l_i}}u_{\tilde{l}_{i-1}+1,\tilde{l}_{i-1}+1}+...+u_{\tilde{l}_i,\tilde{l}_i}:i=1,...,r\right\},$
\end{center}

\begin{center}
$\mathcal{B}_{mn}=\{w_{\tilde{l}_{m-1}+s,\tilde{l}_{n-1}+t}:s=1,...,l_m,\ t=1,...,l_n\}\cup\{u_{\tilde{l}_{m-1}+s,\tilde{l}_{n-1}+t}:s=1,...,l_m,\ t=1,...,l_n\}$
\end{center} 

\begin{center}
$\mathcal{B}_{i}=\left\{\frac{1}{s}\sum\limits_{t=1}^{s}u_{\tilde{l}_{i-1}+t,\tilde{l}_{i-1}+t}-u_{\tilde{l}_{i-1}+s+1,\tilde{l}_{i-1}+s+1}:s=1,...,l_i-1\right\}\cup\left\{u_{\tilde{l}_{i-1}+s,\tilde{l}_{i-1}+t}:1\leq t<s\leq l_i\right\}.$
\end{center}

If $\alpha_l\in\Theta,$ $i$ extends only over $\{1,...,r-1\}.$

\end{prop}

The proof is a lengthy but straightforward calculation and we omit it.

Now, we can obtain a description of invariant metrics on the flags of $C_l,$ $l\neq 4.$

\begin{prop}\label{2.10}
Every invariant metric $A$ on a flag of $C_l,$ $l\neq 4$ is written in the basis \eqref{14} in the following form:

$a)$ If $\alpha_l\notin\Theta$ then

\ \ \ \ \ \ \ \ \ \ \ \ \ \ \ \ \ \ \ \ \ \ \ \ \ \ \ \ \ \ \ \ \ $[A|_{M_0}]_{\mathcal{B}_0}=\left(\begin{array}{ccccc}
\mu^{(0)}_1 & a_{21} & a_{31} & \dots & a_{r1}\\
a_{21} & \mu^{(0)}_2 & a_{32} & \dots & a_{r2} \\
a_{31} & a_{32} & \mu^{(0)}_3 & \dots & a_{r3} \\
\vdots & \vdots & \vdots & \ddots & \vdots \\
a_{r1} & a_{r2} & a_{r3} & \dots & \mu^{(0)}_r \\
\end{array}\right),$

\ \ \ \ \ \ \ \ \ \ \ \ \ \ \ \ \ \ \ \ \ \ \ \ \ \ \ \ \ \ \ \ \ $[A|_{M_{mn}}]_{\mathcal{B}_{mn}}=\left(\begin{array}{cc}
\mu^{(mn)}_1\text{I}_{W_{mn}} & b_{mn}\text{I} \\
 & \\
b_{mn}\text{I} & \mu^{(mn)}_2\text{I}_{U_{mn}} \\
\end{array}\right),$ $1\leq n<m\leq r,$

\

\ \ \ \ \ \ \ \ \ \ \ \ \ \ \ \ \ \ \ \ \ \ \ \ \ \ \ \ \ \ \ \ \ $[A|_{M_i}]_{\mathcal{B}_{i}}=\mu^{(i)}\text
{I}_{U_i},$\  $1\leq i\leq r$ and $l_i>1.$

$b)$ If $\alpha_l\in\Theta$ then 

\ \ \ \ \ \ \ \ \ \ \ \ \ \ \ \ \ \ \ \ \ \ \ \ \ \ \ \ $[A|_{M_0}]_{\mathcal{B}_0}=\left(\begin{array}{ccccc}
\mu^{(0)}_1 & a_{21} & a_{31} & \dots & a_{r-1,1}\\
a_{21} & \mu^{(0)}_2 & a_{32} & \dots & a_{r-1,2} \\
a_{31} & a_{32} & \mu^{(0)}_3 & \dots & a_{r-1,3} \\
\vdots & \vdots & \vdots & \ddots & \vdots \\
a_{r-1,1} & a_{r-1,2} & a_{r-1,3} & \dots & \mu^{(0)}_{r-1} \\
\end{array}\right),$

\

\ \ \ \ \ \ \ \ \ \ \ \ \ \ \ \ \ \ \ \ \ \ \ \ \ \ \ \ $[A|_{M_{mn}}]_{\mathcal{B}_{mn}}=\left(\begin{array}{cc}
\mu^{(mn)}_1\text{I}_{W_{mn}} & b_{mn}\text{I} \\
 & \\
b_{mn}\text{I} & \mu^{(mn)}_2\text{I}_{U_{mn}} \\
\end{array}\right),$ $1\leq n<m\leq r-1,$

\

\ \ \ \ \ \ \ \ \ \ \ \ \ \ \ \ \ \ \ \ \ \ \ \ \ \ \ \ $[A|_{M_{rn}}]_{\mathcal{B}_{rn}}=\mu^{(rn)}\text{I}_{M_{rn}},$ $1\leq n\leq r-1,$

\

\ \ \ \ \ \ \ \ \ \ \ \ \ \ \ \ \ \ \ \ \ \ \ \ \ \ \ \ $[A|_{M_i}]_{\mathcal{B}_{i}}=\mu^{(i)}\text{I}_{U_i},$\  $1\leq i\leq r-1$ and $l_i>1.$
\end{prop}

\begin{proof} \textit{Case 1.} $\alpha_l\notin\Theta.$

Because of equations \eqref{5} and \eqref{6}, it is enough to prove the result for $A|_{M_{mn}}.$ In fact, we know $[A|_{M_{mn}}]_{\mathcal{B}_{mn}}$ has the form 

\begin{center}
$[A|_{M_{mn}}]_{\mathcal{B}_{mn}}=\left(\begin{array}{cc}
\mu^{(mn)}_1\text{I}_{W_{mn}} & B^T \\
 & \\
B & \mu^{(mn)}_2\text{I}_{U_{mn}} \\
\end{array}\right).$
\end{center}

Given $k\in K_\Theta\stackrel{\text{dif.}}{\approx}O(l_1)\times...\times O(l_r),$ we have that $k$ has the form

\begin{center}
$k=\left(\begin{array}{cc}
P & \textbf{0} \\
\textbf{0} & P \\
\end{array}\right),$
\end{center}
 
where $P$ is a block diagonal matrix 

\begin{center}
$P=\left(\begin{array}{cccc}
P_1 & \textbf{0} & \dots & \textbf{0} \\
\textbf{0} & P_2 & \dots & \textbf{0} \\
\vdots & \vdots & \ddots & \vdots \\
\textbf{0} & \textbf{0} & \dots & P_r \\
\end{array}\right),$ $P_i\in O(l_i)$ for $i=1,...,r.$
\end{center}

Writing $P_i=\left(p^{i}_{st}\right)_{l_i\times l_i},$ we can easily verify that for any pair $(s,t)$ with $1\leq s\leq l_m,$ $1\leq t\leq l_n$
\begin{equation}\label{15}
\begin{array}{ccl}
\Ad(k)w_{\tilde{l}_{m-1}+s,\tilde{l}_{n-1}+t} & = & \displaystyle \sum\limits_{e=1}^{l_m}\sum\limits_{f=1}^{l_n}p_{es}^mp_{ft}^n\ w_{\tilde{l}_{m-1}+e,\tilde{l}_{n-1}+f}
\end{array}
\end{equation}
and
\begin{equation}\label{16}
\begin{array}{ccl}
\Ad(k)u_{\tilde{l}_{m-1}+s,\tilde{l}_{n-1}+t} & = & \displaystyle \sum\limits_{e=1}^{l_m}\sum\limits_{f=1}^{l_n}p_{es}^mp_{ft}^n\ u_{\tilde{l}_{m-1}+e,\tilde{l}_{n-1}+f}.
\end{array}
\end{equation}

Because of the form of $A|_{M_{mn}},$ we also have that for every $(s,t)$ there exist a set of real numbers $\{b^{st}_{ef}:1\leq e\leq l_m,\ 1\leq f\leq l_n\}$ such that  
\begin{equation}\label{17}
\begin{array}{ccl}
Aw_{\tilde{l}_{m-1}+s,\tilde{l}_{n-1}+t} & = & \mu_1^{(mn)}w_{\tilde{l}_{m-1}+s,\tilde{l}_{n-1}+t}+ \displaystyle \sum\limits_{e=1}^{l_m}\sum\limits_{f=1}^{l_n}b^{st}_{ef}\ u_{\tilde{l}_{m-1}+e,\tilde{l}_{n-1}+f}.
\end{array}
\end{equation}

From \eqref{15}, \eqref{16} and \eqref{17} we get
\begin{center}
$\begin{array}{ccl}
\Ad(k)\circ Aw_{\tilde{l}_{m-1}+s,\tilde{l}_{n-1}+t} & = & \displaystyle \mu_1^{(mn)}\sum\limits_{e=1}^{l_m}\sum\limits_{f=1}^{l_n}p^{m}_{es}p^{n}_{ft}\ w_{\tilde{l}_{m-1}+e,\tilde{l}_{n-1}+f}\\
 & & \displaystyle+\sum\limits_{\tilde{e},e=1}^{l_m}\sum\limits_{\tilde{f},f=1}^{l_n}b^{st}_{ef}p^{m}_{\tilde{e}e}p^{n}_{\tilde{f}f}\ u_{\tilde{l}_{m-1}+\tilde{e},\tilde{l}_{n-1}+\tilde{f}}
\end{array}$
\end{center}
and
\begin{center}
$\begin{array}{ccl}
A\circ \Ad(k)w_{\tilde{l}_{m-1}+s,\tilde{l}_{n-1}+t} & = & \displaystyle \mu_1^{(mn)}\sum\limits_{e=1}^{l_m}\sum\limits_{f=1}^{l_n}p^{m}_{es}p^{n}_{ft}\ w_{\tilde{l}_{m-1}+e,\tilde{l}_{n-1}+f}\\
 & & \displaystyle+\sum\limits_{\tilde{e},e=1}^{l_m}\sum\limits_{\tilde{f},f=1}^{l_n}b^{ef}_{\tilde{e}\tilde{f}}p^{m}_{es}p^{n}_{ft}\ u_{\tilde{l}_{m-1}+\tilde{e},\tilde{l}_{n-1}+\tilde{f}}.
\end{array}$
\end{center}

Since $A\circ \Ad(k)=\Ad(k)\circ A,$ then 
\begin{equation}\label{18}
\displaystyle\sum\limits_{e=1}^{l_m}\sum\limits_{f=1}^{l_n}b^{st}_{ef}p^{m}_{\tilde{e}e}p^{n}_{\tilde{f}f}=\sum\limits_{e=1}^{l_m}\sum\limits_{f=1}^{l_n}b^{ef}_{\tilde{e}\tilde{f}}p^{m}_{es}p^{n}_{ft},
\end{equation}

for $1\leq s,\tilde{e}\leq l_m,$ and $1\leq t,\tilde{f}\leq l_n.$ Fixing $s,t,\tilde{e},\tilde{f},$ we shall show that $b^{st}_{\tilde{e}\tilde{f}}=0$ if $(s,t)\neq(\tilde{e},\tilde{f}).$ The equation \eqref{18} is true for every $k\in K_\Theta$ in particular, if $P_m=\diag(1,..,1,-1,1,...,1)$, with $-1$ in the $(\tilde{e},\tilde{e})-$entry and $P_i=\text{I}$, $i\neq m$ we have

\begin{center}
$-b^{st}_{\tilde{e}\tilde{f}}=b^{st}_{\tilde{e}\tilde{f}}p^m_{ss},$
\end{center}

since $p^m_{ss}=1$ for $s\neq\tilde{e},$ then $b^{st}_{\tilde{e}\tilde{f}}=0$  if $s\neq\tilde{e}.$ By taking $P_n=\diag(1,...,1,-1,1,...,1),$ with $-1$ in the $(\tilde{f},\tilde{f})-$entry and $P_i=\text{I},$ $i\neq n$ we have

\begin{center}
$-b^{st}_{\tilde{e}\tilde{f}}=b^{st}_{\tilde{e}\tilde{f}}p^n_{tt},$
\end{center}

again, $p^{n}_{tt}=1$ for $t\neq\tilde{f}$,then $b^{st}_{\tilde{e}\tilde{f}}=0$  if $t\neq\tilde{f}.$ We conclude that $b^{st}_{\tilde{e}\tilde{f}}=0$ of $(s,t)\neq(\tilde{e},\tilde{f}),$ whereupon equation \eqref{18} becomes
\begin{equation}\label{19}
b^{st}_{st}p^m_{\tilde{e}s}p^n_{\tilde{f}t}=b^{\tilde{e}\tilde{f}}_{\tilde{e}\tilde{f}}p^m_{\tilde{e}s}p^n_{\tilde{f}t},
\end{equation}

by taking $P_m\in O(l_m)$ with non-zero $(\tilde{e},s)-$entry and $P_n\in O(l_n)$ with non-zero $(\tilde{f},t)-$entry, we have $b^{st}_{st}=b^{\tilde{e}\tilde{f}}_{\tilde{e}\tilde{f}}=:b_{mn}$ for all $s,t,\tilde{e},\tilde{f}.$ Equation \eqref{17} implies
\begin{equation}\label{20}
\begin{array}{ccl}
Aw_{\tilde{l}_{m-1}+s,\tilde{l}_{n-1}+t} & = & \mu_1^{(mn)}w_{\tilde{l}_{m-1}+s,\tilde{l}_{n-1}+t}+ b_{mn}\ u_{\tilde{l}_{m-1}+s,\tilde{l}_{n-1}+t}.
\end{array}
\end{equation}

as we wanted.

The same argument works for the case $\alpha_l\in\Theta$ taking into account that 

\begin{center}
$O(l_1)\times...\times O(l_r)\subseteq O(l_1)\times...\times O(l_{r-1})\times U(l_r)=K_\Theta$
\end{center}
\end{proof}

When $l = 4$, in addition to the subspaces described in Proposition \ref{2.7}, we have more equivalent subspaces for some subsets $\Theta.$ The table below shows the equivalence classes for the flags $\mathbb{F}_\Theta$ of $C_4$ where there exist equivalences different to those presented in Proposition \ref{2.7}.

\begin{center}
$\begin{array}{|c|c|c|c|c|c|}
\hline
\Theta & l_1 & l_2 & l_3 & l_4 & \text{Equivalence classes} \\ \hline
\emptyset & 1 & 1 & 1 & 1 & \{V_1,V_2,V_3,V_4\},$ $\{W_{21},W_{43},U_{21},U_{43}\}, \{W_{31},W_{42},U_{31},U_{42}\},$ $\{W_{32},W_{41},U_{32},U_{41}\}\\ \hline
\{\alpha_1\} & 2 & 1 & 1 &  & \{V_1,V_2,V_3\},\{W_{21},W_{31},U_{21},U_{31}\},\{W_{32},U_{32}\},\{U_1\} \\ \hline 
\{\alpha_2\} & 1 & 2 & 1 &  & \{V_1,V_2,V_3\},\{W_{21},W_{32},U_{21},U_{32}\},\{W_{31},U_{31}\},\{U_2\} \\ \hline
\{\alpha_3\} & 1 & 1 & 2 &  & \{V_1,V_2,V_3\},\{W_{31},W_{32},U_{31},U_{32}\},\{W_{21},U_{21}\},\{U_3\} \\ \hline
\end{array}$

\small Table 1.
\end{center}

Therefore, we have the isotypical summands and adapted, ordered, $(\cdot,\cdot)-$orthogonal bases for every $\Theta$ in Table 1. Let us analyse case by case: 

$\bullet$ $\Theta=\emptyset,$ $\mathbb{F}_{\emptyset}=U(4)/(O(1)\times O(1)\times O(1)\times O(1)).$ 

The isotypical summands are

\begin{center}
$\begin{array}{lcl}
M_0 & = & V_1\oplus V_2\oplus V_3\oplus V_4\\
N_1 & = & W_{21}\oplus W_{43}\oplus U_{21}\oplus U_{43}\\
N_2 & = & W_{31}\oplus W_{42}\oplus U_{31}\oplus U_{42}\\
N_3 & = & W_{32}\oplus W_{41}\oplus U_{32}\oplus U_{41}\\
\end{array}$
\end{center}

and $\mathcal{B}_\emptyset=\mathcal{B}_0\cup\mathcal{B}_1\cup\mathcal{B}_2\cup\mathcal{B}_3$ is an adapted basis, where

\begin{center}
$\begin{array}{lcl}
\mathcal{B}_0 & = & \{u_{11},u_{22},u_{33},u_{44}\}\\
\mathcal{B}_1 & = & \{w_{21},w_{43},u_{21},u_{43}\}\\
\mathcal{B}_2 & = & \{w_{31},w_{42},u_{31},u_{42}\}\\
\mathcal{B}_3 & = & \{w_{32},w_{41},u_{32},u_{41}\}\\
\end{array}$
\end{center}

$\bullet$ $\Theta=\{\alpha_1\},$ $\mathbb{F}_{\{\alpha_1\}}=U(4)/(O(2)\times O(1)\times O(1)).$ 

The isotypical summands are

\begin{center}
$\begin{array}{lcl}
M_0 & = & V_1\oplus V_2\oplus V_3\\
M_1 & = & U_1\\
M & = & W_{32}\oplus U_{32}\\
N & = & W_{21}\oplus W_{31}\oplus U_{21}\oplus U_{31}\\
\end{array}$
\end{center}

and $\mathcal{B}_{\{\alpha_1\}}=\mathcal{B}_0\cup\mathcal{B}_1\cup\mathcal{B}_M\cup\mathcal{B}_N$ is an adapted basis, where

\begin{center}
$\begin{array}{lcl}
\mathcal{B}_0 & = & \left\{\frac{1}{\sqrt{2}}(u_{11}+u_{22}),u_{33},u_{44}\right\}\\
\mathcal{B}_1 & = & \{u_{21}\}\\
\mathcal{B}_M & = & \{w_{43},u_{43}\}\\
\mathcal{B}_N & = & \{w_{31},w_{32},w_{41},w_{42},u_{31},u_{32},u_{41},u_{42}\}\\
\end{array}$
\end{center}

$\bullet$ $\Theta=\{\alpha_2\},$ $\mathbb{F}_{\{\alpha_2\}}=U(4)/(O(1)\times O(2)\times O(1)).$ 

The isotypical summands are

\begin{center}
$\begin{array}{lcl}
M_0 & = & V_1\oplus V_2\oplus V_3\\
M_2 & = & U_2\\
M & = & W_{31}\oplus U_{31}\\
N & = & W_{21}\oplus W_{32}\oplus U_{21}\oplus U_{32}\\
\end{array}$
\end{center}

and $\mathcal{B}_{\{\alpha_2\}}=\mathcal{B}_0\cup\mathcal{B}_2\cup\mathcal{B}_M\cup\mathcal{B}_N$ is an adapted basis, where

\begin{center}
$\begin{array}{lcl}
\mathcal{B}_0 & = & \left\{u_{11},\frac{1}{\sqrt{2}}(u_{22}+u_{33}),u_{44}\right\}\\
\mathcal{B}_2 & = & \{u_{32}\}\\
\mathcal{B}_M & = & \{w_{41},u_{41}\}\\
\mathcal{B}_N & = & \{w_{21},w_{31},w_{42},w_{43},u_{21},u_{31},u_{42},u_{43}\}\\
\end{array}$
\end{center}

$\bullet$ $\Theta=\{\alpha_3\},$ $\mathbb{F}_{\{\alpha_3\}}=U(4)/(O(1)\times O(1)\times O(2)).$ 

The isotypical summands are

\begin{center}
$\begin{array}{lcl}
M_0 & = & V_1\oplus V_2\oplus V_3\\
M_3 & = & U_3\\
M & = & W_{21}\oplus U_{21}\\
N & = & W_{31}\oplus W_{32}\oplus U_{31}\oplus U_{32}\\
\end{array}$
\end{center}

and $\mathcal{B}_{\{\alpha_3\}}=\mathcal{B}_0\cup\mathcal{B}_3\cup\mathcal{B}_M\cup\mathcal{B}_N$ is an adapted basis, where

\begin{center}
$\begin{array}{lcl}
\mathcal{B}_0 & = & \left\{u_{11},u_{22},\frac{1}{\sqrt{2}}(u_{33}+u_{44})\right\}\\
\mathcal{B}_3 & = & \{u_{43}\}\\
\mathcal{B}_M & = & \{w_{21},u_{21}\}\\
\mathcal{B}_N & = & \{w_{31},w_{41},w_{32},w_{42},u_{31},u_{41},u_{32},u_{42}\}.\\
\end{array}$
\end{center}

\begin{prop}\label{2.11} Let $A$ be an invariant metric  on a flag $\mathbb{F}_\Theta$ of $C_4.$

$a)$ If $\Theta=\emptyset$, then $A$ is written in the basis $\mathcal{B}_\emptyset$ as

\ \ \ \ \ \ \ \ \ \ \ \ \ \ \ \ \ \ \ \ \ \ \ \ \ \ \ \ \ \ \ \ \ \ \ \ \ $[A|_{M_0}]_{\mathcal{B}_0}=\left(\begin{array}{cccc}
\mu^{(0)}_1 & a_{21} & a_{31} & a_{41}\\
a_{21} & \mu^{(0)}_2 & a_{32} & a_{42}\\
a_{31} & a_{32} & \mu^{(0)}_3 & a_{43}\\
a_{41} & a_{42} & a_{43} & \mu^{(0)}_4\\
\end{array}\right),$

\

\ \ \ \ \ \ \ \ \ \ \ \ \ \ \ \ \ \ \ \ \ \ \ \ \ \ \ \ \ \ \ \ \ \ \ \ \ $[A|_{N_i}]_{\mathcal{B}_i}=\left(\begin{array}{cccc}
\mu^{(i)}_1 & 0 & b^{(i)}_{1} & 0\\
0 & \mu^{(i)}_2 & 0 & b^{(i)}_{2} \\
b^{(i)}_{1} & 0 & \mu^{(i)}_3 & 0 \\
0 & b^{(i)}_{2} & 0 & \mu^{(i)}_4 \\
\end{array}\right),$ \ $i=1,2,3.$

\

$b)$ If $\Theta=\{\alpha_1\},\{\alpha_2\},$ or $\{\alpha_3\},$ then $A$ is written in the basis $\mathcal{B}_\Theta$ as

\ \ \ \ \ \ \ \ \ \ \ \ \ \ \ \ \ \ \ \ \ \ \ \ \ \ \ \ \ \ \ \ \ \ \ \ \ $[A|_{M_0}]_{\mathcal{B}_0}=\left(\begin{array}{ccc}
\mu^{(0)}_1 & a_{21} & a_{31} \\
a_{21} & \mu^{(0)}_2 & a_{32} \\
a_{31} & a_{32} & \mu^{(0)}_3 \\
\end{array}\right),$

\

\ \ \ \ \ \ \ \ \ \ \ \ \ \ \ \ \ \ \ \ \ \ \ \ \ \ \ \ \ \ \ \ \ \ \ \ \ $A|_{M_i}=\mu^{(i)}\text{I}_{U_i},$ $i=1,2,3.$ (if $\Theta=\{\alpha_i\}$)

\

\ \ \ \ \ \ \ \ \ \ \ \ \ \ \ \ \ \ \ \ \ \ \ \ \ \ \ \ \ \ \ \ \ \ \ \ \ $[A|_{M}]_{\mathcal{B}_M}=\left(\begin{array}{cc}
\mu^M_1 & b_M\\
b_M & \mu^M_2\\
\end{array}\right),$

\

\ \ \ \ \ \ \ \ \ \ \ \ \ \ \ \ \ \ \ \ \ \ \ \ \ \ \ \ \ \ \ \ \ \ \ \ \ $[A|_{N}]_{\mathcal{B}_N}=\left(\begin{array}{cccccccc}
\mu^N_1 & 0 & 0 & 0 & b_1^N & 0 & 0 & 0\\
0 & \mu^N_1 & 0 & 0 & 0 & b_1^N & 0 & 0\\
0 & 0 & \mu^N_2 & 0 & 0 & 0 & b_2^N & 0\\
0 & 0 & 0 & \mu^N_2 & 0 & 0 & 0 & b_2^N\\
b_1^N & 0 & 0 & 0 & \mu^N_3 & 0 & 0 & 0\\
0 & b_1^N & 0 & 0 & 0 & \mu^N_3 & 0 & 0\\
0 & 0 & b_2^N & 0 & 0 & 0 & \mu^N_4 & 0\\
0 & 0 & 0 & b_2^N & 0 & 0 & 0 & \mu^N_4\\
\end{array}\right).$
\end{prop}

\begin{proof}
$a)$ For each $i\in\{1,2,3\}$ we write $A|_{N_i}$ in the basis $\mathcal{B}_i$ in the form

\ \ \ \ \ \ \ \ \ \ \ \ \ \ \ \ \ \ \ \ \ \ \ \ \ \ \ \ \ \ \ \ \ \ \ \ \ $[A|_{N_i}]_{\mathcal{B}_i}=\left(\begin{array}{cccc}
\mu^{(i)}_1 & b^{(i)}_{21} & b^{(i)}_{31} & b^{(i)}_{41}\\
b^{(i)}_{21} & \mu^{(i)}_2 & b^{(i)}_{32} & b^{(i)}_{42} \\
b^{(i)}_{31} & b^{(i)}_{32} & \mu^{(i)}_3 & b^{(i)}_{43} \\
b^{(i)}_{41} & b^{(i)}_{42} & b^{(i)}_{43} & \mu^{(i)}_4 \\
\end{array}\right),$

Given $k\in K_\emptyset\stackrel{\text{dif.}}{\approx}O(1)\times O(1)\times O(1)\times O(1),$ we have that $k$ has the form

\begin{center}
$k=\left(\begin{array}{cc}
P & \textbf{0} \\
\textbf{0} & P \\
\end{array}\right),$
\end{center}
 
where $P=\diag(p_1,p_2,p_3,p_4),$ $p_i=\pm 1,$ $i=1,2,3,4.$ It is easy to see that 

\ \ \ \ \ \ \ \ \ \ \ \ \ \ \ \ \ \ \ \ \ \ \ \ \ \ \ \ \ \ \ \ \ $[\Ad(k)|_{N_i}]_{\mathcal{B}_i}=\left(\begin{array}{cccc}
p_{i_1}p_{i_2} & 0 & 0 & 0\\
0 & p_{i_3}p_{i_4} & 0 & 0\\
0 & 0 & p_{i_1}p_{i_2} & 0\\
0 & 0 & 0 & p_{i_3}p_{i_4}\\
\end{array}\right),$

where $\{i_1,i_2,i_3,i_4\}=\{1,2,3,4\}.$ Since $A$ commutes with $\Ad(k)$ for all $k\in K_\emptyset$, then 

\ \ \ \ \ \ \ \ \ \ \ \ \ \ \ \ \ \ \ \ \ \ \ \ \ \ \ \ \ \ \ \ \ $[A|_{N_i}]_{\mathcal{B}_i}[\Ad(k)|_{N_i}]_{\mathcal{B}_i}-[\Ad(k)|_{N_i}]_{\mathcal{B}_i}[A|_{N_i}]_{\mathcal{B}_i}=\textbf{0}$\\

\ \ \ \ \ \ \ \ \ \ \ \ \ \ \ \ \ \ \ \ \ \ \ \ \ \ $\Longleftrightarrow$ $(p_{i_3}p_{i_4}-p_{i_1}p_{i_2})\displaystyle\left(\begin{array}{cccc}
0 & -b^{(i)}_{21} & 0 & -b^{(i)}_{41} \\
b^{(i)}_{21} & 0 & b^{(i)}_{32} & 0 \\
0 & -b^{(i)}_{32} & 0 & -b^{(i)}_{43} \\
b^{(i)}_{41} & 0 & b^{(i)}_{43} & 0 \\
\end{array}\right)=\textbf{0}.$

Taking $-p_{i_1}=1=p_{i_2}=p_{i_3}=p_{i_4},$ we can conclude that $b^{(i)}_{ab}=0$ if $(a,b)\notin\{(3,1),(4,2)\}.$ Defining $b^{(i)}_{1}:=b^{(i)}_{31}$ and $b^{(i)}_2:=b^{(i)}_{42}$, we have the result.

$b)$ Let us consider $\Theta=\{\alpha_1\}.$ Again, it is enough to show the result for $A|_N.$ We know that $A|_N$ is written in the basis $\mathcal{B}_N$ in the form 

\ \ \ \ \ \ \ \ \ \ \ \ \ \ \ \ \ \ \ \ \ \ \ \ \ \ \ \ \ \ \ \ \ $[A|_{N}]_{\mathcal{B}_N}=\left(\begin{array}{cccccccc}
\mu^N_1 & 0 & b^N_{31} & b^N_{41} & b^N_{51} & b^N_{61} & b^N_{71} & b^N_{81}\\
0 & \mu^N_1 & b^N_{32} & b^N_{42} & b^N_{52} & b^N_{62} & b^N_{72} & b^N_{82}\\
b^N_{31} & b^N_{32} & \mu^N_2 & 0 & b^N_{53} & b^N_{63} & b^N_{73} & b^N_{83}\\
b^N_{41} & b^N_{42} & 0 & \mu^N_2 & b^N_{54} & b^N_{64} & b^N_{74} & b^N_{84}\\
b^N_{51} & b^N_{52} & b^N_{53} & b^N_{54} & \mu^N_3 & 0 & b^N_{75} & b^N_{85}\\
b^N_{61} & b^N_{62} & b^N_{63} & b^N_{64} & 0 & \mu^N_3 & b^N_{76} & b^N_{86}\\
b^N_{71} & b^N_{72} & b^N_{73} & b^N_{74} & b^N_{75} & b^N_{76} & \mu^N_4 & 0\\
b^N_{81} & b^N_{82} & b^N_{83} & b^N_{84} & b^N_{85} & b^N_{86} & 0 & \mu^N_4\\
\end{array}\right).$

Given $k\in K_{\{\alpha_1\}}\stackrel{\text{dif.}}{\approx}O(2)\times O(1)\times O(1)$, $k$ has the form 

\begin{center}
$k=\left(\begin{array}{cc}
P & \textbf{0} \\
\textbf{0} & P 
\end{array}\right),$ with $P=\left(\begin{array}{cccc}
r & s & 0 & 0 \\
t & u & 0 & 0 \\
0 & 0 & v & 0 \\
0 & 0 & 0 & z \\
\end{array}\right),$
\end{center}

where the columns of $P$ are orthonormal. Then 

\ \ \ \ \ \ \ \ \ \ \ \ \ \ \ \ \ \ \ \ \ \ \ \ \ \ \ \ \ \ \ \ \ $[\Ad(k)|_{N}]_{\mathcal{B}_N}=\left(\begin{array}{cccccccc}
vr & vs & 0 & 0 & 0 & 0 & 0 & 0\\
vt & vu & 0 & 0 & 0 & 0 & 0 & 0\\
0 & 0 & zr & zs & 0 & 0 & 0 & 0\\
0 & 0 & zt & zu & 0 & 0 & 0 & 0\\
0 & 0 & 0 & 0 & vr & vs & 0 & 0\\
0 & 0 & 0 & 0 & vt & vu & 0 & 0\\
0 & 0 & 0 & 0 & 0 & 0 & zr & zs\\
0 & 0 & 0 & 0 & 0 & 0 & zt & zu\\
\end{array}\right).$

For $-r=1=u=v=z$ and $s=t=0$

\ \ \ \ \ \ \ \ \ \ \ \ \ \ \ \ \ \ \ \ \ \ \ \ \ \ \ \ \ \ \ \ $[A|_N]_{\mathcal{B}_N}[\Ad(k)|_N]_{\mathcal{B}_N}-[\Ad(k)|_N]_{\mathcal{B}_N}[A|_N]_{\mathcal{B}_N}=\textbf{0}$\\

\ \ \ \ \ \ \ \ \ \ \ \ \ \ \ \ \ \ \ \ \ \ \ \ \ $\Longleftrightarrow$ $\displaystyle\left(\begin{array}{cccccccc}
0 & 0 & 0 & 2b^N_{41} & 0 & 2b^N_{61} & 0 & 2b^N_{81}\\
0 & 0 & -2b^N_{32} & 0 & -2b^N_{52} & 0 & -2b^N_{72} & 0\\
0 & 2b^N_{32} & 0 & 0 & 0 & 2b^N_{63} & 0 & 2b^N_{83}\\
-2b^N_{41} & 0 & 0 & 0 & -2b^N_{54} & 0 & -2b^N_{74} & 0\\
0 & 2b^N_{52} & 0 & 2b^N_{54} & 0 & 0 & 0 & 2b^N_{85}\\
-2b^N_{61} & 0 & -2b^N_{63} & 0 & 0 & 0 & -2b^N_{76} & 0\\
0 & 2b^N_{72} & 0 & 2b^N_{74} & 0 & 2b^N_{76} & 0 & 0\\
-2b^N_{81} & 0 & -2b^N_{83} & 0 & -2b^N_{85} & 0 & 0 & 0\\
\end{array}\right)=\textbf{0},$

therefore $b^N_{32}=b^N_{41}=b^N_{52}=b^N_{54}=b^N_{61}=b^N_{63}=b^N_{72}=b^N_{74}=b^N_{76}=b^N_{81}=b^N_{83}=b^N_{85}=0.$ By taking $r=u=z=1=-v$ and $s=t=0,$ we have

\ \ \ \ \ \ \ \ \ \ \ \ \ \ \ \ \ \ \ \ \ \ \ \ \ \ \ \ \ \ \ \ $[A|_N]_{\mathcal{B}_N}[\Ad(k)|_N]_{\mathcal{B}_N}-[\Ad(k)|_N]_{\mathcal{B}_N}[A|_N]_{\mathcal{B}_N}=\textbf{0}$\\

\ \ \ \ \ \ \ \ \ \ \ \ \ \ \ \ \ \ \ \ \ \ \ \ \ $\Longleftrightarrow$ $\displaystyle\left(\begin{array}{cccccccc}
0 & 0 & 2b^N_{31} & 0 & 0 & 0 & 2b^N_{71} & 0\\
0 & 0 & 0 & 2b^N_{42} & 0 & 0 & 0 & 2b^N_{82}\\
-2b^N_{31} & 0 & 0 & 0 & -2b^N_{53} & 0 & 0 & 0\\
0 & -2b^N_{42} & 0 & 0 & 0 & -2b^N_{64} & 0 & 0\\
0 & 0 & 2b^N_{53} & 0 & 0 & 0 & 2b^N_{75} & 0\\
0 & 0 & 0 & 2b^N_{64} & 0 & 0 & 0 & 2b^N_{86}\\
-2b^N_{71} & 0 & 0 & 0 & -2b^N_{75} & 0 & 0 & 0\\
0 & -2b^N_{82} & 0 & 0 & 0 & -2b^N_{86} & 0 & 0\\
\end{array}\right)=\textbf{0},$

which implies $b^N_{31}=b^N_{42}=b^N_{53}=b^N_{64}=b^N_{71}=b^N_{75}=b^N_{82}=b^N_{86}=0.$ Finally, if $r=s=t=-u=\frac{1}{\sqrt{2}}$ and $v=z=1,$ then 

\ \ \ \ \ \ \ \ \ \ \ \ \ \ \ \ \ \ \ \ $[A|_N]_{\mathcal{B}_N}[\Ad(k)|_N]_{\mathcal{B}_N}-[\Ad(k)|_N]_{\mathcal{B}_N}[A|_N]_{\mathcal{B}_N}=\textbf{0}$\\

\ \ \ \ \ \ \ \ \ \ \ \ \ $\Longleftrightarrow$ $\displaystyle\left(\begin{array}{cccccccc}
0 & 0 & 0 & 0 & 0 & \frac{b^N_{51}-b^N_{62}}{\sqrt{2}} & 0 & 0\\
0 & 0 & 0 & 0 & \frac{b^N_{62}-b^N_{51}}{\sqrt{2}} & 0 & 0 & 0\\
0 & 0 & 0 & 0 & 0 & 0 & 0 & \frac{b^N_{73}-b^N_{84}}{\sqrt{2}}\\
0 & 0 & 0 & 0 & 0 & 0 & \frac{b^N_{84}-b^N_{73}}{\sqrt{2}} & 0\\
0 & \frac{b^N_{51}-b^N_{62}}{\sqrt{2}} & 0 & 0 & 0 & 0 & 0 & 0\\
\frac{b^N_{62}-b^N_{51}}{\sqrt{2}} & 0 & 0 & 0 & 0 & 0 & 0 & 0\\
0 & 0 & 0 & \frac{b^N_{73}-b^N_{84}}{\sqrt{2}} & 0 & 0 & 0 & 0\\
0 & 0 & \frac{b^N_{84}-b^N_{73}}{\sqrt{2}} & 0 & 0 & 0 & 0 & 0\\
\end{array}\right)=\textbf{0},$

thus, $b^N_{51}=b^N_{62}=:b^N_{1}$ and $b^N_{73}=b^N_{84}=:b^N_{2}.$ The cases $\Theta=\{\alpha_2\}$ and $\Theta=\{\alpha_3\}$ are analogous.
\end{proof}

\subsection{Flags of $D_l$, $l\geq 5$} The roots are $\pm(\lambda_i-\lambda_j),$ $\pm(\lambda_i+\lambda_j),$ $1\leq i<j\leq l,$ where 
\begin{center}
$\lambda_i:\displaystyle\left\{H=\left(\begin{array}{cc}\Lambda & \textbf{0}\\ \textbf{0} & -\Lambda\end{array}\right): \Lambda=\diag(a_1,.,..,a_l)\right\}\longrightarrow\mathbb{R},$ $\lambda_i(H)=a_i,$ $i=1,...,l.$
\end{center}
The simple roots are $\alpha_i=\lambda_i-\lambda_{i+1},$ $i=1,...,l-1$ and $\alpha_l=\lambda_{l-1}+\lambda_l.$ The subalgebra $\mathfrak{k}$ is the set of skew-symmetric $2l\times2l$ matrices of the form
\begin{center}
$\left(\begin{array}{cc}A & B\\ B & A \end{array}\right)$, $A+A^T=B+B^T=0$
\end{center}
which is isomorphic to $\mathfrak{so}(l)\oplus \mathfrak{so}(l)$ via the decomposition 
\begin{center}
$\left(\begin{array}{cc}A & B\\ B & A \end{array}\right)=\left(\begin{array}{cc}\frac{A+B}{2} & \frac{A+B}{2}\\ \frac{A+B}{2} & \frac{A+B}{2} \end{array}\right)+\left(\begin{array}{cc}\frac{A-B}{2} & -\frac{A-B}{2}\\ -\frac{A-B}{2} & \frac{A-B}{2} \end{array}\right).$
\end{center}
We fix the $\Ad(K)-$invariant inner product $(\cdot,\cdot)$ on $\mathfrak{k}$ given by 
\begin{equation}\label{D1}
\left(\left(\begin{array}{cc}A & B\\ B & A \end{array}\right),\left(\begin{array}{cc}C & D\\ D & C \end{array}\right)\right)=\displaystyle -\frac{Tr(AC)+Tr(BD)}{2}
\end{equation}
The matrices
\begin{equation}\label{special2}
\begin{array}{l}
w_{ij}=E_{ij}-E_{ji}+E_{l+i,l+j}-E_{l+j,l+i},\\
u_{ij}=E_{l+i,j}-E_{l+j,i}+E_{i,l+j}-E_{j,l+i},\ 1\leq j<i\leq l
\end{array}\end{equation}
form a $(\cdot,\cdot)-$orthogonal basis for $\mathfrak{k}.$ Given $\Theta\subseteq \Sigma$, we take $l_1,...,l_r$ as in \eqref{B2}. As in the case of $B_l$, we can show that 

\begin{center}
$K=\left\{\left(\begin{array}{cc}\frac{P+Q}{2} & \frac{P-Q}{2}\\ \frac{P-Q}{2}&\frac{P+Q}{2}\end{array}\right): P,Q\in SO(l)\right\}.$ 
\end{center}  
If 
\begin{center}
$H_\Theta=\left(\begin{array}{cc}
\Lambda_\Theta & \textbf{0}\\
\textbf{0} & \Lambda_\Theta
\end{array}\right)$
\end{center}
is characteristic for $\Theta$, then
\begin{center}
$K_\Theta=\left\{\left(\begin{array}{cc}\frac{P+Q}{2} & \frac{P-Q}{2}\\ \frac{P-Q}{2}&\frac{P+Q}{2}\end{array}\right): P,Q\in SO(l),\ P\Lambda_\Theta=\Lambda_\Theta Q\right\},$
\end{center}
in particular, when $P=Q$, we obtain
\begin{center}
$S(O(l_1)\times...\times O(l_r))\stackrel{\text{dif.}}{\approx}\left\{\left(\begin{array}{cc}P & \textbf{0}\\ \textbf{0} & P\end{array}\right):P=\left(\begin{array}{ccc}P_1 & \cdots & \textbf{0}\\ \vdots & \ddots & \vdots\\\textbf{0} & \cdots & P_r\end{array}\right),\ \det(P)=1,\ P_i\in O(l_i)\right\}\subseteq K_\Theta.$
\end{center}
\begin{prop}\label{2.12}(\cite{PSM}) Let $\mathbb{F}_\Theta$ be a flag manifold of $D_l,$ $l\geq 5.$ Then the following subspaces are $K_\Theta-$invariant and irreducible:

$a)$\begin{center}
$W_{mn}=\vspan\{w_{\tilde{l}_{m-1}+s,\tilde{l}_{n-1}+t}:1\leq s\leq l_m,\ 1\leq t\leq l_n\},$
\

$U_{mn}=\vspan\{u_{\tilde{l}_{m-1}+s,\tilde{l}_{n-1}+t}:1\leq s\leq l_m,\ 1\leq t\leq l_n\},$ 
\end{center}
with $1\leq n<m\leq r$ if $\alpha_l\notin\Theta$, $1\leq n<m\leq r-1$ if $\alpha_l,\alpha_{l-1}\in\Theta$ and $1\leq n<m\leq r-2$ if $\alpha_l\in\Theta$ and $\alpha_{l-1}\notin\Theta.$ For each $(m,n),$ $W_{mn}$ is equivalent to $U_{mn}.$ We set $M_{mn}=W_{mn}\oplus U_{mn}.$

$b)$\begin{center}
$U_i=\vspan\{u_{\tilde{l}_{i-1}+s,\tilde{l}_{i-1}+t}:1\leq t<s\leq l_i\},$
\end{center}
with $l_i>1,$ $1\leq i\leq r$ if $\alpha_l\notin\Theta,$ $1\leq i\leq r-1$ if $\alpha_l,\alpha_{l-1}\in\Theta$ and $1\leq i\leq r-2$ if $\alpha_l\in\Theta$ and $\alpha_{l-1}\notin\Theta.$ All these subspaces are not equivalent.

$c)$\begin{center}
$M_{rn}=\displaystyle\bigoplus\limits_{\begin{subarray}{c} 1\leq s\leq l_r\\
1\leq t\leq l_n\end{subarray}}\vspan\{w_{\tilde{l}_{r-1}+s,\tilde{l}_{n-1}+t}\}\oplus\bigoplus\limits_{\begin{subarray}{c} 1\leq s\leq l_r\\
1\leq t\leq l_n\end{subarray}}\vspan\{u_{\tilde{l}_{r-1}+s,\tilde{l}_{n-1}+t}\},$ 
\end{center}
with $1\leq n\leq r-1$ when $\alpha_l\in\Theta$ and $\alpha_{l-1}\in\Theta.$ All these subspaces are not equivalent.

$d)$\begin{center}
$M_n=\vspan\{w_{\tilde{l}_{r-2}+s,\tilde{l}_{n-1}+t}:1\leq s\leq l_{r-1},\ 1\leq t\leq l_n\}\cup\{u_{l,\tilde{l}_{n-1}+t}:1\leq t\leq l_n\},$

$N_n=\vspan\{u_{\tilde{l}_{r-2}+s,\tilde{l}_{n-1}+t}:1\leq s\leq l_{r-1},\ 1\leq t\leq l_n\}\cup\{w_{l,\tilde{l}_{n-1}+t}:1\leq t\leq l_n\},$
\end{center}
with $1\leq n\leq r-2$ when $\alpha_l\in\Theta$ and $\alpha_{l-1}\notin\Theta.$ For each $n\in\{1,...,r-2\}$, $M_n$ is equivalent to $N_n.$ We set $S_n=M_n\oplus N_n.$

$e)$\begin{center}
$V_{r-1}=\vspan\{u_{\tilde{l}_{r-2}+s,\tilde{l}_{r-2}+t}:1\leq t<s\leq l_{r-1}\}\cup\{w_{l,\tilde{l}_{r-2}+t}:1\leq t\leq l_{r-1}\},$
\end{center}
when $\alpha_l\in\Theta$ and $\alpha_{l-1}\notin\Theta.$ \hfill $\qed$
\end{prop}
\begin{prop}\label{2.13} Let $\Theta\in\Sigma$ and $l\geq 5.$

$a)$ If $\alpha_l\notin\Theta$ then 
\begin{equation}\label{D2}
\mathcal{B}_\Theta=\displaystyle\left(\bigcup\limits_{1\leq n<m\leq r}\mathcal{B}_{mn}\right)\cup\left(\bigcup\limits_{l_i>1}\mathcal{B}_i\right)
\end{equation}
is a $(\cdot,\cdot)-$orthonormal basis for $\mathfrak{m}_\Theta$ adapted to the subspaces of Proposition \ref{2.12}, where 
\begin{center}
$\mathcal{B}_{mn}=\{w_{\tilde{l}_{m-1}+s,\tilde{l}_{n-1}+t}:1\leq s\leq l_m,\ 1\leq t\leq l_n\}\cup\{u_{\tilde{l}_{m-1}+s,\tilde{l}_{n-1}+t}:1\leq s\leq l_m,\ 1\leq t\leq l_n\}$
\end{center}
and
\begin{center}
$\mathcal{B}_i=\{u_{\tilde{l}_{i-1}+s,\tilde{l}_{i-1}+t}:1\leq t<s\leq l_i\}$\ \ \ $(1\leq i\leq r).$
\end{center}
If we have $\{\alpha_{l-1},\alpha_l\}\subseteq\Theta$ instead of $\alpha_l\notin\Theta$ then $i$ extends only over $\{1,...,r-1\}.$

$b)$ If $\alpha_l\in\Theta$ and $\alpha_{l-1}\notin \Theta$ then
\begin{equation}\label{D3}
\mathcal{B}_\Theta=\displaystyle\left(\bigcup\limits_{1\leq n<m\leq r-2}\mathcal{B}_{mn}\right)\cup\left(\bigcup\limits_{\begin{subarray}{c}l_i>1\\1\leq i\leq r-2\end{subarray}}\mathcal{B}_i\right)\cup\left(\bigcup\limits_{n=1}^{r-2}\mathcal{B}^M_n\cup\mathcal{B}^N_n\right)\cup\mathcal{B}^V,
\end{equation}
is a $(\cdot,\cdot)-$orthonormal basis for $\mathfrak{m}_\Theta$ adapted to the subspaces of Proposition \ref{2.12}, where $\mathcal{B}_{mn}$ and $\mathcal{B}_i$ are as before and
\begin{center} 
$\begin{array}{l}
\mathcal{B}^M_n=\{w_{\tilde{l}_{r-2}+s,\tilde{l}_{n-1}+t}:1\leq s\leq l_{r-1},\ 1\leq t\leq l_n\}\cup\{u_{l,\tilde{l}_{n-1}+t}:1\leq t\leq l_n\},\\
\\
\mathcal{B}^N_n=\{u_{\tilde{l}_{r-2}+s,\tilde{l}_{n-1}+t}:1\leq s\leq l_{r-1},\ 1\leq t\leq l_n\}\cup\{w_{l,\tilde{l}_{n-1}+t}:1\leq t\leq l_n\},\\
\\
\mathcal{B}^V=\{u_{\tilde{l}_{r-2}+s,\tilde{l}_{r-2}+t}:1\leq t<s\leq l_{r-1}\}\cup\{w_{l,\tilde{l}_{r-2}+t}:1\leq t\leq l_{r-1}\}.
\end{array}$ 
\end{center}\hfill $\qed$
\end{prop}
\begin{prop}\label{2.14} Every invariant metric $A$ on a flag of $D_l,$ $l\geq 5$ is written in the basis of Proposition \ref{2.13} in the following form:

$a)$ If $\alpha_l\notin \Theta$

\ \ \ \ \ \ \ \ \ \ \ \ \ \ \ \ \ \ \ \ \ \ $[A|_{M_{mn}}]_{\mathcal{B}_{mn}}=\left(\begin{array}{cc}
\lambda^{(mn)}_1\text{I}_{W_{mn}} & b_{mn}\text{I} \\
 & \\
b_{mn}\text{I} & \lambda^{(mn)}_2\text{I}_{U_{mn}} \\
\end{array}\right),$ $1\leq n<m\leq r,$

\

\ \ \ \ \ \ \ \ \ \ \ \ \ \ \ \ \ \ \ \ \ \ $A|_{U_i}=\gamma^{(i)}\text
{I}_{U_i},$\  $1\leq i\leq r$ and $l_i>1.$\\

$b)$ If $\{\alpha_{l-1},\alpha_l\}\subseteq\Theta$

\ \ \ \ \ \ \ \ \ \ \ \ \ \ \ \ \ \ \ \ \ \ $[A|_{M_{mn}}]_{\mathcal{B}_{mn}}=\left(\begin{array}{cc}
\lambda^{(mn)}_1\text{I}_{W_{mn}} & b_{mn}\text{I} \\
 & \\
b_{mn}\text{I} & \lambda^{(mn)}_2\text{I}_{U_{mn}} \\
\end{array}\right),$ $1\leq n<m\leq r-1,$

\

\ \ \ \ \ \ \ \ \ \ \ \ \ \ \ \ \ \ \ \ \ \ $A|_{M_{rn}}=\lambda^{(rn)}\text{I}_{M_{rn}},$ $1\leq n\leq r-1,$

\

\ \ \ \ \ \ \ \ \ \ \ \ \ \ \ \ \ \ \ \ \ \ $A|_{U_i}=\gamma^{(i)}\text
{I}_{U_i},$\  $1\leq i\leq r-1$ and $l_i>1.$\\

$c)$ If $\alpha_l\in\Theta$ and $\alpha_{l-1}\notin\Theta$

\ \ \ \ \ \ \ \ \ \ \ \ \ \ \ \ \ \ \ \ \ \ $[A|_{M_{mn}}]_{\mathcal{B}_{mn}}=\left(\begin{array}{cc}
\lambda^{(mn)}_1\text{I}_{W_{mn}} & b_{mn}\text{I} \\
 & \\
b_{mn}\text{I} & \lambda^{(mn)}_2\text{I}_{U_{mn}} \\
\end{array}\right),$ $1\leq n<m\leq r-2,$

\

\ \ \ \ \ \ \ \ \ \ \ \ \ \ \ \ \ \ \ \ \ \ $A|_{U_i}=\gamma^{(i)}\text
{I}_{U_i},$\  $1\leq i\leq r-2$ and $l_i>1,$

\

\ \ \ \ \ \ \ \ \ \ \ \ \ \ \ \ \ \ \ \ \ \ $[A|_{S_n}]_{\mathcal{B}^M_{n}\cup\mathcal{B}^N_{n}}=\left(\begin{array}{cc}
\lambda^{(r-1,n)}_1\text{I}_{M_n} & \begin{array}{cc}b^{(1)}_{r-1,n}\text{I} & \textbf{0}\\ \textbf{0} & b^{(2)}_{r-1,n}\textbf{I}\end{array} \\
 & \\
\begin{array}{cc}b^{(1)}_{r-1,n}\text{I} & \textbf{0}\\ \textbf{0} & b^{(2)}_{r-1,n}\textbf{I}\end{array} & \lambda^{(r-1,n)}_2\text{I}_{N_n} \\
\end{array}\right),$ $1\leq n\leq r-2,$

\

\ \ \ \ \ \ \ \ \ \ \ \ \ \ \ \ \ \ \ \ \ \ $A|_{V_{r-1}}=\gamma^{(r-1)}\text{I}_{V_{r-1}}.$ 
\end{prop}
\begin{proof}
We have the result for $U_i,$ for $M_{rn}$ when $\{\alpha_{l-1},\alpha_l\}\subseteq\Theta$ and for $V_{r-1}$ when $\alpha_l\in\Theta$ and $\alpha_{l-1}\notin\Theta,$ because these subspaces are $K_\Theta-$invariant, irreducible and not equivalent to any other. For $M_{mn},$ let us take 
\begin{center}
$k=\left(\begin{array}{ccc}
P_1 & \cdots & \textbf{0}\\
\vdots & \ddots & \vdots\\
\textbf{0} & \cdots & P_r
\end{array}\right)\in K_\Theta,$
\end{center}
where $P_i=(p^i_{st})_{l_i\times l_i}\in O(l_i)$ and $\det(P)=1,$ then

\begin{center}
$\Ad(k)w_{\tilde{l}_{m-1}+s,\tilde{l}_{n-1}+t}=\displaystyle\sum\limits_{e=1}^{l_m}\sum\limits_{f=1}^{l_n}p^m_{es}p^n_{ft}w_{\tilde{l}_{m-1}+e,\tilde{l}_{n-1}+f}$
\end{center}
and
\begin{center}
$\Ad(k)u_{\tilde{l}_{m-1}+s,\tilde{l}_{n-1}+t}=\displaystyle\sum\limits_{e=1}^{l_m}\sum\limits_{f=1}^{l_n}p^m_{es}p^n_{ft}u_{\tilde{l}_{m-1}+e,\tilde{l}_{n-1}+f}.$
\end{center}
There exist real numbers $\lambda_1^{mn}>0$ and $b^{st}_{ef}$ such that 
\begin{center}
$Aw_{\tilde{l}_{m-1}+s,\tilde{l}_{n-1}+t}=\lambda_1^{(mn)}w_{\tilde{l}_{m-1}+s,\tilde{l}_{n-1}+t}+\displaystyle\sum\limits_{e=1}^{l_m}\sum\limits_{f=1}^{l_n}b^{st}_{ef}u_{\tilde{l}_{m-1}+e,\tilde{l}_{n-1}+f},$
\end{center}
so 
\begin{center}
$\begin{array}{ccl}
\Ad(k)\circ Aw_{\tilde{l}_{m-1}+s,\tilde{l}_{n-1}+t} & = & \displaystyle \lambda_1^{(mn)}\sum\limits_{e=1}^{l_m}\sum\limits_{f=1}^{l_n}p^{m}_{es}p^{n}_{ft}\ w_{\tilde{l}_{m-1}+e,\tilde{l}_{n-1}+f}\\
 & & \displaystyle+\sum\limits_{\tilde{e},e=1}^{l_m}\sum\limits_{\tilde{f},f=1}^{l_n}b^{st}_{ef}p^{m}_{\tilde{e}e}p^{n}_{\tilde{f}f}\ u_{\tilde{l}_{m-1}+\tilde{e},\tilde{l}_{n-1}+\tilde{f}}
\end{array}$
\end{center}
and
\begin{center}
$\begin{array}{ccl}
A\circ \Ad(k)w_{\tilde{l}_{m-1}+s,\tilde{l}_{n-1}+t} & = & \displaystyle \lambda_1^{(mn)}\sum\limits_{e=1}^{l_m}\sum\limits_{f=1}^{l_n}p^{m}_{es}p^{n}_{ft}\ w_{\tilde{l}_{m-1}+e,\tilde{l}_{n-1}+f}\\
 & & \displaystyle+\sum\limits_{\tilde{e},e=1}^{l_m}\sum\limits_{\tilde{f},f=1}^{l_n}b^{ef}_{\tilde{e}\tilde{f}}p^{m}_{es}p^{n}_{ft}\ u_{\tilde{l}_{m-1}+\tilde{e},\tilde{l}_{n-1}+\tilde{f}}.
\end{array}$
\end{center}
Since $A$ commutes with $\Ad(k)$, we obtain
\begin{center}
$\displaystyle\sum\limits_{e=1}^{l_m}\sum\limits_{f=1}^{l_n}b^{st}_{ef}p^{m}_{\tilde{e}e}p^{n}_{\tilde{f}f}=\sum\limits_{e=1}^{l_m}\sum\limits_{f=1}^{l_n}b^{ef}_{\tilde{e}\tilde{f}}p^{m}_{es}p^{n}_{ft}$\ \ for all $s,t,\tilde{e},\tilde{f}.$
\end{center}
We can proceed as in the proof of Proposition \ref{2.6} to show that $b_{\tilde{e}\tilde{f}}^{st}=0$ if $(s,t)\neq (\tilde{e},\tilde{f})$ and that $b^{st}_{st}=b_{\tilde{e}\tilde{f}}^{\tilde{e}\tilde{f}}=:b_{mn}$ for all $s,t,\tilde{e},\tilde{f}.$ Now, we shall prove the result for $S_n.$ When $\alpha_l\in\Theta$ and $\alpha_{l-1}\notin\Theta$ we have that $l_r=1.$ Also,
\begin{center}
$Aw_{\tilde{l}_{r-2}+s,\tilde{l}_{n-1}+t}=\lambda_1^{(r-1,n)}w_{\tilde{l}_{r-2}+s,\tilde{l}_{n-1}+t}+\displaystyle\sum\limits_{e=1}^{l_{r-1}}\sum\limits_{f=1}^{l_n}b^{st}_{ef}u_{\tilde{l}_{r-2}+e,\tilde{l}_{n-1}+f}+\sum\limits_{f=1}^{l_n}b^{st}_f w_{l,\tilde{l}_{n-1}+f},$
\end{center}
and
\begin{center}
$Au_{l,\tilde{l}_{n-1}+t}=\lambda_1^{(r-1,n)}u_{l,\tilde{l}_{n-1}+t}+\displaystyle\sum\limits_{e=1}^{l_{r-1}}\sum\limits_{f=1}^{l_n}b^{t}_{ef}u_{\tilde{l}_{r-2}+e,\tilde{l}_{n-1}+f}+\sum\limits_{f=1}^{l_n}b^{t}_{f} w_{l,\tilde{l}_{n-1}+f}.$ 
\end{center}

As before, since $\Ad(k)\circ Aw_{\tilde{l}_{r-2}+s,\tilde{l}_{n-1}+t}=A\circ \Ad(k)w_{\tilde{l}_{r-2}+s,\tilde{l}_{n-1}+t},$ then

\begin{equation}\label{D4}
\displaystyle\sum\limits_{e=1}^{l_{r-1}}\sum\limits_{f=1}^{l_n}b^{st}_{ef}p^{r-1}_{\tilde{e}e}p^{n}_{\tilde{f}f}=\sum\limits_{e=1}^{l_{r-1}}\sum\limits_{f=1}^{l_n}b^{ef}_{\tilde{e}\tilde{f}}p^{r-1}_{es}p^{n}_{ft}
\end{equation}
and
\begin{equation}\label{D5}
\displaystyle\sum\limits_{f=1}^{l_n}b^{st}_{f}p^{r}_{11}p^{n}_{\tilde{f}f}=\sum\limits_{e=1}^{l_{r-1}}\sum\limits_{f=1}^{l_n}b^{ef}_{\tilde{f}}p^{r-1}_{es}p^{n}_{ft}.
\end{equation}
By the same arguments in the proof of Proposition \ref{2.6}, equations \eqref{D4} and \eqref{D5} implies $b_{\tilde{e}\tilde{f}}^{st}=0$ if $(s,t)\neq (\tilde{e},\tilde{f}),$ $b^{st}_{st}=b_{\tilde{e}\tilde{f}}^{\tilde{e}\tilde{f}}=:b^{(1)}_{r-1,n}$ and $b^{st}_{\tilde{f}}=0$ for all $s,t,\tilde{e},\tilde{f}.$ On the other hand, $\Ad(k)\circ Au_{l,\tilde{l}_{n-1}+t}=A\circ \Ad(k)u_{l,\tilde{l}_{n-1}+t}$ implies
\begin{equation}\label{D6}
\displaystyle\sum\limits_{f=1}^{l_n}b^{t}_{f}p^{n}_{\tilde{f}f}=\sum\limits_{f=1}^{l_n}b^{f}_{\tilde{f}}p^{n}_{ft}
\end{equation}
and
\begin{equation}\label{D7}
\displaystyle\sum\limits_{e=1}^{l_{r-1}}\sum\limits_{f=1}^{l_n}b^{t}_{ef}p^{r-1}_{\tilde{e}e}p^{n}_{\tilde{f}f}=\sum\limits_{f=1}^{l_n}b^{f}_{\tilde{e}\tilde{f}}p^{r}_{11}p^{n}_{ft}.
\end{equation}
Analogously to the proof of Proposition \ref{2.6}, we obtain $b^{t}_{\tilde{e}\tilde{f}}=0,$ $b^{t}_{\tilde{f}}=0$ if $t\neq \tilde{f}$ and $b^{t}_t=b^{\tilde{f}}_{\tilde{f}}=:b^{(2)}_{r-1,n}$ for all $t,\tilde{e},\tilde{f}.$
\end{proof}

\section{Geodesic Orbit Spaces}

In this section we shall find  all the invariant g.o. metrics in real flag manifolds.

\begin{dfn}\label{3.1}
Let $G/H$ be a homogeneous manifold with an invariant metric $g.$ A geodesic $\gamma$ starting at $eH$ is called \textit{homogeneous} if it is the orbit of a $1-$parameter subgroup of $G,$ that it

\begin{equation}\label{21}
\gamma(t)=exp(tX)H
\end{equation}

where $X$ is in the Lie algebra of $G.$ In this case, $X$ is called a \textit{geodesic vector.}
\end{dfn}

\begin{dfn}\label{3.2}
We say that a Riemannian homogeneous manifold $(G/H,g)$ ($g$ an invariant metric) is a \textit{g.o. space} if every geodesic starting at $eH$ is homogeneous. In this case $g$ is called a \textit{g.o. metric.}
\end{dfn}

Let us consider a homogeneous compact manifold $G/H,$ $(\cdot,\cdot)$ an $\Ad(G)-$invariant  inner product in the Lie algebra $\mathfrak{g}$ of $G$ and a $(\cdot,\cdot)-$orthogonal reductive decomposition $\mathfrak{g}=\mathfrak{h}\oplus\mathfrak{m}.$ In order to determine the real flag manifolds $(\mathbb{F}_\Theta, g)$ which are g.o. spaces, we use the following propositions.

\begin{prop}\label{3.3}(\cite{N}) $(G/H,g)$ is a g.o. space if and only if for all $X\in\mathfrak{m},$ there exist a vector $Z\in\mathfrak{h}$ such that 
\begin{equation}\label{22}
[Z+X,AX]=0
\end{equation}  

where $A$ is the metric operator corresponding to $g.$ \hfill $\qed$
\end{prop}

Since every invariant metric $A$ is a positive $(\cdot,\cdot)-$self-adjoint operator, we have that $\mathfrak{m}$ admits a decomposition $\mathfrak{m}=\mathfrak{m}_{\lambda_1}\oplus...\oplus\mathfrak{m}_{\lambda_r}$ into eigenspaces of $A$.

\begin{prop}\label{3.4}(\cite{N}) Let $(G/H,A)$ be a g.o. space.

$a)$ If $\lambda_1,\ \lambda_2$ are eigenvalues of $A,$ such that there exist $\Ad(H)-$invariant, pairwise $(\cdot,\cdot)-$orthogonal subspaces $\mathfrak{m}_1,$ $\mathfrak{m}_2$ of $\mathfrak{m}$ with 

\begin{center}
$\mathfrak{m}_i\subseteq \mathfrak{m}_{\lambda_i},\ i=1,2 \text{ and } 
[\mathfrak{m}_1,\mathfrak{m}_2]_{(\mathfrak{m}_1\oplus\mathfrak{m}_2)^{\perp}}\neq\{ 0 \},$
\end{center}

where the orthogonal complement is taking with respect to $(\cdot,\cdot)$, then $\lambda_1=\lambda_2.$

$b)$ If $\lambda_1,\ \lambda_2,$ $\lambda_3$ are eigenvalues of $A,$ such that there exist $\Ad(H)-$invariant, pairwise $(\cdot,\cdot)-$orthogonal subspaces $\mathfrak{m}_1,$ $\mathfrak{m}_2$ $\mathfrak{m}_3$ of $\mathfrak{m}$ with 

\begin{center}
$\mathfrak{m}_i\subseteq \mathfrak{m}_{\lambda_i},\ i=1,2,3 \text{ and } 
[\mathfrak{m}_1,\mathfrak{m}_2]_{\mathfrak{m}_3}\neq\{ 0 \},$
\end{center}

then $\lambda_1=\lambda_2=\lambda_3.$ \hfill $\qed$
\end{prop}

\textbf{Remark.} Every invariant metric $A$ which is a scalar multiple of the identity endomorphism is a trivial solution for equation \eqref{22} (by taking $Z=0$ for every $X),$ such a metric is called \textit{normal.} 
\subsection{Flags of $A_l$, $l\geq1$}
As in section 2, we fix $(\cdot,\cdot)=-\langle\cdot,\cdot\rangle,$ where $\langle\cdot,\cdot\rangle$ is the killing form of $\mathfrak{so}(l+1)$ and the $(\cdot,\cdot)-$orthogonal basis $\{w_{ij}:1\leq j<i\leq l+1\}.$

\begin{prop}\label{3.5} Let $\mathbb{F}_\Theta$ be a flag of $A_l$, with $l\neq 3$ or $l=3 $ and $\Theta\in\{\{\alpha_1,\alpha_2\},\{\alpha_2,\alpha_3\}\}.$ Then, $(\mathbb{F}_\Theta,A)$ is a g.o. space if and only if $A$ is normal, i.e., $A=\mu\text{I}_{\mathfrak{m}_\Theta},$ $\mu>0.$
\end{prop}

\begin{proof} Let $A$ a g.o. metric in a flag $\mathbb{F}_\Theta$ of $A_l.$ If $l=3$ and $\Theta\in\{\{\alpha_1,\alpha_2\},\{\alpha_2,\alpha_3\}\}$, then, by Proposition \ref{2.3}, $A$ has the form $A=\mu\textbf{I}_{\mathfrak{m}_\Theta}.$ If $l\neq3,$ we have by Corollary \ref{2.2} that $A$ is determined by $\frac{r(r-1)}{2}$ positive numbers $\mu_{mn},$ $1\leq n<m\leq r$, such that $A|_{M_{nm}}=\mu_{mn}\text{I}_{M_{mn}},$ with $M_{mn}$ as in \eqref{8}. We shall prove that $\mu_{mm}=\mu_{m'n'}$ for all $(m,n),\ (m',n').$ First, we prove $\mu_{mn}=\mu_{m'n}$ for $1\leq n<m,m'\leq r.$ In fact, without loss of generality, let us suppose $m<m',$ then 

\begin{center}
$w_{\tilde{l}_{m'-1}+1,\tilde{l}_{m-1}+1}=\left[w_{\tilde{l}_{m-1}+1,\tilde{l}_{n-1}+1},w_{\tilde{l}_{m'-1}+1,\tilde{l}_{n-1}+1}\right]\in \left[M_{mn},M_{m'n}\right].$
\end{center}

Since $\tilde{l}_{m-1}+1\neq\tilde{l}_{n-1}+t$ for all $t\in\{1,...,l_n\},$ then 

\begin{center}
$(w_{\tilde{l}_{m'-1}+1,\tilde{l}_{m-1}+1},w_{\tilde{l}_{m-1}+s,\tilde{l}_{n-1}+t})=0$, $1\leq s\leq l_m,$ $1\leq t\leq l_n$
\end{center}
and
\begin{center}
$(w_{\tilde{l}_{m'-1}+1,\tilde{l}_{m-1}+1},w_{\tilde{l}_{m'-1}+s,\tilde{l}_{n-1}+t})=0$, $1\leq s\leq l_{m'},$ $1\leq t\leq l_n,$
\end{center}

thus, $w_{\tilde{l}_{m'-1}+1,\tilde{l}_{m-1}+1}\in(M_{mn}\oplus M_{m'n})^\perp.$ Evidently, $M_{mn}$ and $M_{m'n}$ are contained in the eigenspaces of $A$ corresponding to the eigenvalues $\mu_{mn}$ and $\mu_{m'n}$, respectively and, by Proposition \ref{2.1}, they are $K_\Theta-$invariant. Also $M_{mn}$ and $M_{m'n}$ are $(\cdot,\cdot)-$orthogonal. By Proposition \ref{3.4}, we conclude that $\mu_{mn}=\mu_{m'n}.$ Now, we will show that $\mu_{mn}=\mu_{mn'},$ $1\leq n,n'<m\leq r.$ Let us suppose $n<n',$ then 

\begin{center}
$w_{\tilde{l}_{n'-1}+1,\tilde{l}_{n-1}+1}=\left[w_{\tilde{l}_{m-1}+1,\tilde{l}_{n-1}+1},w_{\tilde{l}_{m-1}+1,\tilde{l}_{n'-1}+1}\right]\in \left[M_{mn},M_{mn'}\right].$
\end{center}

Since $\tilde{l}_{n-1}+1\neq\tilde{l}_{m-1}+t$ for all $t\in\{1,...,l_m\},$ then 

\begin{center}
$(w_{\tilde{l}_{n'-1}+1,\tilde{l}_{n-1}+1},w_{\tilde{l}_{m-1}+s,\tilde{l}_{n-1}+t})=0$, $1\leq s\leq l_m,$ $1\leq t\leq l_n$
\end{center}
and
\begin{center}
$(w_{\tilde{l}_{n'-1}+1,\tilde{l}_{n-1}+1},w_{\tilde{l}_{m-1}+s,\tilde{l}_{n'-1}+t})=0$, $1\leq s\leq l_{m},$ $1\leq t\leq l_{n'},$
\end{center}

thus, $w_{\tilde{l}_{n'-1}+1,\tilde{l}_{n-1}+1}\in(M_{mn}\oplus M_{mn'})^\perp.$ By Proposition \ref{3.4}, $\mu_{mn}=\mu_{mn'}.$ Because of the above, we have $\mu_{mn}=\mu_{m'n}=\mu_{m'n'}.$
\end{proof}

\begin{prop}\label{3.6} Let $\mathbb{F}_\Theta$ be a flag of $A_3$ with $\Theta\notin\{\{\alpha_1,\alpha_2\},\{\alpha_2,\alpha_3\}\}.$ Then, $(\mathbb{F}_\Theta,A)$ is a g.o. space if and only if $A$ is written in the basis \eqref{10} as

$[A]_{\mathcal{B}_{\emptyset}}=\left(\begin{array}{cccccc}
\mu & b & 0 & 0 & 0 & 0\\
b & \mu & 0 & 0 & 0 & 0\\
0 & 0 & \mu & -b & 0 & 0\\
0 & 0 & -b & \mu & 0 & 0\\
0 & 0 & 0 & 0 & \mu & b\\
0 & 0 & 0 & 0 & b & \mu\\
\end{array}\right),$ $[A]_{\mathcal{B}_{\{\alpha_1,\alpha_3\}}}=\left(\begin{array}{cccc}
\mu_1 & 0 & 0 & 0 \\
0 & \mu_1 & 0 & 0 \\
0 & 0 & \mu_2 & 0 \\
0 & 0 & 0 & \mu_2 \\
\end{array}\right)$ and

\

$[A]_{\mathcal{B}_{\Theta}}=\left(\begin{array}{ccccc}
\mu_1 & 0 & 0 & 0 & 0 \\
0 & \mu_2 & 0 & \pm\sqrt{\mu_2(\mu_2-\mu_1)} & 0 \\
0 & 0 & \mu_2 & 0 & \mp\sqrt{\mu_2(\mu_2-\mu_1)} \\
0 & \pm\sqrt{\mu_2(\mu_2-\mu_1)} & 0 & \mu_2 & 0 \\
0 & 0 & \mp\sqrt{\mu_2(\mu_2-\mu_1)} & 0 & \mu_2 \\
\end{array}\right),$ $\mu_2>\mu_1$\\

for $\Theta=\{\alpha_1\},\{\alpha_2\}$ or $\{\alpha_3\}.$
\end{prop}
\begin{proof} Let us analyse case by case:
 
$\bullet$ $\Theta=\emptyset.$

Since $\mathfrak{k}_{\emptyset}=\{0\},$ by Proposition \ref{3.3}, $A$ is a g.o. metric if and only if $[X,AX]=0$ for every $X\in\mathfrak{m}_\emptyset.$  We write $A$ as in Proposition \ref{2.3}, for $X=w_{21}+w_{31}+w_{41},$ we have $AX=\mu_1^{(1)}w_{21}+b_1w_{43}+\mu^{(2)}_1w_{31}+b_2w_{42}+b_3w_{32}+\mu^{(3)}_2w_{41}$ and $[X,AX]=0$ if and only if
\begin{center}
$(b_2+b_3)w_{21}+(\mu^{(3)}_2-\mu^{(2)}_1)w_{43}+(b_1-b_3)w_{31}+(\mu^{(3)}_2-\mu^{(1)}_1)w_{42}+(\mu^{(2)}_1-\mu^{(1)}_1)w_{32}-(b_1+b_2)w_{41}=0,$
\end{center}

thus, $\mu^{(1)}_1=\mu^{(2)}_1=\mu^{(3)}_2$ and $b_1=-b_2=b_3=:b.$ For $X=w_{43}+w_{42}+w_{41},$ we have $AX=b_1w_{21}+\mu_2^{(1)}w_{43}++b_2w_{31}+\mu^{(2)}_2w_{42}+b_3w_{32}+\mu^{(3)}_2w_{41}$ and $[X,AX]=0$ if and only if 

\begin{center}
$(\mu^{(2)}_2-\mu^{(3)}_2)w_{21}-(b_2+b_3)w_{43}+(\mu^{(1)}_2-\mu^{(3)}_2)w_{31}+(b_3-b_1)w_{42}+(\mu^{(1)}_2-\mu^{(2)}_2)w_{32}+(b_1+b_2)w_{41}=0,$
\end{center}

concluding that $\mu^{(1)}_1=\mu^{(1)}_2=\mu^{(2)}_1=\mu^{(2)}_2=\mu^{(3)}_1=\mu^{(3)}_2=:\mu.$ It is a lengthy calculation to verify that if $\mu^{(1)}_1=\mu^{(1)}_2=\mu^{(2)}_1=\mu^{(2)}_2=\mu^{(3)}_1=\mu^{(3)}_2=:\mu$ and $b_1=-b_2=b_3=:b$, then $[X,AX]=0$ for all $X\in\mathfrak{m}_\emptyset.$

$\bullet$ $\Theta=\{\alpha_1,\alpha_3\}$

Let $A$ be a g.o. metric. By Proposition \ref{2.3}. $A$ has the form 
\begin{center}
$[A]_{\mathcal{B}_{\{\alpha_1,\alpha_3\}}}=\left(\begin{array}{cccc}
\mu_1 & 0 & 0 & 0 \\
0 & \mu_1 & 0 & 0 \\
0 & 0 & \mu_2 & 0 \\
0 & 0 & 0 & \mu_2 \\
\end{array}\right).$
\end{center}

Given $X=x_1(w_{31}-w_{42})+x_2(w_{41}+w_{32})+x_3(w_{31}+w_{42})+x_4(w_{41}-w_{32})$ written in the basis $\mathcal{B}_{\{\alpha_1,\alpha_3\}},$ we have

\begin{center}
$\begin{array}{ccl}
[X,AX] & = & \ \ \mu_1x_1x_2[w_{31}-w_{42},w_{41}+w_{32}]+\mu_2x_1x_4[w_{31}-w_{42},w_{41}-w_{32}]\\
\\
 & & +\mu_1x_1x_2[w_{41}+w_{32},w_{31}-w_{42}]+\mu_2x_2x_3[w_{41}+w_{32},w_{31}+w_{42}]\\
\\
 & & +\mu_1x_2x_3[w_{31}+w_{42},w_{41}+w_{32}]+\mu_2x_3x_4[w_{31}+w_{42},w_{41}-w_{32}]\\
\\
 & & +\mu_1x_1x_4[w_{41}-w_{32},w_{31}-w_{42}]+\mu_2x_3x_4[w_{41}-w_{32},w_{31}+w_{42}]\\
\\
& = & \ \ 2\mu_1x_1x_2(w_{21}+w_{43})-2\mu_1x_1x_2(w_{21}+w_{43})\\
\\
& & +2\mu_2x_3x_4(w_{43}-w_{21})-2\mu_2x_3x_4(w_{43}-w_{21})\\
\end{array}$
\end{center}

\ \ \ \ \ \ \ \ \ \ \ \ \ \ \ \ \ \ \ \ \ \ \ \ \ \ \ $\begin{array}{ccl}
& = & 0.
\end{array}$

Therefore $[0+X,AX]=0$ for all $X\in\mathfrak{m}_{\{\alpha_1,\alpha_3\}}$ and $A$ is g.o.

$\bullet$ $\Theta=\{\alpha_1\}.$

Since $\mathfrak{k}_{\{\alpha_1\}}=\vspan\{w_{21}\},$ then $A$ is a g.o. metric if and only if for all $X\in\mathfrak{m}_{\{\alpha_1\}},$ there exists $\lambda\in\mathbb{R}$ such that 

\begin{center}
$[\lambda w_{21}+X,AX]=0.$
\end{center}
 
By Proposition \ref{2.3} we have 

\begin{center}
$[A]_{\mathcal{B}_{\{\alpha_1\}}}=\left(\begin{array}{ccccc}
\mu^{(1)}_1 & 0 & 0 & 0 & 0 \\
0 & \mu^{(2)}_1 & 0 & b & 0 \\
0 & 0 & \mu^{(2)}_1 & 0 & -b \\
0 & b & 0 & \mu^{(2)}_2 & 0 \\
0 & 0 & -b & 0 & \mu^{(2)}_2 \\
\end{array}\right).$
\end{center}

We set $\mu_1:=\mu_1^{(1)}.$ For $X=w_{31}+w_{32}+w_{42}+w_{41},$ we take $\lambda\in\mathbb{R}$ such that $[\lambda w_{21}+X,AX]=0,$ then

\begin{center}
$\begin{array}{ccl}
0 & = & [\lambda w_{21}+X,AX] \\
\\
  & = & 2(\mu_2^{(2)}-\mu_1^{(2)})w_{43}-\lambda(\mu_1^{(2)}-b)w_{31}+\lambda(\mu^{(2)}_2-b)w_{42}+\lambda(\mu^{(2)}_1+b)w_{32}-\lambda(\mu^{(2)}_2+b)w_{41},
\end{array}$
\end{center}

by linear independence we have $\mu^{(2)}_2-\mu_1^{(2)}=0,$ i.e., $\mu^{(2)}_2=\mu_1^{(2)}=:\mu_2.$ Now, let us consider the vector $X=w_{43}+w_{32}+2w_{41}$ and $\lambda\in\mathbb{R}$ such that $[\lambda w_{21}+X,AX]=0,$ then 

\begin{center}
$\begin{array}{ccl}
0 & = & [\lambda w_{21}+X,AX] \\
\\
  & = & (2(\mu_1-\mu_2)+b-\lambda(\mu_2-2b))w_{31}+(\mu_2-\mu_1-2b+\lambda(2\mu_2-b))w_{42},
\end{array}$
\end{center}

by linear independence\begin{equation}\label{23}
2(\mu_1-\mu_2)+b-\lambda(\mu_2-2b)=0
\end{equation}
and \begin{equation}\label{24}
\mu_2-\mu_1-2b+\lambda(2\mu_2-b)=0,
\end{equation}multiplying equation \eqref{23} by $(2\mu_2-b)$ and equation \eqref{24} by $(\mu_2-2b)$, we obtain respectively\begin{equation}\label{25}
-4\mu_2^2+4\mu_1\mu_2+4b\mu_2-2b\mu_1-b^2-\lambda(2\mu_2-b)(\mu_2-2b)=0
\end{equation}
and \begin{equation}\label{26}
\mu_2^2-4b\mu_2-\mu_1\mu_2+2b\mu_1+4b^2+\lambda(2\mu_2-b)(\mu_2-2b)=0,
\end{equation}adding equations \eqref{25} and \eqref{26} we have\begin{equation}\label{27}
3(b^2-\mu_2(\mu_2-\mu_1))=0,
\end{equation}therefore $b^2-\mu_2(\mu_2-\mu_1)=0$, i.e., $b=\pm\sqrt{\mu_2(\mu_2-\mu_1)}.$

Conversely, let us suppose that $\mu_2:=\mu_1^{(2)}=\mu_2^{(2)},$ $\mu_1:=\mu^{(1)}_1$ and $b=\sqrt{\mu_2(\mu_2-\mu_1)},$ then, for $X=x_{43}w_{43}+x_{31}w_{31}+x_{42}w_{42}+x_{32}w_{32}+x_{41}w_{41}$ in $\mathfrak{m}_{\{\alpha_1\}}$ we have

\begin{center}
$\begin{array}{ccl}
[X,AX] & = & \ \ x_{43}\sqrt{\mu_2-\mu_1}(x_{32}\sqrt{\mu_2}-x_{41}\sqrt{\mu_2-\mu_1})w_{31}-x_{43}\sqrt{\mu_2-\mu_1}(x_{41}\sqrt{\mu_2}-x_{32}\sqrt{\mu_2-\mu_1})w_{42}\\
\\
 & & -x_{43}\sqrt{\mu_2-\mu_1}(x_{31}\sqrt{\mu_2}+x_{42}\sqrt{\mu_2-\mu_1})w_{32}+x_{43}\sqrt{\mu_2-\mu_1}(x_{42}\sqrt{\mu_2}+x_{31}\sqrt{\mu_2-\mu_1})w_{41}
\end{array}$
\end{center}

and for every $\lambda\in\mathbb{R}$
\begin{center}
$\begin{array}{ccl}
[\lambda w_{21},AX] & = & -\lambda\sqrt{\mu_2}(x_{32}\sqrt{\mu_2}-x_{41}\sqrt{\mu_2-\mu_1})w_{31}+\lambda\sqrt{\mu_2}(x_{41}\sqrt{\mu_2}-x_{32}\sqrt{\mu_2-\mu_1})w_{42}\\
\\
 & & +\lambda\sqrt{\mu_2}(x_{31}\sqrt{\mu_2}+x_{42}\sqrt{\mu_2-\mu_1})w_{32}-\lambda\sqrt{\mu_2}(x_{42}\sqrt{\mu_2}+x_{31}\sqrt{\mu_2-\mu_1})w_{41},
\end{array}$
\end{center}
thus
\begin{center}
$\begin{array}{ccl}
[\lambda w_{21}+X,AX] & = & \ \ (x_{43}\sqrt{\mu_2-\mu_1}-\lambda\sqrt{\mu_2})(x_{32}\sqrt{\mu_2}-x_{41}\sqrt{\mu_2-\mu_1})w_{31}\\
\\
 & & -(x_{43}\sqrt{\mu_2-\mu_1}-\lambda\sqrt{\mu_2})(x_{41}\sqrt{\mu_2}-x_{32}\sqrt{\mu_2-\mu_1})w_{42}\\
\\
 & & -(x_{43}\sqrt{\mu_2-\mu_1}-\lambda\sqrt{\mu_2})(x_{31}\sqrt{\mu_2}+x_{42}\sqrt{\mu_2-\mu_1})w_{32}\\
 \\
 & & +(x_{43}\sqrt{\mu_2-\mu_1}-\lambda\sqrt{\mu_2})(x_{42}\sqrt{\mu_2}+x_{31}\sqrt{\mu_2-\mu_1})w_{41}.
\end{array}$
\end{center}
By taking $\lambda=\sqrt{\frac{\mu_2-\mu_1}{\mu_2}}x_{43},$ we obtain $[\lambda w_{21}+X,AX]=0.$ For $b=-\sqrt{\mu_2(\mu_2-\mu_1)},$ the argument is analogous.

$\bullet$ $\Theta=\{\alpha_2\}$ or $\{\alpha_3\}.$

Let us consider the diffeomorphisms  

\begin{center}
$\begin{array}{cccc}
\varphi_i: & \mathbb{F}_{\{\alpha_1\}} & \longrightarrow & \mathbb{F}_{\{\alpha_i\}}\\
 & kK_{\{\alpha_1\}} & \longmapsto & e_i^Tke_iK_{\{\alpha_i\}}
\end{array},$ $i=2,3$
\end{center}

where $e_1=\left(\begin{array}{cccc}
0 & 1 & 0 & 0\\
0 & 0 & 1 & 0\\
1 & 0 & 0 & 0\\
0 & 0 & 0 & 1\\
\end{array}\right)$ and $e_2=\left(\begin{array}{cccc}
0 & 0 & 1 & 0\\
0 & 0 & 0 & 1\\
1 & 0 & 0 & 0\\
0 & 1 & 0 & 0\\
\end{array}\right).$ It is easy to verify that for every invariant metric

\begin{center}
$[A]_{\mathcal{B}_{\{\alpha_1\}}}=\left(\begin{array}{ccccc}
\mu^{(1)}_1 & 0 & 0 & 0 & 0 \\
0 & \mu^{(2)}_1 & 0 & b & 0 \\
0 & 0 & \mu^{(2)}_1 & 0 & -b \\
0 & b & 0 & \mu^{(2)}_2 & 0 \\
0 & 0 & -b & 0 & \mu^{(2)}_2 \\
\end{array}\right)$
\end{center}

we have

\begin{center}
$[\varphi_2^*A]_{\mathcal{B}_{\{\alpha_2\}}}=\left(\begin{array}{ccccc}
\mu^{(1)}_1 & 0 & 0 & 0 & 0 \\
0 & \mu^{(2)}_1 & 0 & -b & 0 \\
0 & 0 & \mu^{(2)}_1 & 0 & b \\
0 & -b & 0 & \mu^{(2)}_2 & 0 \\
0 & 0 & b & 0 & \mu^{(2)}_2 \\
\end{array}\right)$ and 
$[\varphi_3^*A]_{\mathcal{B}_{\{\alpha_3\}}}=\left(\begin{array}{ccccc}
\mu^{(1)}_1 & 0 & 0 & 0 & 0 \\
0 & \mu^{(2)}_1 & 0 & b & 0 \\
0 & 0 & \mu^{(2)}_1 & 0 & -b \\
0 & b & 0 & \mu^{(2)}_2 & 0 \\
0 & 0 & -b & 0 & \mu^{(2)}_2 \\
\end{array}\right).$
\end{center}

By Proposition \ref{2.3}, every invariant metric in $\mathbb{F}_{\{\alpha_i\}}$ has the form $\varphi_i^*A.$ Since $A$ is g.o. if and only if $\varphi_i^*A$ is g.o. then we obtain the result.
\end{proof}
\subsection{Flags of $B_l$, $l\geq 2$} We consider $(\cdot,\cdot)$ as in \eqref{B1} and the $(\cdot,\cdot)-$orthonormal basis in \eqref{Special1}.

\begin{prop}\label{3.7} Let $\mathbb{F}_\Theta$ be a flag of $B_l,$ $l\geq 5$ and $A$ an invariant metric as in Proposition \ref{2.6}.

$a)$ If $\alpha_l\notin\Theta,$ then  $(\mathbb{F}_\Theta,A)$ is a g.o. space if and only if 

\begin{equation}\label{B7}
\left\{\begin{array}{l}
\mu^{(1)}=...=\mu^{(r)}=:\mu,\\
\\
\lambda^{(mn)}_1=\lambda^{(mn)}_2=:\lambda, \text{ for all } (m,n),\\
\\
b_{mn}=:b, \text{ for all } (m,n),\\
\\
\gamma^{(i)}=:\gamma \text{ for all } i\in\{1,...,r\} \text{ with } l_i>1,\\
\\
\mu-\lambda=b\\
\\
\gamma=\frac{\lambda^2-b^2}{\lambda}
\end{array}\right.
\end{equation}

$b)$ If $\alpha_l\in\Theta,$ then  $(\mathbb{F}_\Theta,A)$ is a g.o. space if and only if 

\begin{equation}\label{B8}
\left\{\begin{array}{l}
\mu^{(1)}=...=\mu^{(r-1)}=:\mu,\\
\\
\rho^{(1)}=...=\rho^{(r-1)}=:\rho\\
\\
\lambda^{(mn)}_1=\lambda^{(mn)}_2=:\lambda, \text{ for all } (m,n),\\
\\
b_{mn}=:b, \text{ for all } (m,n),\\
\\
\gamma^{(i)}=:\gamma \text{ for all } i\in\{1,...,r-1\} \text{ with } l_i>1,\\
\\
\mu-\lambda=b=\lambda-\rho\\
\\
\gamma=\frac{2\mu\rho}{\mu+\rho}=\frac{\lambda^2-b^2}{\lambda}
\end{array}\right.
\end{equation}
\end{prop}
\begin{proof}
$a)$ Let us suppose that $(\mathbb{F}_\Theta,A)$ is a g.o. space. We have that for $i<j$, $V_i$ and $V_j$  are $(\cdot,\cdot)-$orthogonal, are $K_\Theta-$invariant, are contained in the eigenspaces corresponding to the eigenvalues $\mu^{(i)}$ and $\mu^{(j)}$ respectively and 

\begin{center}
$w_{\tilde{l}_{j-1}+1,\tilde{l}_{i-1}+1}+u_{\tilde{l}_{j-1}+1,\tilde{l}_{i-1}+1}=\displaystyle\left[v_{l_{i-1}+1},v_{l_{i-1}+1}\right]\in(V_i\oplus V_j)^{\perp}\cap\left[V_i,V_j\right]$
\end{center}
then, by Proposition \ref{3.4}, $\mu^{(i)}=\mu^{(j)}=\mu.$ For $X=w_{\tilde{l}_{m-1}+1,\tilde{l}_{n-1}+1}+w_{\tilde{l}_{m-1}+1,\tilde{l}_{n}+1},$ there exists a
$Z\in\mathfrak{k}_\Theta$ such that $[Z+X,AX]=0,$ but 

\begin{center}
$\begin{array}{ccl}
\left[X,AX\right] & = &\ \ \ \left[w_{\tilde{l}_{m-1}+1,\tilde{l}_{n-1}+1}+w_{\tilde{l}_{m-1}+1,\tilde{l}_{n}+1},\lambda_1^{(mn)}w_{\tilde{l}_{m-1}+1,\tilde{l}_{n-1}+1}+b_{mn}u_{\tilde{l}_{m-1}+1,\tilde{l}_{n-1}+1}\right]\\
\\
 & & +\left[w_{\tilde{l}_{m-1}+1,\tilde{l}_{n-1}+1}+w_{\tilde{l}_{m-1}+1,\tilde{l}_{n}+1},\lambda_1^{(m,n+1)}w_{\tilde{l}_{m-1}+1,\tilde{l}_{n}+1}+b_{m,n+1}u_{\tilde{l}_{m-1}+1,\tilde{l}_{n}+1}\right]\\
 \\
 & = &\lambda^{(m,n+1)}_1\left[w_{\tilde{l}_{m-1}+1,\tilde{l}_{n-1}+1},w_{\tilde{l}_{m-1}+1,\tilde{l}_{n}+1}\right]+b_{m,n+1}\left[w_{\tilde{l}_{m-1}+1,\tilde{l}_{n-1}+1},u_{\tilde{l}_{m-1}+1,\tilde{l}_{n}+1}\right]\\
\\
 &  &+\lambda^{(mn)}_1\left[w_{\tilde{l}_{m-1}+1,\tilde{l}_{n}+1},w_{\tilde{l}_{m-1}+1,\tilde{l}_{n-1}+1}\right]+b_{mn}\left[w_{\tilde{l}_{m-1}+1,\tilde{l}_{n}+1},u_{\tilde{l}_{m-1}+1,\tilde{l}_{n-1}+1}\right]\\
\\
 & = & (\lambda^{(m,n+1)}_1-\lambda^{(mn)}_1)w_{\tilde{l}_{n}+1,\tilde{l}_{n-1}+1}+(b_{m,n+1}-b_{mn})u_{\tilde{l}_{n}+1,\tilde{l}_{n-1}+1}\in M_{n+1,n}
\end{array}$
\end{center}

and since $M_{mn}\oplus M_{m,n+1}$ is $\mathfrak{k}_\Theta-$invariant, then $[Z,AX]\in M_{mn}\oplus M_{m,n+1}.$ By linear independence, $[Z+X,AX]=0\Rightarrow [Z,AX]=0=[X,AX],$ thus $b_{m,n+1}=b_{mn}$ and $\lambda^{(m,n+1)}_1=\lambda_1^{(mn)}.$ By taking $X=u_{\tilde{l}_{m-1}+1,\tilde{l}_{n-1}+1}+u_{\tilde{l}_{m-1}+1,\tilde{l}_{n}+1}$, $w_{\tilde{l}_{m-1}+1,\tilde{l}_{n-1}+1}+w_{\tilde{l}_{m}+1,\tilde{l}_{n-1}+1}$ or $u_{\tilde{l}_{m-1}+1,\tilde{l}_{n-1}+1}+u_{\tilde{l}_{m}+1,\tilde{l}_{n-1}+1}$ we obtain that

\begin{center}
$\begin{array}{ccl}
[X,AX] & = &(\lambda_2^{(m,n+1)}-\lambda_2^{(mn)})w_{\tilde{l}_{n}+1,\tilde{l}_{n-1}+1}+(b_{m,n+1}-b_{mn})u_{\tilde{l}_{n}+1,\tilde{l}_{n-1}+1},\\
\\
 & & (\lambda_1^{(m+1,n)}-\lambda_1^{(mn)})w_{\tilde{l}_{m}+1,\tilde{l}_{m-1}+1}+(b_{m+1,n}-b_{mn})u_{\tilde{l}_{m}+1,\tilde{l}_{m-1}+1}\\
\\ 
 & or & (\lambda_2^{(m+1,n)}-\lambda_2^{(mn)})w_{\tilde{l}_{m}+1,\tilde{l}_{m-1}+1}+(b_{m+1,n}-b_{mn})u_{\tilde{l}_{m}+1,\tilde{l}_{m-1}+1}
\end{array}$
\end{center}
respectively. As before, this implies $\lambda^{(m,n+1)}_2=\lambda_2^{(mn)},$ $\lambda_1^{(m+1,n)}=\lambda_1^{(mn)},$ $\lambda_2^{(m+1,n)}=\lambda_2^{(mn)}$ and $b_{m+1,n}=b_{mn}.$ Since this argument works for every pair $(m,n),$ then
\begin{center}
$\lambda^{(mn)}_1=\lambda^{(m'n')}_1=:\lambda_1$\\
$\lambda^{(mn)}_2=\lambda^{(m'n')}_2=:\lambda_2$\\
$b_{mn}=b_{m'n'}=:b$
\end{center} 
for all $m,n,m',n'.$ For $X=v_{\tilde{l}_1+1}+w_{\tilde{l}_1+1,1},$ take $Z\in\mathfrak{k}_\Theta$ such hat $[Z+X,AX]=0,$ then

\begin{center}
$\begin{array}{ccl}
\left[X,AX\right] & = & \left[v_{\tilde{l}_{1}+1}+w_{\tilde{l}_{1}+1,1},\mu\ v_{\tilde{l}_{1}+1}+\lambda_1w_{\tilde{l}_{1}+1,1}+b\ u_{\tilde{l}_{1}+1,1}\right]\\
 \\
 & = &\lambda_1\left[v_{\tilde{l}_{1}+1},w_{\tilde{l}_{1}+1,1}\right]+b\left[v_{\tilde{l}_{1}+1},u_{\tilde{l}_{1}+1,1}\right]+\mu\left[w_{\tilde{l}_{1}+1,1},v_{\tilde{l}_{1}+1}\right]\\
\\
 & = & (\lambda_1+b-\mu)v_1\in V_1
\end{array}$
\end{center}
and $[Z,AX]\in V_2\oplus M_{21}$. Thus, by linear independence we conclude that $\lambda_1+b-\mu=0$. Using the same argument for $X=v_{\tilde{l}_1+1}+u_{\tilde{l}_1+1,1},$ we can show that $\lambda_2+b-\mu=0$. Therefore, $\lambda_1=\lambda_2=:\lambda$ and $b=\mu-\lambda.$ 

Next, we will prove that $\gamma^{(i)}=\gamma^{(j)}$ for all $i, j$ with $l_i,l_j>1.$ In fact, if $l_i>1$ and $i\leq r-1$, take
\begin{center}
$X=v_{\tilde{l}_{i-1}+1}+u_{\tilde{l}_{i-1}+2,\tilde{l}_{i-1}+1}+w_{\tilde{l}_{i}+1,\tilde{l}_{i-1}+1}$
\end{center}
and $Z\in\mathfrak{k}_\Theta$ such that $[Z+X,AX]=0.$ Then, $AX=\mu\ v_{\tilde{l}_{i-1}+1}+\gamma^{(i)}u_{\tilde{l}_{i-1}+2,\tilde{l}_{i-1}+1}+\lambda\ w_{\tilde{l}_{i}+1,\tilde{l}_{i-1}+1}+b\ u_{\tilde{l}_{i}+1,\tilde{l}_{i-1}+1}$ and 

\begin{center}
$\begin{array}{ccl}
\left[X,AX\right] & = & \gamma^{(i)}\left[v_{\tilde{l}_{i-1}+1},u_{\tilde{l}_{i-1}+2,\tilde{l}_{i-1}+1}\right]+\lambda\left[v_{\tilde{l}_{i-1}+1},w_{\tilde{l}_{i}+1,\tilde{l}_{i-1}+1}\right]+b\left[v_{\tilde{l}_{i-1}+1},u_{\tilde{l}_{i}+1,\tilde{l}_{i-1}+1}\right]\\
 \\
 & &+\mu\left[u_{\tilde{l}_{i-1}+2,\tilde{l}_{i-1}+1},v_{\tilde{l}_{i-1}+1}\right]+\lambda\left[u_{\tilde{l}_{i-1}+2,\tilde{l}_{i-1}+1},w_{\tilde{l}_{i}+1,\tilde{l}_{i-1}+1}\right]\\
\\
& &+b\left[u_{\tilde{l}_{i-1}+2,\tilde{l}_{i-1}+1},u_{\tilde{l}_{i}+1,\tilde{l}_{i-1}+1}\right]+\mu\left[w_{\tilde{l}_{i}+1,\tilde{l}_{i-1}+1},v_{\tilde{l}_{i-1}+1}\right]\\
\\
 & & +\gamma^{(i)}\left[w_{\tilde{l}_{i}+1,\tilde{l}_{i-1}+1},u_{\tilde{l}_{i-1}+2,\tilde{l}_{i-1}+1}\right]\\
 \\
 & = & (\mu-\gamma^{(i)})v_{\tilde{l}_{i-1}+2}+(\mu-\lambda-b)v_{\tilde{l}_{i}+1}+(\lambda-\gamma^{(i)})u_{\tilde{l}_i+1,\tilde{l}_{i-1}+2}+b\ w_{\tilde{l}_i+1,\tilde{l}_{i-1}+2}\\
 \\
& = & (\mu-\gamma^{(i)})v_{\tilde{l}_{i-1}+2}+(\lambda-\gamma^{(i)})u_{\tilde{l}_i+1,\tilde{l}_{i-1}+2}+b\ w_{\tilde{l}_i+1,\tilde{l}_{i-1}+2}.
\end{array}$
\end{center}
We can write $Z=\displaystyle\sum\limits_{l_j>1}\left(\sum\limits_{1\leq t<s\leq l_j}z^{(j)}_{st}w_{\tilde{l}_{j-1}+s,\tilde{l}_{j-1}+t}\right),$ $z^{(j)}_{st}\in\mathbb{R},$ so

\begin{center}
$\begin{array}{ccl}
\left[Z,AX\right] & = & \mu z^{(i)}_{21}v_{\tilde{l}_{i-1}+2}+bz^{(i)}_{21}u_{\tilde{l}_i+1,\tilde{l}_{i-1}+2}+\lambda z^{(i)}_{21}w_{\tilde{l}_i+1,\tilde{l}_{i-1}+2}+Z',
\end{array}$
\end{center}
where $Z'$ is linearly independent of $\{v_{\tilde{l}_{i-1}+2},u_{\tilde{l}_i+1,\tilde{l}_{i-1}+2},w_{\tilde{l}_i+1,\tilde{l}_{i-1}+2}\}.$ Therefore,

\begin{center}
$[Z+X,AX]=(\mu-\gamma^{(i)}+\mu z^{(i)}_{21})v_{\tilde{l}_{i-1}+2}+(z^{(i)}_{21}b+\lambda-\gamma^{(i)})u_{\tilde{l}_i+1,\tilde{l}_{i-1}+2}+(b+\lambda z^{(i)}_{21})w_{\tilde{l}_i+1,\tilde{l}_{i-1}+2}+Z',$
\end{center}
so $Z'=0$ and
\begin{center}
$\left\{\begin{array}{l}
\mu-\gamma^{(i)}+\mu z^{(i)}_{21}=0\\
\\
z^{(i)}_{21}b+\lambda-\gamma^{(i)}=0\\
\\
b+\lambda z^{(i)}_{21}=0
\end{array}.\right.$
\end{center}
This implies $\gamma^{(i)}=\frac{\lambda^2-b^2}{\lambda}$ (and $z^{(i)}_{21}=-\frac{b}{\lambda}).$ If $i=r$ and $l_r>1$, we use the last argument but taking $X=v_{\tilde{l}_{r-1}+1}+u_{\tilde{l}_{r-1}+2,\tilde{l}_{r-1}+1}+w_{\tilde{l}_{r-1}+1,\tilde{l}_{r-2}+1}$ to conclude $\gamma^{(r)}=\frac{\lambda^2-b^2}{\lambda}.$ Therefore $\gamma^{(i)}=\gamma^{(j)}=:\gamma$ for all $i,j$ with $l_i,l_j>1$ and $\gamma=\frac{\lambda^2-b^2}{\lambda}.$

Conversely, let us suppose that $A$ satisfies the equations in \eqref{B7}. We can write every $X\in\mathfrak{m}_\Theta$ as

\begin{center}
$X=\displaystyle\sum\limits_{i=1}^r v^{(i)}+\sum\limits_{\begin{subarray}{c}i=1\\ l_i>1\end{subarray}}^{r}Y_i+\sum\limits_{1\leq n<m\leq r}X_{mn}+\sum\limits_{1\leq n<m\leq r}Y_{mn},$
\end{center}

where $v^{(i)}\in V_i,$ $Y_i\in U_i,$ $X_{mn}\in W_{mn},$ and $Y_{mn}\in U_{mn}.$ If

\begin{center}
$X_{mn}=\left(\begin{array}{ccc}
0 & \textbf{0} & \textbf{0}\\
\textbf{0} & C_{mn} &\textbf{0}\\
\textbf{0} & \textbf{0} & C_{mn}\\
\end{array}\right),$ $Y_{mn}=\left(\begin{array}{ccc}
0 & \textbf{0} & \textbf{0}\\
\textbf{0} & \textbf{0} & D_{mn}\\
\textbf{0} & D_{mn} & \textbf{0}\\
\end{array}\right)$ and \ $Y_{i}=\left(\begin{array}{ccc}
0 & \textbf{0} & \textbf{0}\\
\textbf{0} & \textbf{0} & D_{i}\\
\textbf{0} & D_{i} & \textbf{0}\\
\end{array}\right)$
\end{center}
we denote by
\begin{center}
$\tilde{X}_{mn}=\left(\begin{array}{ccc}
0 & \textbf{0} & \textbf{0}\\
\textbf{0} & \textbf{0} & C_{mn}\\
\textbf{0} & C_{mn} & \textbf{0}\\
\end{array}\right),$ $\tilde{Y}_{mn}=\left(\begin{array}{ccc}
0 & \textbf{0} & \textbf{0}\\
\textbf{0} & D_{mn} & \textbf{0}\\
\textbf{0} & \textbf{0} & D_{mn}\\
\end{array}\right)$ and \ $\tilde{Y}_{i}=\left(\begin{array}{ccc}
0 & \textbf{0} & \textbf{0}\\
\textbf{0} & D_{i} & \textbf{0}\\
\textbf{0} & \textbf{0} & D_{i}\\
\end{array}\right),$
\end{center}
so we have $AX_{mn}=\lambda X_{mn}+b\tilde{X}_{mn},$ $AY_{mn}=b\tilde{Y}_{mn}+\lambda Y_{mn}$ and 
\begin{center}
$AX=\displaystyle\mu\sum\limits_{i=1}^r v^{(i)}+\gamma\sum\limits_{\begin{subarray}{c}i=1\\ l_i>1\end{subarray}}^{r}Y_i+\lambda\sum\limits_{1\leq n<m\leq r}X_{mn}+b\sum\limits_{1\leq n<m\leq r}\tilde{X}_{mn}+b\sum\limits_{1\leq n<m\leq r}\tilde{Y}_{mn}+\lambda\sum\limits_{1\leq n<m\leq r}Y_{mn}.$
\end{center}
Then
\begin{center}
$\begin{array}{ccl}
\displaystyle [X,AX] & = & \displaystyle\gamma\left[\sum\limits_{i=1}^rv^{(i)},\sum\limits_{\begin{subarray}{c}i=1\\l_i>1\end{subarray}}^rY_i\right]+\lambda\left[\sum\limits_{i=1}^rv^{(i)},\sum\limits_{1\leq n<m\leq r}X_{mn}\right]+b\left[\sum\limits_{i=1}^rv^{(i)},\sum\limits_{1\leq n<m\leq r}\tilde{X}_{mn}\right]\\
\\
 & & \displaystyle+b\left[\sum\limits_{i=1}^rv^{(i)},\sum\limits_{1\leq n<m\leq r}\tilde{Y}_{mn}\right]+\lambda\left[\sum\limits_{i=1}^rv^{(i)},\sum\limits_{1\leq n<m\leq r}Y_{mn}\right]+\mu\left[\sum\limits_{\begin{subarray}{c}i=1\\l_i>1\end{subarray}}^rY_i,\sum\limits_{i=1}^rv^{(i)}\right]\\
 \\
 & & \displaystyle+\lambda\left[\sum\limits_{\begin{subarray}{c}i=1\\l_i>1\end{subarray}}^rY_i,\sum\limits_{1\leq n<m\leq r}X_{mn}\right]+b\left[\sum\limits_{\begin{subarray}{c}i=1\\l_i>1\end{subarray}}^rY_i,\sum\limits_{1\leq n<m\leq r}\tilde{X}_{mn}\right]+b\left[\sum\limits_{\begin{subarray}{c}i=1\\l_i>1\end{subarray}}^rY_i,\sum\limits_{1\leq n<m\leq r}\tilde{Y}_{mn}\right]\\
 \\
 & & \displaystyle+\lambda\left[\sum\limits_{\begin{subarray}{c}i=1\\l_i>1\end{subarray}}^rY_i,\sum\limits_{1\leq n<m\leq r}Y_{mn}\right]+\mu\left[\sum\limits_{1\leq n<m\leq r}X_{mn},\sum\limits_{i=1}^rv^{(i)}\right]+\gamma\left[\sum\limits_{1\leq n<m\leq r}X_{mn},\sum\limits_{\begin{subarray}{c}i=1\\l_i>1\end{subarray}}^rY_i\right]\\
 \\
 & & \displaystyle+b\left[\sum\limits_{1\leq n<m\leq r}X_{mn},\sum\limits_{1\leq n<m\leq r}\tilde{X}_{mn}\right]+b\left[\sum\limits_{1\leq n<m\leq r}X_{mn},\sum\limits_{1\leq n<m\leq r}\tilde{Y}_{mn}\right]\\
 \\
 & & \displaystyle+\lambda\left[\sum\limits_{1\leq n<m\leq r}X_{mn},\sum\limits_{1\leq n<m\leq r}Y_{mn}\right]+\mu\left[\sum\limits_{1\leq n<m\leq r}Y_{mn},\sum\limits_{i=1}^rv^{(i)}\right]\\
 \\
 & & \displaystyle+\gamma\left[\sum\limits_{1\leq n<m\leq r}Y_{mn},\sum\limits_{\begin{subarray}{c}i=1\\l_i>1\end{subarray}}^rY_i\right]+\lambda\left[\sum\limits_{1\leq n<m\leq r}Y_{mn},\sum\limits_{1\leq n<m\leq r}X_{mn}\right]\\
\\
 & & \displaystyle+b\left[\sum\limits_{1\leq n<m\leq r}Y_{mn},\sum\limits_{1\leq n<m\leq r}\tilde{X}_{mn}\right]+b\left[\sum\limits_{1\leq n<m\leq r}Y_{mn},\sum\limits_{1\leq n<m\leq r}\tilde{Y}_{mn}\right],
\end{array}$
\end{center}since
\begin{center}
$\begin{array}{l}
\displaystyle\left[\sum\limits_{1\leq n<m\leq r}X_{mn},\sum\limits_{1\leq n<m\leq r}\tilde{X}_{mn}\right]=0=\left[\sum\limits_{1\leq n<m\leq r}Y_{mn},\sum\limits_{1\leq n<m\leq r}\tilde{Y}_{mn}\right],\\
\\
\displaystyle\left[\sum\limits_{i=1}^rv^{(i)},\sum\limits_{1\leq n<m\leq r}X_{mn}\right]=\left[\sum\limits_{i=1}^rv^{(i)},\sum\limits_{1\leq n<m\leq r}\tilde{X}_{mn}\right],\\
\\
\displaystyle\left[\sum\limits_{i=1}^rv^{(i)},\sum\limits_{1\leq n<m\leq r}Y_{mn}\right]=\left[\sum\limits_{i=1}^rv^{(i)},\sum\limits_{1\leq n<m\leq r}\tilde{Y}_{mn}\right] \text{ and }\\ 
\end{array}$
\end{center}
\begin{center}
$\begin{array}{l}
\displaystyle\left[\sum\limits_{1\leq n<m\leq r}X_{mn},\sum\limits_{1\leq n<m\leq r}\tilde{Y}_{mn}\right]=-\left[\sum\limits_{1\leq n<m\leq r}Y_{mn},\sum\limits_{1\leq n<m\leq r}\tilde{X}_{mn}\right],
\end{array}$
\end{center}
then
\begin{center}
$\begin{array}{ccl}
\displaystyle [X,AX] & = & \displaystyle(\gamma-\mu)\left[\sum\limits_{i=1}^rv^{(i)},\sum\limits_{\begin{subarray}{c}i=1\\l_i>1\end{subarray}}^rY_i\right]+(\lambda+b-\mu)\left[\sum\limits_{i=1}^rv^{(i)},\sum\limits_{1\leq n<m\leq r}X_{mn}\right]\\
\\
& & \displaystyle+(\lambda+b-\mu)\left[\sum\limits_{i=1}^rv^{(i)},\sum\limits_{1\leq n<m\leq r}Y	_{mn}\right]+(\lambda-\gamma)\left[\sum\limits_{\begin{subarray}{c}i=1\\l_i>1\end{subarray}}^rY_i,\sum\limits_{1\leq n<m\leq r}X_{mn}\right]\\
\\
& & \displaystyle+(\lambda-\gamma)\left[\sum\limits_{\begin{subarray}{c}i=1\\l_i>1\end{subarray}}^rY_i,\sum\limits_{1\leq n<m\leq r}Y_{mn}\right]+b\left[\sum\limits_{\begin{subarray}{c}i=1\\l_i>1\end{subarray}}^rY_i,\sum\limits_{1\leq n<m\leq r}\tilde{X}_{mn}\right]\\
\\
& & \displaystyle+b\left[\sum\limits_{\begin{subarray}{c}i=1\\l_i>1\end{subarray}}^rY_i,\sum\limits_{1\leq n<m\leq r}\tilde{Y}_{mn}\right]\\
\\
& = & \displaystyle(\gamma-\mu)\left[\sum\limits_{i=1}^rv^{(i)},\sum\limits_{\begin{subarray}{c}i=1\\l_i>1\end{subarray}}^rY_i\right]+(\lambda-\gamma)\left[\sum\limits_{\begin{subarray}{c}i=1\\l_i>1\end{subarray}}^rY_i,\sum\limits_{1\leq n<m\leq r}X_{mn}\right]\\
\end{array}$
\end{center}
\begin{center}
$\begin{array}{ccl}
 & & \displaystyle+(\lambda-\gamma)\left[\sum\limits_{\begin{subarray}{c}i=1\\l_i>1\end{subarray}}^rY_i,\sum\limits_{1\leq n<m\leq r}Y_{mn}\right]+b\left[\sum\limits_{\begin{subarray}{c}i=1\\l_i>1\end{subarray}}^rY_i,\sum\limits_{1\leq n<m\leq r}\tilde{X}_{mn}\right]\\
 \\
 & & \displaystyle+b\left[\sum\limits_{\begin{subarray}{c}i=1\\l_i>1\end{subarray}}^rY_i,\sum\limits_{1\leq n<m\leq r}\tilde{Y}_{mn}\right].\\
\end{array}$
\end{center}
Last equality is because $\lambda+b-\mu=0.$ By taking $Z=\displaystyle-\frac{b}{\lambda}\sum\limits_{\begin{subarray}{c}i=1\\l_i>1\end{subarray}}^r\tilde{Y}_i,$ we obtain
\begin{center}
$\begin{array}{ccl}
\displaystyle [Z,AX] & = & \displaystyle-\frac{b\mu}{\lambda}\left[\sum\limits_{\begin{subarray}{c}i=1\\l_i>1\end{subarray}}^r\tilde{Y}_i,\sum\limits_{i=1}^rv^{(i)}\right]-b\left[\sum\limits_{\begin{subarray}{c}i=1\\l_i>1\end{subarray}}^r\tilde{Y}_i,\sum\limits_{1\leq n<m\leq r}X_{mn}\right]-b\left[\sum\limits_{\begin{subarray}{c}i=1\\l_i>1\end{subarray}}^r\tilde{Y}_i,\sum\limits_{1\leq n<m\leq r}Y_{mn}\right]\\
\end{array}$
\end{center}
\begin{center}
$\begin{array}{ccl}
& & \displaystyle-\frac{b^2}{\lambda}\left[\sum\limits_{\begin{subarray}{c}i=1\\l_i>1\end{subarray}}^r\tilde{Y}_i,\sum\limits_{1\leq n<m\leq r}\tilde{X}_{mn}\right]-\frac{b^2}{\lambda}\left[\sum\limits_{\begin{subarray}{c}i=1\\l_i>1\end{subarray}}^r\tilde{Y}_i,\sum\limits_{1\leq n<m\leq r}\tilde{Y}_{mn}\right],\\
\end{array}$
\end{center}
but 
\begin{center}
$\begin{array}{l}
\displaystyle\left[\sum\limits_{\begin{subarray}{c}i=1\\l_i>1\end{subarray}}^r\tilde{Y}_i,\sum\limits_{i=1}^rv^{(i)}\right]=\left[\sum\limits_{\begin{subarray}{c}i=1\\l_i>1\end{subarray}}^rY_i,\sum\limits_{i=1}^rv^{(i)}\right],\ \left[\sum\limits_{\begin{subarray}{c}i=1\\l_i>1\end{subarray}}^r\tilde{Y}_i,\sum\limits_{1\leq n<m\leq r}X_{mn}\right]=\left[\sum\limits_{\begin{subarray}{c}i=1\\l_i>1\end{subarray}}^rY_i,\sum\limits_{1\leq n<m\leq r}\tilde{X}_{mn}\right],\\
\\
\displaystyle\left[\sum\limits_{\begin{subarray}{c}i=1\\l_i>1\end{subarray}}^r\tilde{Y}_i,\sum\limits_{1\leq n<m\leq r}Y_{mn}\right]=\left[\sum\limits_{\begin{subarray}{c}i=1\\l_i>1\end{subarray}}^rY_i,\sum\limits_{1\leq n<m\leq r}\tilde{Y}_{mn}\right],\\
\\
\left[\sum\limits_{\begin{subarray}{c}i=1\\l_i>1\end{subarray}}^r\tilde{Y}_i,\sum\limits_{1\leq n<m\leq r}\tilde{X}_{mn}\right]=\left[\sum\limits_{\begin{subarray}{c}i=1\\l_i>1\end{subarray}}^rY_i,\sum\limits_{1\leq n<m\leq r}X_{mn}\right],\\
\\
\left[\sum\limits_{\begin{subarray}{c}i=1\\l_i>1\end{subarray}}^r\tilde{Y}_i,\sum\limits_{1\leq n<m\leq r}\tilde{Y}_{mn}\right]=\left[\sum\limits_{\begin{subarray}{c}i=1\\l_i>1\end{subarray}}^rY_i,\sum\limits_{1\leq n<m\leq r}Y_{mn}\right],
\end{array}$
\end{center}
then
\begin{center}
$\begin{array}{ccl}
\displaystyle [Z+X,AX] & = & \displaystyle\left(\gamma-\mu+\frac{b\mu}{\lambda}\right)\left[\sum\limits_{i=1}^rv^{(i)},\sum\limits_{\begin{subarray}{c}i=1\\l_i>1\end{subarray}}^rY_i\right]+\left(\lambda-\gamma-\frac{b^2}{\lambda}\right)\left[\sum\limits_{\begin{subarray}{c}i=1\\l_i>1\end{subarray}}^rY_i,\sum\limits_{1\leq n<m\leq r}X_{mn}\right]\\
\\
 & & \displaystyle+\left(\lambda-\gamma-\frac{b^2}{\lambda}\right)\left[\sum\limits_{\begin{subarray}{c}i=1\\l_i>1\end{subarray}}^rY_i,\sum\limits_{1\leq n<m\leq r}Y_{mn}\right]\\
 \\
 & = & 0\ \ \ \  \text{ (by \eqref{B7}). } 
\end{array}$
\end{center}

By Proposition \ref{3.3}, $A$ is a g.o. metric. 

$b)$ Let us suppose that $A$ is a g.o. metric. Interchanging $V_i$ by $(V_i)_2,$ the arguments of item $a)$ work to show that
\begin{center}
$\begin{array}{l}
\mu^{(1)}=...=\mu^{(r-1)}=:\mu,\\
\\
\lambda^{(mn)}_1=\lambda^{(mn)}_2=:\lambda, \text{ for all } (m,n),\\
\\
b_{mn}=\mu-\lambda=:b, \text{ for all } (m,n),\\
\\
\gamma^{(i)}=\frac{\lambda^2-b^2}{\lambda}=:\gamma \text{ for all } i\in\{1,...,r-1\} \text{ with } l_i>1.\\
\end{array}$
\end{center}
For $i<j$, we have that $(V_i)_1$ and $(V_j)_1$ are $K_\Theta-$invariant, $(\cdot,\cdot)-$orthogonal and are contained in the $A-$eigenspaces corresponding to $\rho^{(i)}$ and $\rho^{(j)}$ respectively. Also

\begin{center}
$2(w_{\tilde{l}_{j-1}+1,\tilde{l}_{i-1}+1}-u_{\tilde{l}_{j-1}+1,\tilde{l}_{i-1}+1})=\left[w_{\tilde{l}_{r-1}+1,\tilde{l}_{i-1}+1}-u_{\tilde{l}_{r-1}+1,\tilde{l}_{i-1}+1},w_{\tilde{l}_{r-1}+1,\tilde{l}_{j-1}+1}-u_{\tilde{l}_{r-1}+1,\tilde{l}_{j-1}+1}\right]$
\end{center}
is an non-zero element in $[(V_i)_1,(V_j)_1]\cap((V_i)_1\oplus(V_j)_1)^{\perp},$ then, by Proposition \ref{3.4}, $\rho^{(i)}=\rho^{(j)}.$ Therefore, $\rho^{(1)}=...=\rho^{(r-1)}=:\rho.$ For $X=w_{\tilde{l}_{r-1}+1,1}-u_{\tilde{l}_{r-1}+1,1}+w_{\tilde{l}_{r-2}+1,1},$ there exists a $Z\in\mathfrak{k}_\Theta$ such that $[Z+X,AX]=0,$ but

\begin{center}
$AX=\rho(w_{\tilde{l}_{r-1}+1,1}-u_{\tilde{l}_{r-1}+1,1})+\lambda w_{\tilde{l}_{r-2}+1,1}+bu_{\tilde{l}_{r-2}+1,1}$
\end{center}
and
\begin{center}
$\begin{array}{ccl}
[X,AX] & = & (\lambda-\rho-b)(w_{\tilde{l}_{r-1}+1,\tilde{l}_{r-2+}1}-u_{\tilde{l}_{r-1}+1,\tilde{l}_{r-2}+1})\in M_{r,r-1}. 
\end{array}$
\end{center}
Since $AX\in(V_1)_1\oplus M_{r-1,1}$ which is $\mathfrak{k}_\Theta-$invariant, we have $[Z,AX]\in(V_1)_1\oplus M_{r-1,1},$ so, by linear independence $\lambda-\rho-b=0$, i.e., $b=\lambda-\rho.$ Thus, $A$ satisfies \eqref{B8}.

Conversely, if $A$ satisfies \eqref{B8} and we write 
\begin{center}
$X=\displaystyle\sum\limits_{i=1}^{r-1} v^{(i)}+\sum\limits_{\begin{subarray}{c}i=1\\ l_i>1\end{subarray}}^{r-1}Y_i+\sum\limits_{1\leq n<m\leq r-1}X_{mn}+\sum\limits_{1\leq n<m\leq r-1}Y_{mn}+\sum\limits_{n=1}^{r-1}(X_{rn}+\tilde{X}_{rn})+\sum\limits_{n=1}^{r-1}(X'_{rn}-\tilde{X'}_{rn}),$
\end{center}
where $v^{(i)}\in \vspan\{v_{\tilde{l}_{i-1}+s}:1\leq s\leq l_i\},$ $Y_i\in U_i,$ $X_{mn}\in W_{mn},$ $Y_{mn}\in U_{mn},$ for $1\leq n<m\leq r-1$, $X_{rn},X'_{rn}\in \vspan\{w_{\tilde{l}_{r-1}+s,\tilde{l}_{n-1}+t}:1\leq t\leq l_n,\ 1\leq s\leq l_{r}\}$ and $\tilde{X}_{rn},\tilde{X'}_{rn}$ are as in the proof of item $a),$ then 
\begin{center}
$\begin{array}{ccl}
AX & = & \displaystyle\mu\sum\limits_{i=1}^{r-1} v^{(i)}+\gamma\sum\limits_{\begin{subarray}{c}i=1\\ l_i>1\end{subarray}}^{r-1}Y_i+\lambda\sum\limits_{1\leq n<m\leq r-1}X_{mn}+b\sum\limits_{1\leq n<m\leq r-1}\tilde{X}_{mn}+b\sum\limits_{1\leq n<m\leq r-1}\tilde{Y}_{mn}\\
\\
 & & +\displaystyle\lambda\sum\limits_{1\leq n<m\leq r-1}Y_{mn}+\mu\sum\limits_{n=1}^{r-1}(X_{rn}+\tilde{X}_{rn})+\rho\sum\limits_{n=1}^{r-1}(X'_{rn}-\tilde{X'}_{rn}).
\end{array}$
\end{center}
Similar to item $a),$ we have that
\begin{center}
$\begin{array}{ccl}
[X,AX] & = &\displaystyle(\gamma-\mu)\left[\sum\limits_{i=1}^{r-1} v^{(i)},\sum\limits_{\begin{subarray}{c}i=1\\ l_i>1\end{subarray}}^{r-1}Y_i\right]+(\lambda-\gamma)\left[\sum\limits_{\begin{subarray}{c}i=1\\ l_i>1\end{subarray}}^{r-1}Y_i,\sum\limits_{1\leq n<m\leq r-1}X_{mn}\right]\\
\\
& & +\displaystyle b\left[\sum\limits_{\begin{subarray}{c}i=1\\ l_i>1\end{subarray}}^{r-1}Y_i,\sum\limits_{1\leq n<m\leq r-1}\tilde{X}_{mn}\right]+b\left[\sum\limits_{\begin{subarray}{c}i=1\\ l_i>1\end{subarray}}^{r-1}Y_i,\sum\limits_{1\leq n<m\leq r-1}\tilde{Y}_{mn}\right]\ \ \ \ \ \ \ \ \ \ \ \ \ \ \ \ \ \\
\end{array}$
\end{center}
\begin{center}
$\begin{array}{ccl}\ \ \ \ \ \ \ \ \ \ \ \ \ \ 
& &\displaystyle+(\lambda-\gamma)\left[\sum\limits_{\begin{subarray}{c}i=1\\ l_i>1\end{subarray}}^{r-1}Y_i,\sum\limits_{1\leq n<m\leq r-1}Y_{mn}\right]+(\mu-\gamma)\left[\sum\limits_{\begin{subarray}{c}i=1\\ l_i>1\end{subarray}}^{r-1}Y_i,\sum\limits_{n=1}^{r-1}(X_{rn}+\tilde{X}_{rn})\right]\\
\\
& & \displaystyle+(\rho-\gamma)\left[\sum\limits_{\begin{subarray}{c}i=1\\ l_i>1\end{subarray}}^{r-1}Y_i,\sum\limits_{n=1}^{r-1}(X'_{rn}-\tilde{X'}_{rn})\right]. 
\end{array}$
\end{center}
By taking $Z=\displaystyle-\frac{b}{\lambda}\sum\limits_{\begin{subarray}{c}i=1\\ l_i>1\end{subarray}}^{r-1}\tilde{Y}_i\in\mathfrak{k}_\Theta,$ we have

\begin{center}
$\begin{array}{ccl}
\displaystyle [Z+X,AX] & = & \displaystyle\left(\gamma-\mu+\frac{b\mu}{\lambda}\right)\left[\sum\limits_{i=1}^{r-1}v^{(i)},\sum\limits_{\begin{subarray}{c}i=1\\l_i>1\end{subarray}}^{r-1}Y_i\right]+\left(\lambda-\gamma-\frac{b^2}{\lambda}\right)\left[\sum\limits_{\begin{subarray}{c}i=1\\l_i>1\end{subarray}}^{r-1}Y_i,\sum\limits_{1\leq n<m\leq r-1}X_{mn}\right]\\
\\
 & & \displaystyle+\left(\lambda-\gamma-\frac{b^2}{\lambda}\right)\left[\sum\limits_{\begin{subarray}{c}i=1\\l_i>1\end{subarray}}^{r-1}Y_i,\sum\limits_{1\leq n<m\leq r-1}Y_{mn}\right]\\
\end{array}$
\end{center}

\ \ \ \ \ \ \ \ \ \ \ \ \ \ \ \ \ \ \ \ \ \ \ \ \ \ \ \ \ $\displaystyle+\left(\mu-\gamma-\frac{b\mu}{\lambda}\right)\left[\sum\limits_{\begin{subarray}{c}i=1\\l_i>1\end{subarray}}^{r-1}Y_i,\sum\limits_{n=1}^{r-1}(X_{rn}+\tilde{X}_{rn})\right]$

\

\ \ \ \ \ \ \ \ \ \ \ \ \ \ \ \ \ \ \ \ \ \ \ \ \ \ \ \ \ $\displaystyle+\left(\rho-\gamma+\frac{b\rho}{\lambda}\right)\left[\sum\limits_{\begin{subarray}{c}i=1\\l_i>1\end{subarray}}^{r-1}Y_i,\sum\limits_{n=1}^{r-1}(X'_{rn}-\tilde{X'}_{rn})\right]$

\

\ \ \ \ \ \ \ \ \ \ \ \ \ \ \ \ \ \ \ \ \ \ \ \ \ = \ $0 \ \ \ \ \ (\text{by \eqref{B8}}).$

Thus, $A$ is a g.o. metric.
\end{proof}
\subsection{Flags of $C_l,$ $l\geq 3$} We fix $(\cdot,\cdot)$ as in equation \eqref{11} and the $(\cdot,\cdot)-$orthogonal basis 

\begin{center}
$\begin{array}{ll}
u_{kk}=E_{l+k,k}-E_{k,l+k}, & 1\leq k \leq l, \\
w_{ij}=E_{ij}-E_{ji}+E_{l+i,l+j}-E_{l+j,l+i}, & \\
u_{ij}=E_{l+i,j}+E_{l+j,i}-E_{i,l+j}-E_{j,l+i}, & 1\leq j<i\leq l.\\
\end{array}$
\end{center}

\begin{prop}\label{3.8}
Let $\mathbb{F}_\Theta$ be a flag of $C_l.$ Then, $(\mathbb{F}_\Theta,A)$ is a g.o. space if and only if $A$ is written in the basis \eqref{14} as:

\begin{center}
$[A|_{M_0}]_{\mathcal{B}_0}=\left(\begin{array}{ccccc}
\mu^{(0)}_1 & a_{21} & a_{31} & \dots & a_{\tilde{r}1}\\
a_{21} & \mu^{(0)}_2 & a_{32} & \dots & a_{\tilde{r}2} \\
a_{31} & a_{32} & \mu^{(0)}_3 & \dots & a_{\tilde{r}3} \\
\vdots & \vdots & \vdots & \ddots & \vdots \\
a_{\tilde{r}1} & a_{\tilde{r}2} & a_{\tilde{r}3} & \dots & \mu^{(0)}_{\tilde{r}} \\
\end{array}\right),$ \ \ \ $A|_{M_{mn}}=\mu\text{I}_{M_{mn}}$, \ \ \ $A|_{M_i}=\mu\text{I}_{M_i}$ ($l_i>1$)
\end{center}
where $\tilde{r}=r$ if $\alpha_l\notin\Theta$ and $\tilde{r}=r-1$ if $\alpha_l\in\Theta$, $1\leq n<m\leq r$, $1\leq i\leq \tilde{r}$, and 
\begin{equation}\label{28}
\begin{array}{l}\mu^{(0)}_{n'}=\frac{l_{n'}}{l_{m'}}\mu_{m'}^{(0)}+\left(1-\frac{l_{n'}}{l_{m'}}\right)\mu \\
\\
a_{m'n'}=\sqrt{\frac{l_{m'}}{l_{n'}}}(\mu_{n'}^{(0)}-\mu)=\sqrt{\frac{l_{n'}}{l_{m'}}}(\mu^{(0)}_{m'}-\mu),\end{array}\end{equation}
for all $1\leq n'<m'\leq \tilde{r}.$
\end{prop}
\begin{proof}
Let us suppose $A$ is a g.o. metric. We take $\Theta$ as in equation \eqref{12} and we write $A$ as in Proposition \ref{2.10}. First, we will show that $b_{mn}=0$. In fact, given $s\in\{1,...,l_m\}$ and $t\in\{1,...,l_n\},$ there exists $Z\in \mathfrak{k}_\Theta$ such that 
\begin{equation}\label{29}
\left[Z+w_{\tilde{l}_{m-1}+s,\tilde{l}_{n-1}+t},Aw_{\tilde{l}_{m-1}+s,\tilde{l}_{n-1}+t}\right]=0.
\end{equation}

But,

\begin{center}
$\begin{array}{ccl}
\left[ w_{\tilde{l}_{m-1}+s,\tilde{l}_{n-1}+t},Aw_{\tilde{l}_{m-1}+s,\tilde{l}_{n-1}+t}\right] & = & \left[ w_{\tilde{l}_{m-1}+s,\tilde{l}_{n-1}+t},\mu_1^{(mn)}w_{\tilde{l}_{m-1}+s,\tilde{l}_{n-1}+t}+b_{mn}u_{\tilde{l}_{m-1}+s,\tilde{l}_{n-1}+t}\right]\\
\\
 & = & b_{mn}\left[w_{\tilde{l}_{m-1}+s,\tilde{l}_{n-1}+t},u_{\tilde{l}_{m-1}+s,\tilde{l}_{n-1}+t}\right]\\
 \\
 & = & 2b_{mn}(u_{\tilde{l}_{m-1}+s,\tilde{l}_{m-1}+s}-u_{\tilde{l}_{n-1}+t,\tilde{l}_{n-1}+t})
\end{array}$
\end{center}

and $[Z,Aw_{\tilde{l}_{m-1}+s,\tilde{l}_{n-1}+t}]\in M_{mn}$ (this is because $M_{mn}$ is $K_\Theta-$invariant, so it is $\mathfrak{k}_\Theta-$invariant). Since $M_{mn}$ is $(\cdot,\cdot)-$orthogonal to $2b_{mn}(u_{\tilde{l}_{m-1}+s,\tilde{l}_{m-1}+s}-u_{\tilde{l}_{n-1}+t,\tilde{l}_{n-1}+t})$, then equation \eqref{29} implies $[w_{\tilde{l}_{m-1}+s,\tilde{l}_{n-1}+t},Aw_{\tilde{l}_{m-1}+s,\tilde{l}_{n-1}+t}]=0,$ concluding that $b_{mn}=0.$ Next, we prove that $\mu^{(mn)}_1=\mu^{(mn)}_2$. Since $b_{mn}=0$, we have that $W_{mn}$ and $U_{mn}$ are contained in the eigenspaces of $A$ corresponding to the eigenvalues $\mu_1^{(mn)}$ and $\mu^{(mn)}_2$ respectively. Also, they are $K_\Theta-$invariant and\begin{center}
$\begin{array}{ccl}
u_{\tilde{l}_{m-1}+1,\tilde{l}_{m-1}+1}-u_{\tilde{l}_{n-1}+1,\tilde{l}_{n-1}+1}=\left[w_{\tilde{l}_{m-1}+1,\tilde{l}_{n-1}+1},u_{\tilde{l}_{m-1}+1,\tilde{l}_{n-1}+1}\right]\in\left[W_{mn},U_{mn}\right]\cap (W_{mn}\oplus U_{mn})^\perp
\end{array}.$
\end{center}
By Proposition \ref{3.4} we have $\mu_1^{(mn)}=\mu_2^{(mn)}=:\mu^{(mn)}.$ To prove that $\mu^{(mn)}=\mu^{(m'n')}=:\mu$ for all $(m,n)$ and $(m',n'),$ is analogous to Proposition \ref{3.5}. To show the result for $A|_{M_0},$ we consider the vector $X=\frac{1}{\sqrt{l_n}}(u_{\tilde{l}_{n-1}+1,\tilde{l}_{n-1}+1}+...+u_{\tilde{l}_n,\tilde{l}_n})+w_{\tilde{l}_{m-1}+1,\tilde{l}_{n-1}+1},$ where $1\leq n<m\leq \tilde{r}$, then 

\begin{center}
$\begin{array}{ccl}
[X,AX] & = & \left[\frac{1}{\sqrt{l_n}}\sum\limits_{\begin{subarray}{c}i=1\\ \ \end{subarray}}^{l_n}u_{\tilde{l}_{n-1}+i,\tilde{l}_{n-1}+i}+w_{\tilde{l}_{m-1}+1,\tilde{l}_{n-1}+1},\mu_n^{(0)}\frac{1}{\sqrt{l_n}}\sum\limits_{i=1}^{l_n}u_{\tilde{l}_{n-1}+i,\tilde{l}_{n-1}+i}\right.\\
 \\
 & & \left.+\sum\limits_{\begin{subarray}{c}j=1\\j\neq n\end{subarray}}^{\tilde{r}}a_{jn}\frac{1}{\sqrt{l_j}}\sum\limits_{i=1}^{l_j}u_{\tilde{l}_{j-1}+i,\tilde{l}_{j-1}+i}+\mu\ w_{\tilde{l}_{m-1}+1,\tilde{l}_{n-1}+1}\Huge\right] \text{ (here }a_{jn}=a_{nj} \text{ for }j<n)\\
 \\
 & = & \frac{(\mu^{(0)}_n-\mu)}{\sqrt{l_n}}\left[w_{\tilde{l}_{m-1}+1,\tilde{l}_{n-1}+1},\sum\limits_{i=1}^{l_n}u_{\tilde{l}_{n-1}+i,\tilde{l}_{n-1}+i}\right]+\frac{a_{mn}}{\sqrt{l_m}}\left[w_{\tilde{l}_{m-1}+1,\tilde{l}_{n-1}+1},\sum\limits_{i=1}^{l_m}u_{\tilde{l}_{m-1}+i,\tilde{l}_{m-1}+i}\right]\\
\\
 & = & \frac{(\mu^{(0)}_n-\mu)}{\sqrt{l_n}}u_{\tilde{l}_{m-1}+1,\tilde{l}_{n-1}+1}-\frac{a_{mn}}{\sqrt{l_m}}u_{\tilde{l}_{m-1}+1,\tilde{l}_{n-1}+1}\\
 \\
 & = & \left(\frac{\mu^{(0)}_n-\mu}{\sqrt{l_n}}-\frac{a_{mn}}{\sqrt{l_m}}\right)u_{\tilde{l}_{m-1}+1,\tilde{l}_{n-1}+1}\in U_{mn}.
\end{array}$
\end{center} 

Since $M_0$ and $W_{mn}$ are $K_\Theta-$invariant, they are $\mathfrak{k}_\Theta-$invariant too, so is  $M_0\oplus W_{mn}.$ Also, we have $AX=\mu_n^{(0)}\frac{1}{\sqrt{l_n}}\sum\limits_{i=1}^{l_n}u_{\tilde{l}_{n-1}+i,\tilde{l}_{n-1}+i}+\sum\limits_{\begin{subarray}{c}j=1\\j\neq n\end{subarray}}^ra_{jn}\frac{1}{\sqrt{l_j}}\sum\limits_{i=1}^{l_j}u_{\tilde{l}_{j-1}+i,\tilde{l}_{j-1}+i}+\mu\ w_{\tilde{l}_{m-1}+1,\tilde{l}_{n-1}+1}\in M_0\oplus W_{mn},$ therefore
\begin{center}
$[Z,AX]\in M_0\oplus W_{mn}.$
\end{center}
By linear independence $[Z,AX]+[X,AX]=0$ implies $[X,AX]=0,$ thus, $\frac{\mu^{(0)}_n-\mu}{\sqrt{l_n}}-\frac{a_{mn}}{\sqrt{l_m}}=0.$ By taking $X=\frac{1}{\sqrt{l_m}}\sum\limits_{i=1}^{l_m}u_{\tilde{l}_{m-1}+i,\tilde{l}_{m-1}+i}+w_{\tilde{l}_{m-1}+1,\tilde{l}_{n-1}+1},$ where $1\leq n<m\leq \tilde{r},$ we have 
\begin{center}
$\begin{array}{ccl}
[X,AX] & = & \left(\frac{\mu-\mu_m^{(0)}}{\sqrt{l_m}}+\frac{a_{mn}}{\sqrt{l_n}}\right)u_{\tilde{l}_{m-1}+1,\tilde{l}_{n-1}+1}\in U_{mn}.
\end{array}$
\end{center}

As before, we conclude $\frac{\mu-\mu_m^{(0)}}{\sqrt{l_m}}+\frac{a_{mn}}{\sqrt{l_n}}=0.$ Summarizing, 
\begin{center}
$\left\{\begin{array}{l}
\frac{\mu^{(0)}_n-\mu}{\sqrt{l_n}}-\frac{a_{mn}}{\sqrt{l_m}}=0\\
\\
\frac{\mu^{(0)}_m-\mu}{\sqrt{l_m}}-\frac{a_{mn}}{\sqrt{l_n}}=0\\
\end{array}\right.,$ $1\leq n<m\leq \tilde{r},$ 
\end{center}

thus, $\mu^{(0)}_n=\frac{l_n}{l_m}\mu_m^{(0)}+\left(1-\frac{l_n}{l_m}\right)\mu$ \ and $a_{mn}=\sqrt{\frac{l_m}{l_n}}(\mu_n^{(0)}-\mu)=\sqrt{\frac{l_n}{l_m}}(\mu^{(0)}_m-\mu)$ for all $1\leq n<m\leq \tilde{r}.$ Now, we take $i$ such that $l_i>1$. If $1\leq i \leq r-1$, we have that $M_i$ and $W_{i+1,i}$ are contained in the eigenspaces corresponding to $\mu^{(i)}$ and $\mu$, respectively. Since

\begin{center}
$-u_{\tilde{l}_{i}+1,\tilde{l}_{i-1}+2}=\left[u_{\tilde{l}_{i-1}+2,\tilde{l}_{i-1}+1},w_{\tilde{l}_{i}+1,\tilde{l}_{i-1}+1}\right]\in\left[M_i,W_{i+1,i}\right]\cap(M_i\oplus W_{i+1,i})^\perp,$
\end{center}

then, by Proposition \ref{3.4}, we obtain 
$\mu^{(i)}=\mu.$ If $i=r$, then 

\begin{center}
$u_{\tilde{l}_{r-1}+2,\tilde{l}_{r-2}+1}=\left[u_{\tilde{l}_{r-1}+2,\tilde{l}_{r-1}+1},w_{\tilde{l}_{r-1}+1,\tilde{l}_{r-2}+1}\right]\in\left[M_r,W_{r,r-1}\right]\cap(M_r\oplus W_{r,r-1})^\perp,$
\end{center}
so $\mu^{(r)}=\mu$. Conversely, let us suppose $A$ has the form of the statement. Given $X\in\mathfrak{m}_\Theta$, we can write 

\begin{center}
$X=\displaystyle\sum\limits_{j=1}^{\tilde{r}}\frac{x_j}{\sqrt{l_j}}\sum\limits_{s=1}^{l_j}u_{\tilde{l}_{j-1}+s,\tilde{l}_{j-1}+s}+\sum\limits_{\begin{subarray}{c}i=1\\ l_i>1\end{subarray}}^{\tilde{r}}X_i+\sum\limits_{1\leq n<m\leq r}X_{mn}+\sum\limits_{1\leq n<m\leq r}Y_{mn},$
\end{center}
where $x_j\in\mathbb{R}$, $X_i\in M_i,$ $X_{mn}\in W_{mn}$ and $Y_{mn}\in U_{mn}.$ So
\begin{center}
$\begin{array}{ccl}
AX & = & \displaystyle\sum\limits_{j=1}^{\tilde{r}}\left(\frac{x_j}{\sqrt{l_j}}\mu_j^{(0)}\sum\limits_{s=1}^{l_j}u_{\tilde{l}_{j-1}+s,\tilde{l}_{j-1}+s}+\sum\limits_{\begin{subarray}{c}t=1\\t\neq j\end{subarray}}^{\tilde{r}}\frac{a_{tj}}{\sqrt{l_t}}\sum\limits_{s=1}^{l_t}u_{\tilde{l}_{t-1}+s,\tilde{l}_{t-1}+s}\right)+\mu\sum\limits_{\begin{subarray}{c}i=1\\ l_i>1\end{subarray}}^{\tilde{r}}X_i\\
\\
 & & \displaystyle+\mu\sum\limits_{1\leq n<m\leq r}X_{mn}+\mu\sum\limits_{1\leq n<m\leq r}Y_{mn}.
\end{array}$
\end{center}

We have two cases:

\textit{Case 1. $\alpha_l\notin\Theta:$} We have
\begin{center}
$\begin{array}{l}
\displaystyle [X,AX]=\displaystyle\ \mu\sum\limits_{1\leq n<m\leq r}\frac{x_n}{\sqrt{l_n}}\left[\sum\limits_{s=1}^{l_n}u_{\tilde{l}_{n-1}+s,\tilde{l}_{n-1}+s},X_{mn}\right]+\mu\sum\limits_{1\leq n<m\leq r}\frac{x_m}{\sqrt{l_m}}\left[\sum\limits_{s=1}^{l_m}u_{\tilde{l}_{m-1}+s,\tilde{l}_{m-1}+s},X_{mn}\right]\\
\\
\displaystyle+\mu\sum\limits_{1\leq n<m\leq r}\frac{x_n}{\sqrt{l_n}}\left[\sum\limits_{s=1}^{l_n}u_{\tilde{l}_{n-1}+s,\tilde{l}_{n-1}+s},Y_{mn}\right]+\mu\sum\limits_{1\leq n<m\leq r}\frac{x_m}{\sqrt{l_m}}\left[\sum\limits_{s=1}^{l_m}u_{\tilde{l}_{m-1}+s,\tilde{l}_{m-1}+s},Y_{mn}\right] \\
 \\
 \displaystyle+\sum\limits_{1\leq n<m\leq r}\left(\frac{x_n}{\sqrt{l_n}}\mu^{(0)}_n\left[X_{mn},\sum\limits_{s=1}^{l_n}u_{\tilde{l}_{n-1}+s,\tilde{l}_{n-1}+s}\right]+\frac{x_m}{\sqrt{l_m}}\mu_m^{(0)}\left[X_{mn},\sum\limits_{s=1}^{l_m}u_{\tilde{l}_{m-1}+s,\tilde{l}_{m-1}+s}\right]\right)\\
\\
\displaystyle+\sum\limits_{1\leq n<m\leq r}\left(\frac{x_n}{\sqrt{l_m}}a_{mn}\left[X_{mn},\sum\limits_{s=1}^{l_m}u_{\tilde{l}_{m-1}+s,\tilde{l}_{m-1}+s}\right]+\frac{x_m}{\sqrt{l_n}}a_{mn}\left[X_{mn},\sum\limits_{s=1}^{l_n}u_{\tilde{l}_{n-1}+s,\tilde{l}_{n-1}+s}\right]\right)\\
\\
\displaystyle+\sum\limits_{1\leq n<m\leq r}\sum\limits_{\begin{subarray}{c}j=1\\ j\neq m,n\end{subarray}}^r\left(\frac{x_j}{\sqrt{l_m}}a_{mj}\left[X_{mn},\sum\limits_{s=1}^{l_m}u_{\tilde{l}_{m-1}+s,\tilde{l}_{m-1}+s}\right]+\frac{x_j}{\sqrt{l_n}}a_{nj}\left[X_{mn},\sum\limits_{s=1}^{l_n}u_{\tilde{l}_{n-1}+s,\tilde{l}_{n-1}+s}\right]\right)\\
\\
 \displaystyle+\sum\limits_{1\leq n<m\leq r}\left(\frac{x_n}{\sqrt{l_n}}\mu^{(0)}_n\left[Y_{mn},\sum\limits_{s=1}^{l_n}u_{\tilde{l}_{n-1}+s,\tilde{l}_{n-1}+s}\right]+\frac{x_m}{\sqrt{l_m}}\mu_m^{(0)}\left[Y_{mn},\sum\limits_{s=1}^{l_m}u_{\tilde{l}_{m-1}+s,\tilde{l}_{m-1}+s}\right]\right)\\
\\
\displaystyle+\sum\limits_{1\leq n<m\leq r}\left(\frac{x_n}{\sqrt{l_m}}a_{mn}\left[Y_{mn},\sum\limits_{s=1}^{l_m}u_{\tilde{l}_{m-1}+s,\tilde{l}_{m-1}+s}\right]+\frac{x_m}{\sqrt{l_n}}a_{mn}\left[Y_{mn},\sum\limits_{s=1}^{l_n}u_{\tilde{l}_{n-1}+s,\tilde{l}_{n-1}+s}\right]\right)\\
\end{array}$
\end{center}
\begin{center}
$\begin{array}{l}
\displaystyle+\sum\limits_{1\leq n<m\leq r}\sum\limits_{\begin{subarray}{c}j=1\\ j\neq m,n\end{subarray}}^r\left(\frac{x_j}{\sqrt{l_m}}a_{mj}\left[Y_{mn},\sum\limits_{s=1}^{l_m}u_{\tilde{l}_{m-1}+s,\tilde{l}_{m-1}+s}\right]+\frac{x_j}{\sqrt{l_n}}a_{nj}\left[Y_{mn},\sum\limits_{s=1}^{l_n}u_{\tilde{l}_{n-1}+s,\tilde{l}_{n-1}+s}\right]\right),\ \ \ \ \ \ \ \ \ \\
\end{array}$
\end{center} 
since $\left[\sum\limits_{s=1}^{l_n}u_{\tilde{l}_{n-1}+s,\tilde{l}_{n-1}+s},Z\right]=-\left[\sum\limits_{s=1}^{l_m}u_{\tilde{l}_{m-1}+s,\tilde{l}_{m-1}+s},Z\right],$ $Z\in\{X_{mn},Y_{mn}\},$ $1\leq n<m\leq r,$ then 
\begin{center}
$\begin{array}{ccl}
\displaystyle [X,AX] & = & \displaystyle\ \ \sum\limits_{1\leq n<m\leq r}x_n\left(\frac{\mu^{(0)}_n-\mu}{\sqrt{l_n}}-\frac{a_{mn}}{\sqrt{l_m}}\right)\left[\sum\limits_{s=1}^{l_m}u_{\tilde{l}_{m-1}+s,\tilde{l}_{m-1}+s},X_{mn}\right]\\
\\
& & \displaystyle+\sum\limits_{1\leq n<m\leq r}x_m\left(\frac{\mu^{(0)}_m-\mu}{\sqrt{l_m}}-\frac{a_{mn}}{\sqrt{l_n}}\right)\left[\sum\limits_{s=1}^{l_n}u_{\tilde{l}_{n-1}+s,\tilde{l}_{n-1}+s},X_{mn}\right]\\
\\
 & & \displaystyle+\sum\limits_{1\leq n<m\leq r}x_n\left(\frac{\mu^{(0)}_n-\mu}{\sqrt{l_n}}-\frac{a_{mn}}{\sqrt{l_m}}\right)\left[\sum\limits_{s=1}^{l_m}u_{\tilde{l}_{m-1}+s,\tilde{l}_{m-1}+s},Y_{mn}\right]\\
\\
& & \displaystyle+\sum\limits_{1\leq n<m\leq r}x_m\left(\frac{\mu^{(0)}_m-\mu}{\sqrt{l_m}}-\frac{a_{mn}}{\sqrt{l_n}}\right)\left[\sum\limits_{s=1}^{l_n}u_{\tilde{l}_{n-1}+s,\tilde{l}_{n-1}+s},Y_{mn}\right]\ \ \ \ \ \ \ \ \ \ \ \ \ \ \ \ \ \\
\end{array}$
\end{center}
\begin{center}
$\begin{array}{ccl}
& & \displaystyle+\sum\limits_{1\leq n<m\leq r}\sum\limits_{\begin{subarray}{c}j=1\\ j\neq m,n\end{subarray}}^rx_j\left(\frac{a_{nj}}{\sqrt{l_n}}-\frac{a_{mj}}{\sqrt{l_m}}\right)\left[X_{mn},\sum\limits_{s=1}^{l_n}u_{\tilde{l}_{n-1}+s,\tilde{l}_{n-1}+s}\right]\\
\\
& & \displaystyle+\sum\limits_{1\leq n<m\leq r}\sum\limits_{\begin{subarray}{c}j=1\\ j\neq m,n\end{subarray}}^rx_j\left(\frac{a_{nj}}{\sqrt{l_n}}-\frac{a_{mj}}{\sqrt{l_m}}\right)\left[Y_{mn},\sum\limits_{s=1}^{l_n}u_{\tilde{l}_{n-1}+s,\tilde{l}_{n-1}+s}\right],\\
\end{array}$
\end{center}
by \eqref{28}, $\frac{a_{nj}}{\sqrt{l_n}}-\frac{a_{mj}}{\sqrt{l_m}}=\frac{\mu^{(0)}_m-\mu}{\sqrt{l_m}}-\frac{a_{mn}}{\sqrt{l_n}}=\frac{\mu^{(0)}_n-\mu}{\sqrt{l_n}}-\frac{a_{mn}}{\sqrt{l_m}}=0,$ thus $[X,AX]=0.$

\text{Case 2. $\alpha_l\in\Theta:$} Similarly as before we obtain

\begin{center}
$\begin{array}{ccl}
[X,AX] & = & \displaystyle\sum\limits_{n=1}^{r-1}\frac{1}{\sqrt{l_n}}\left(x_n(\mu_n^{(0)}-\mu)+\sum\limits_{\begin{subarray}{c}j=1\\ j\neq n\end{subarray}}^{r-1}a_{nj}x_j\right)\left[X_{rn},\sum\limits_{s=1}^{l_n}u_{\tilde{l}_{n-1}+s,\tilde{l}_{n-1}+s}\right]\\
\\
& & \displaystyle\sum\limits_{n=1}^{r-1}\frac{1}{\sqrt{l_n}}\left(x_n(\mu_n^{(0)}-\mu)+\sum\limits_{\begin{subarray}{c}j=1\\ j\neq n\end{subarray}}^{r-1}a_{nj}x_j\right)\left[Y_{rn},\sum\limits_{s=1}^{l_n}u_{\tilde{l}_{n-1}+s,\tilde{l}_{n-1}+s}\right].
\end{array}$
\end{center}

We consider $Z=\frac{x}{\sqrt{l_r}}\sum\limits_{s=1}^{l_r}u_{\tilde{l}_{r-1}+s,\tilde{l}_{r-1}+s}\in\mathfrak{k}_\Theta,$ for some $x\in\mathbb{R}$, then 

\begin{center}
$\begin{array}{ccl}
[Z,AX] & = & \displaystyle\sum\limits_{n=1}^{r-1}\frac{x}{\sqrt{l_r}}\mu\left[\sum\limits_{s=1}^{l_r}u_{\tilde{l}_{r-1}+s,\tilde{l}_{r-1}+s},X_{rn}\right]+\sum\limits_{n=1}^{r-1}\frac{x}{\sqrt{l_r}}\mu\left[\sum\limits_{s=1}^{l_r}u_{\tilde{l}_{r-1}+s,\tilde{l}_{r-1}+s},Y_{rn}\right],
\end{array}$
\end{center}

since $\left[\sum\limits_{s=1}^{l_r}u_{\tilde{l}_{r-1}+s,\tilde{l}_{r-1}+s},W\right]=-\left[\sum\limits_{s=1}^{l_n}u_{\tilde{l}_{n-1}+s,\tilde{l}_{n-1}+s},W\right]$ for $W\in\{X_{rn},Y_{rn}\}$ and  $1\leq n\leq r-1,$ we have 

 \begin{center}
$\begin{array}{ccl}
[Z+X,AX] & = & \displaystyle\sum\limits_{n=1}^{r-1}\left(\frac{1}{\sqrt{l_n}}\left(x_n(\mu_n^{(0)}-\mu)+\sum\limits_{\begin{subarray}{c}j=1\\ j\neq n\end{subarray}}^{r-1}a_{nj}x_j\right)+\frac{x\mu}{\sqrt{l_r}}\right)\left[X_{rn},\sum\limits_{s=1}^{l_n}u_{\tilde{l}_{n-1}+s,\tilde{l}_{n-1}+s}\right]\\
\\
 & & +\displaystyle\sum\limits_{n=1}^{r-1}\left(\frac{1}{\sqrt{l_n}}\left(x_n(\mu_n^{(0)}-\mu)+\sum\limits_{\begin{subarray}{c}j=1\\ j\neq n\end{subarray}}^{r-1}a_{nj}x_j\right)+\frac{x\mu}{\sqrt{l_r}}\right)\left[Y_{rn},\sum\limits_{s=1}^{l_n}u_{\tilde{l}_{n-1}+s,\tilde{l}_{n-1}+s}\right].
\end{array}$
\end{center}

We observe that 

\begin{center}
$\displaystyle\frac{1}{\sqrt{l_n}}\left(x_n(\mu_n^{(0)}-\mu)+\sum\limits_{\begin{subarray}{c}j=1\\ j\neq n\end{subarray}}^{r-1}a_{nj}x_j\right)+\frac{x\mu}{\sqrt{l_r}}=0,$ $n=1,...,r-1$ $\Longrightarrow$ $[Z+X,AX]=0,$
\end{center}

but 

\begin{center}
$\displaystyle\frac{1}{\sqrt{l_n}}\left(x_n(\mu_n^{(0)}-\mu)+\sum\limits_{\begin{subarray}{c}j=1\\ j\neq n\end{subarray}}^{r-1}a_{nj}x_j\right)+\frac{x\mu}{\sqrt{l_r}}=0$ $\Longleftrightarrow$ $x=\frac{\sqrt{l_r}}{\mu}\left(\sum\limits_{j=1}^{r-1}\sqrt{l_j}x_j\right)\left(\frac{\mu_n^{(0)}-\mu}{l_n}\right).$
\end{center}
Thus, it is enough to show that the number $\frac{\sqrt{l_r}}{\mu}\left(\sum\limits_{j=1}^{r-1}\sqrt{l_j}x_j\right)\left(\frac{\mu_n^{(0)}-\mu}{l_n}\right)$ does not depend on $n$. In fact,

\begin{center}
$\frac{\mu_n^{(0)}-\mu}{l_n}=\frac{\mu_n'^{(0)}-\mu}{l_n'}$ $\Longleftrightarrow$ $\frac{\sqrt{l_{n'}}}{\sqrt{l_n}}(\mu^{(0)}_n-\mu)=\frac{\sqrt{l_{n}}}{\sqrt{l_{n'}}}(\mu^{(0)}_{n'}-\mu)$
\end{center}
which is true by \eqref{28}.
\end{proof}
\subsection{Flags of $D_l,$ $l\geq 5.$} We consider the invariant inner product $(\cdot,\cdot)$ in \eqref{D1} and the $(\cdot,\cdot)-$orthonormal basis \eqref{special2}.

\begin{prop}\label{3.9} Let $\mathbb{F}_\Theta$ be a flag of $D_l,$ $l\geq5$ and $A$ an invariant metric as in Proposition \ref{2.14}.

$a)$ If $\alpha_l\notin\Theta,$ then $(\mathbb{F}_\Theta,A)$ is a g.o.space if and only if  

\begin{equation}\label{47}
\left\{\begin{array}{l}
\lambda^{(mn)}_1=\lambda^{(mn)}_2=:\lambda, \text{ for all } (m,n),\\
\\
b_{mn}=:b, \text{ for all } (m,n),\\
\\
\gamma^{(i)}=:\gamma \text{ for all } i\in\{1,...,r\} \text{ with } l_i>1,\\
\\
\gamma=\frac{\lambda^2-b^2}{\lambda}
\end{array}\right.
\end{equation}

$b)$ If $\{\alpha_{l-1},\alpha_l\}\in\Theta$, then $(\mathbb{F}_\Theta,A)$ is a g.o. space if and only if $A$ is normal.

$c)$ If $\alpha_l\in\Theta$ and $\alpha_{l-1}\notin\Theta,$ then $(\mathbb{F}_\Theta,A)$ is a g.o. space if and only if

\begin{equation}\label{48}
\left\{\begin{array}{l}
\lambda^{(mn)}_1=\lambda^{(mn)}_2=\lambda_1^{(r-1,n)}=\lambda_2^{(r-1,n)}=:\lambda, \text{ for all } (m,n),\\
\\
b_{mn}=b^{(1)}_{r-1,n}=b^{(2)}_{r-1,n}=:b, \text{ for all } (m,n),\\
\\
\gamma^{(i)}=:\gamma \text{ for all } i\in\{1,...,r-1\} \text{ with } l_i>1,\\
\\
\gamma=\frac{\lambda^2-b^2}{\lambda}
\end{array}\right.
\end{equation}
\end{prop}
\begin{proof}
$a)$ Let us suppose that $(\mathbb{F}_\Theta,A)$ is a g.o. space. We take $Z\in\mathfrak{k}_\Theta$ such that $[Z+X,AX],$ where $X=w_{\tilde{l}_{m-1}+1,\tilde{l}_{n-1}+1}+w_{\tilde{l}_{m}+1,\tilde{l}_{n-1}+1}.$ Then we have
\begin{center}
$AX=\lambda^{(mn)}_1w_{\tilde{l}_{m-1}+1,\tilde{l}_{n-1}+1}+b_{mn}u_{\tilde{l}_{m-1}+1,\tilde{l}_{n-1}+1}+\lambda^{(m+1,n)}_1w_{\tilde{l}_{m}+1,\tilde{l}_{n-1}+1}+b_{m+1,n}u_{\tilde{l}_{m}+1,\tilde{l}_{n-1}+1},$
\end{center}
and 
\begin{center}
$\begin{array}{ccl}
[X,AX] & = & \lambda^{(m+1,n)}_1\left[w_{\tilde{l}_{m-1}+1,\tilde{l}_{n-1}+1},w_{\tilde{l}_{m}+1,\tilde{l}_{n-1}+1}\right]+b_{m+1,n}\left[w_{\tilde{l}_{m-1}+1,\tilde{l}_{n-1}+1},u_{\tilde{l}_{m}+1,\tilde{l}_{n-1}+1}\right]\\
\\
 & & +\lambda^{(mn)}_1\left[w_{\tilde{l}_{m}+1,\tilde{l}_{n-1}+1},w_{\tilde{l}_{m-1}+1,\tilde{l}_{n-1}+1}\right]+b_{mn}\left[w_{\tilde{l}_{m}+1,\tilde{l}_{n-1}+1},u_{\tilde{l}_{m-1}+1,\tilde{l}_{n-1}+1}\right]\\
 \\
 & = & (\lambda_1^{(m+1,n)}-\lambda_1^{(mn)})w_{\tilde{l}_{m}+1,\tilde{l}_{m-1}+1}+(b_{m+1,n}-b_{mn})w_{\tilde{l}_{m}+1,\tilde{l}_{m-1}+1}\in M_{m+1,m}.
\end{array}$
\end{center}
Since $AX\in M_{mn}\oplus M_{m+1,n}$, we have $[Z,AX]\in M_{mn}\oplus M_{m+1,n}.$ Thus,
\begin{center}
$[Z+X,AX]=0\Longrightarrow [X,AX]=0=[Z,AX],$
\end{center}
in particular, $\lambda_1^{(m+1,n)}=\lambda_1^{(mn)}$ and $b_{m+1,n}=b_{mn}.$ We can use the same argument for $X=u_{\tilde{l}_{m-1}+1,\tilde{l}_{n-1}+1}+u_{\tilde{l}_{m}+1,\tilde{l}_{n-1}+1}$, $w_{\tilde{l}_{m-1}+1,\tilde{l}_{n-1}+1}+w_{\tilde{l}_{m-1}+1,\tilde{l}_{n}+1}$ and $u_{\tilde{l}_{m-1}+1,\tilde{l}_{n-1}+1}+u_{\tilde{l}_{m-1}+1,\tilde{l}_{n}+1}$ to conclude that
\begin{center}
$\lambda_2^{(m+1,n)}=\lambda_2^{(mn)},$ $\lambda_1^{(m,n+1)}=\lambda_1^{(mn)}$, $\lambda_2^{(m,n+1)}=\lambda_1^{(mn)}$ and $b_{m,n+1}=b_{mn}$.
\end{center}
Then, 
\begin{center}
$\lambda^{(mn)}_1=\lambda^{(m'n')}_1=:\lambda_1$\\
$\lambda^{(mn)}_2=\lambda^{(m'n')}_2=:\lambda_2$\\
$b_{mn}=b_{m'n'}=:b$
\end{center} 
for all $m,n,m',n'.$ Next, we will show that $\gamma^{(i)}=\frac{\lambda_1^2-b^2}{\lambda_1}=\frac{\lambda_2^2-b^2}{\lambda_2}$ for all $i\in\{1,...,r\}$ with $l_i>1.$ In fact, if $i\leq r-1$  we take $X=u_{\tilde{l}_{i-1}+2,\tilde{l}_{i-1}+1}+w_{\tilde{l}_{i}+1,\tilde{l}_{i-1}+1}$ and $Z\in\mathfrak{k}_\Theta$ such that $[Z+X,AX]=0.$ In this case, we can write
\begin{center}
$Z=\displaystyle\sum\limits_{l_j>1}\left(\sum\limits_{1\leq t<s\leq l_j}z^{(j)}_{st}w_{\tilde{l}_{j-1}+s,\tilde{l}_{j-1}+t}\right),$ $z^{(j)}_{st}\in\mathbb{R},$
\end{center}
therefore
\begin{center}
$\begin{array}{ccl}
[Z,AX]& = & \displaystyle\left[Z,\gamma^{(i)}u_{\tilde{l}_{i-1}+2,\tilde{l}_{i-1}+1}+\lambda_1w_{\tilde{l}_{i}+1,\tilde{l}_{i-1}+1}+bu_{\tilde{l}_{i}+1,\tilde{l}_{i-1}+1}\right]\\
\\
 & = & z^{(i)}_{21}b\ u_{\tilde{l}_{i}+1,\tilde{l}_{i-1}+1}+z^{(i)}_{21}\lambda_1w_{\tilde{l}_{i}+1,\tilde{l}_{i-1}+1}+Z',
\end{array}$
\end{center}
where $\{u_{\tilde{l}_{i}+1,\tilde{l}_{i-1}+1},w_{\tilde{l}_{i}+1,\tilde{l}_{i-1}+1},Z'\}$ is linear independent. On the other hand,
\begin{center}
$\begin{array}{ccl}
[X,AX]& = & \displaystyle\left[X,\gamma^{(i)}u_{\tilde{l}_{i-1}+2,\tilde{l}_{i-1}+1}+\lambda_1w_{\tilde{l}_{i}+1,\tilde{l}_{i-1}+1}+bu_{\tilde{l}_{i}+1,\tilde{l}_{i-1}+1}\right]\\
\\
 & = & (\lambda_1-\gamma^{(i)})u_{\tilde{l}_{i}+1,\tilde{l}_{i-1}+1}+b\ w_{\tilde{l}_{i}+1,\tilde{l}_{i-1}+1}.
\end{array}$
\end{center}
Thus, $[Z+X,AX]=0\Longrightarrow Z'=0,\ \lambda_1-\gamma^{(i)}+z_{21}^{(i)}b=b+z_{21}^{(i)}\lambda_1=0,$ so $\gamma^{(i)}=\frac{\lambda_1^2-b^2}{\lambda_1}$. If $i=r$ we take $X=u_{\tilde{l}_{r-1}+2,\tilde{l}_{r-1}+1}+w_{\tilde{l}_{r-1}+1,\tilde{l}_{r-2}+1}$ and proceeding as before we obtain $\gamma^{(r)}=\frac{\lambda_1^2-b^2}{\lambda_1}.$ To show that $\gamma^{(i)}=\frac{\lambda_2^2-b^2}{\lambda_2}$ we can use the same argument but taking $X=u_{\tilde{l}_{i-1}+2,\tilde{l}_{i-1}+1}+u_{\tilde{l}_{i}+1,\tilde{l}_{i-1}+1}$ instead of $u_{\tilde{l}_{i-1}+2,\tilde{l}_{i-1}+1}+w_{\tilde{l}_{i}+1,\tilde{l}_{i-1}+1}$ (when $i\leq r-1$) and $X=u_{\tilde{l}_{r-1}+2,\tilde{l}_{r-1}+1}+u_{\tilde{l}_{r-1}+1,\tilde{l}_{r-2}+1}$ instead of $u_{\tilde{l}_{r-1}+2,\tilde{l}_{r-1}+1}+w_{\tilde{l}_{r-1}+1,\tilde{l}_{r-2}+1}$ (when $i=r$). Summarizing,
\begin{center}
$\gamma^{(i)}=\frac{\lambda_1^2-b^2}{\lambda_1}=\frac{\lambda_2^2-b^2}{\lambda_2}$, for all $i\in\{1,...,r\}$ with $l_i>1.$
\end{center}
We observe that $\frac{\lambda_1^2-b^2}{\lambda_1}=\frac{\lambda_2^2-b^2}{\lambda_2}\Longleftrightarrow \lambda_1\lambda_2(\lambda_1-\lambda_2)=-b^2(\lambda_1-\lambda_2),$ therefore \begin{center}$\lambda_1\neq\lambda_2\Longrightarrow 0<\lambda_1\lambda_2=-b^2\leq 0,$\end{center} which is absurd. Thus, $\lambda_1=\lambda_2.$ We point out that when $\Theta=\emptyset$, the previous argument does not work to show $\lambda_1=\lambda_2=:\lambda$ (because there is no $i$ with $l_i>1$). In that case, we have $[X,AX]=0$ for all $X\in\mathfrak{m}_\Theta$ (because $\mathfrak{k}_\Theta=\{0\}$), in particular, when $X=w_{21}+u_{31}$, $[X,AX]=(\lambda_2-\lambda_1)u_{32},$ so $\lambda_1=\lambda_2.$ Now, we suppose $A$ satisfies \eqref{47}. Let 
\begin{center}
$X=\displaystyle\sum\limits_{\begin{subarray}{c}i=1\\ l_i>1\end{subarray}}^{r}Y_i+\sum\limits_{1\leq n<m\leq r}X_{mn}+\sum\limits_{1\leq n<m\leq r}Y_{mn}$
\end{center}
be a vector in $\mathfrak{m}_\Theta,$ where $Y_i\in U_i,$ $X_{mn}\in W_{mn}$ and $Y_{mn}\in U_{mn}.$ If 
\begin{center}
$X_{mn}=\left(\begin{array}{cc}
C_{mn} &\textbf{0}\\
\textbf{0} & C_{mn}\\
\end{array}\right),$ $Y_{mn}=\left(\begin{array}{cc}
\textbf{0} & D_{mn}\\
D_{mn} & \textbf{0}\\
\end{array}\right)$ and \ $Y_{i}=\left(\begin{array}{cc}
\textbf{0} & D_{i}\\
D_{i} & \textbf{0}\\
\end{array}\right)$
\end{center}
we set
\begin{center}
$\tilde{X}_{mn}=\left(\begin{array}{cc}
\textbf{0} & C_{mn}\\
C_{mn} & \textbf{0}\\
\end{array}\right),$ $\tilde{Y}_{mn}=\left(\begin{array}{cc}
D_{mn} & \textbf{0}\\
\textbf{0} & D_{mn}\\
\end{array}\right)$ and \ $\tilde{Y}_{i}=\left(\begin{array}{cc}
D_{i} & \textbf{0}\\
\textbf{0} & D_{i}\\
\end{array}\right).$
\end{center}
With this notation we have $AX_{mn}=\lambda X_{mn}+b\tilde{X}_{mn},$ $AX_{mn}=b \tilde{Y}_{mn}+\lambda Y_{mn},$ 
\begin{center}
$\begin{array}{l}
\displaystyle\left[\sum\limits_{1\leq n<m\leq r}X_{mn},\sum\limits_{1\leq n<m\leq r}\tilde{X}_{mn}\right]=0=\left[\sum\limits_{1\leq n<m\leq r}Y_{mn},\sum\limits_{1\leq n<m\leq r}\tilde{Y}_{mn}\right] \text{ and }\\
\\
\displaystyle\left[\sum\limits_{1\leq n<m\leq r}X_{mn},\sum\limits_{1\leq n<m\leq r}\tilde{Y}_{mn}\right]=-\left[\sum\limits_{1\leq n<m\leq r}Y_{mn},\sum\limits_{1\leq n<m\leq r}\tilde{X}_{mn}\right].
\end{array}$
\end{center}
Therefore
\begin{center}
$AX=\displaystyle\gamma\sum\limits_{\begin{subarray}{c}i=1\\ l_i>1\end{subarray}}^{r}Y_i+\lambda\sum\limits_{1\leq n<m\leq r}X_{mn}+b\sum\limits_{1\leq n<m\leq r}\tilde{X}_{mn}+b\sum\limits_{1\leq n<m\leq r}\tilde{Y}_{mn}+\lambda\sum\limits_{1\leq n<m\leq r}Y_{mn}$
\end{center}
and
\begin{center}
$\begin{array}{ccl}
\displaystyle [X,AX] & = & \displaystyle\lambda\left[\sum\limits_{\begin{subarray}{c}i=1\\l_i>1\end{subarray}}^rY_i,\sum\limits_{1\leq n<m\leq r}X_{mn}\right]+b\left[\sum\limits_{\begin{subarray}{c}i=1\\l_i>1\end{subarray}}^rY_i,\sum\limits_{1\leq n<m\leq r}\tilde{X}_{mn}\right]+b\left[\sum\limits_{\begin{subarray}{c}i=1\\l_i>1\end{subarray}}^rY_i,\sum\limits_{1\leq n<m\leq r}\tilde{Y}_{mn}\right]\\
 \\
 & & \displaystyle+\lambda\left[\sum\limits_{\begin{subarray}{c}i=1\\l_i>1\end{subarray}}^rY_i,\sum\limits_{1\leq n<m\leq r}Y_{mn}\right]+\gamma\left[\sum\limits_{1\leq n<m\leq r}X_{mn},\sum\limits_{\begin{subarray}{c}i=1\\l_i>1\end{subarray}}^rY_i\right]+\gamma\left[\sum\limits_{1\leq n<m\leq r}Y_{mn},\sum\limits_{\begin{subarray}{c}i=1\\l_i>1\end{subarray}}^rY_i\right]\\
 \\
 & = & \displaystyle(\lambda-\gamma)\left[\sum\limits_{\begin{subarray}{c}i=1\\l_i>1\end{subarray}}^rY_i,\sum\limits_{1\leq n<m\leq r}X_{mn}\right]+(\lambda-\gamma)\left[\sum\limits_{\begin{subarray}{c}i=1\\l_i>1\end{subarray}}^rY_i,\sum\limits_{1\leq n<m\leq r}Y_{mn}\right]\\
 \\
 & & \displaystyle+b\left[\sum\limits_{\begin{subarray}{c}i=1\\l_i>1\end{subarray}}^rY_i,\sum\limits_{1\leq n<m\leq r}\tilde{X}_{mn}\right]+b\left[\sum\limits_{\begin{subarray}{c}i=1\\l_i>1\end{subarray}}^rY_i,\sum\limits_{1\leq n<m\leq r}\tilde{Y}_{mn}\right].\\
\end{array}$
\end{center}
Let $Z=\displaystyle-\frac{b}{\lambda}\sum\limits_{\begin{subarray}{c}i=1\\ l_i>1\end{subarray}}^{r}\tilde{Y}_i\in\mathfrak{k}_\Theta,$ then 
\begin{center}
$\begin{array}{ccl}
\displaystyle [Z,AX] & = & \displaystyle-b\left[\sum\limits_{\begin{subarray}{c}i=1\\l_i>1\end{subarray}}^r\tilde{Y}_i,\sum\limits_{1\leq n<m\leq r}X_{mn}\right]-b\left[\sum\limits_{\begin{subarray}{c}i=1\\l_i>1\end{subarray}}^r\tilde{Y}_i,\sum\limits_{1\leq n<m\leq r}Y_{mn}\right]\\
\\
& & \displaystyle-\frac{b^2}{\lambda}\left[\sum\limits_{\begin{subarray}{c}i=1\\l_i>1\end{subarray}}^r\tilde{Y}_i,\sum\limits_{1\leq n<m\leq r}\tilde{X}_{mn}\right]-\frac{b^2}{\lambda}\left[\sum\limits_{\begin{subarray}{c}i=1\\l_i>1\end{subarray}}^r\tilde{Y}_i,\sum\limits_{1\leq n<m\leq r}\tilde{Y}_{mn}\right].\\
\end{array}$
\end{center}
Since

$\displaystyle\left[\sum\limits_{\begin{subarray}{c}i=1\\l_i>1\end{subarray}}^r\tilde{Y}_i,\sum\limits_{1\leq n<m\leq r}X_{mn}\right]=\left[\sum\limits_{\begin{subarray}{c}i=1\\l_i>1\end{subarray}}^rY_i,\sum\limits_{1\leq n<m\leq r}\tilde{X}_{mn}\right],\ \left[\sum\limits_{\begin{subarray}{c}i=1\\l_i>1\end{subarray}}^r\tilde{Y}_i,\sum\limits_{1\leq n<m\leq r}Y_{mn}\right]=\left[\sum\limits_{\begin{subarray}{c}i=1\\l_i>1\end{subarray}}^rY_i,\sum\limits_{1\leq n<m\leq r}\tilde{Y}_{mn}\right]$

\

$\displaystyle\left[\sum\limits_{\begin{subarray}{c}i=1\\l_i>1\end{subarray}}^r\tilde{Y}_i,\sum\limits_{1\leq n<m\leq r}\tilde{X}_{mn}\right]=\left[\sum\limits_{\begin{subarray}{c}i=1\\l_i>1\end{subarray}}^rY_i,\sum\limits_{1\leq n<m\leq r}X_{mn}\right],\ \left[\sum\limits_{\begin{subarray}{c}i=1\\l_i>1\end{subarray}}^r\tilde{Y}_i,\sum\limits_{1\leq n<m\leq r}\tilde{Y}_{mn}\right]=\left[\sum\limits_{\begin{subarray}{c}i=1\\l_i>1\end{subarray}}^rY_i,\sum\limits_{1\leq n<m\leq r}Y_{mn}\right],$

then 
\begin{center}
$\begin{array}{ccl}
\displaystyle [Z+X,AX] & = & \displaystyle\left(\lambda-\gamma-\frac{b^2}{\lambda}\right)\left[\sum\limits_{\begin{subarray}{c}i=1\\l_i>1\end{subarray}}^rY_i,\sum\limits_{1\leq n<m\leq r}X_{mn}\right]+\left(\lambda-\gamma-\frac{b^2}{\lambda}\right)\left[\sum\limits_{\begin{subarray}{c}i=1\\l_i>1\end{subarray}}^rY_i,\sum\limits_{1\leq n<m\leq r}Y_{mn}\right]\\
 \\
 & = & 0.
\end{array}$
\end{center}
Thus, $A$ is a g.o. metric.

$b)$ Let $A$ be a g.o. metric on $\mathbb{F}_\Theta.$ Analogously to item $a)$ we have
\begin{center}
$\lambda^{(mn)}_1=\lambda^{(m'n')}_1=:\lambda_1$\\
$\lambda^{(mn)}_2=\lambda^{(m'n')}_2=:\lambda_2$\\
$b_{mn}=b_{m'n'}=:b$
\end{center}
for $1\leq n<m\leq r-1$ and $1\leq n'<m'\leq r-1.$  Given $n<n'\leq r-1$, the subspaces $M_{rn}$ and $M_{rn'}$ are $K_\Theta-$invariant, $(\cdot,\cdot)-$orthogonal and are contained in the eigenspaces of $A$ corresponding to the eigenvalues $\lambda^{(rn)}$ and $\lambda^{(rn')}$ respectively. Also,
\begin{center}
$w_{\tilde{l}_{n'-1}+1,\tilde{l}_{n-1}+1}=\left[w_{\tilde{l}_{r-1}+1,\tilde{l}_{n-1}+1},w_{\tilde{l}_{r-1}+1,\tilde{l}_{n'-1}+1}\right]\in \left[M_{rn},M_{rn'}\right]\cap M_{n'n},$
\end{center}
and $M_{n'n}\subseteq(M_{rn}\oplus M_{rn'})^{\perp},$ therefore, by Proposition \ref{3.4}, $\lambda^{(rn)}=\lambda^{(rn')}=:\lambda^{(r)}.$ Next, we shall show that $b=0$ and $\lambda_1=\lambda_2=\lambda^{(r)}.$ In fact, if $X=w_{\tilde{l}_{r-1}+1,1}+w_{\tilde{l}_{r-2}+1,1},$ then
\begin{center}
$\begin{array}{ccl}
[X,AX] & = & \left[w_{\tilde{l}_{r-1}+1,1}+w_{\tilde{l}_{r-2}+1,1},\lambda^{(r)}w_{\tilde{l}_{r-1}+1,1}+\lambda_1w_{\tilde{l}_{r-2}+1,1}+b\ u_{\tilde{l}_{r-2}+1,1}\right]\\
\\
& = & (\lambda^{(r)}-\lambda_1)w_{\tilde{l}_{r-1}+1,\tilde{l}_{r-2}+1}-b\ u_{\tilde{l}_{r-1}+1,\tilde{l}_{r-2}+1}\in M_{r,r-1}.
\end{array}$
\end{center} 
Let $Z\in\mathfrak{k}_\Theta$ such that $[Z+X,AX]=0,$ since $AX\in M_{r1}\oplus M_{r-1,1}$ and $[X,AX]\in M_{r,r-1}$, then $[Z,AX]=0=[X,AX]$, so $\lambda^{(r)}=\lambda_1$ and $b=0.$ By taking $X=w_{\tilde{l}_{r-1}+1,1}+u_{\tilde{l}_{r-2}+1,1}$ instead of $w_{\tilde{l}_{r-1}+1,1}+w_{\tilde{l}_{r-2}+1,1}$ and proceeding analogously, we obtain $\lambda^{(r)}=\lambda_2.$ We define $\lambda:=\lambda_1=\lambda_2=\lambda^{(r)}.$ The same arguments of item $a)$ let us to show that for each $i\in\{1,...,r-1\}$ with $l_i>1$, $\gamma^{(i)}=\frac{\lambda^2-b^2}{\lambda}=\frac{\lambda^2}{\lambda}=\lambda.$ Therefore $A=\lambda\text{I}$, i.e., $A$ is normal.

$c)$ Analogously to item $a)$, we have that
\begin{center}
$\begin{array}{l}
\lambda^{(mn)}_1=:\lambda_1\\
\lambda^{(mn)}_2=:\lambda_2\\
b_{mn}=:b
\end{array}$ \ \ \ $1\leq n<m\leq r-2$
\end{center}
and $\gamma^{(i)}=\frac{\lambda_1^2-b^2}{\lambda_1}=\frac{\lambda_2^2-b^2}{\lambda_2}$ for every $i\in\{1,...,r-2\}$ with $l_i>1,$ in particular, $\lambda_1=\lambda_2=:\lambda.$ By taking $X=w_{\tilde{l}_{r-2}+1,\tilde{l}_{n-1}+1}+w_{\tilde{l}_{r-2}+1,\tilde{l}_{n}+1}$, we have
\begin{center}
$AX=\lambda_1^{(r-1,n)}w_{\tilde{l}_{r-2}+1,\tilde{l}_{n-1}+1}+b^{(1)}_{r-1,n}u_{\tilde{l}_{r-2}+1,\tilde{l}_{n-1}+1}+\lambda^{(r-1,n+1)}_1w_{\tilde{l}_{r-2}+1,\tilde{l}_{n}+1}+b^{(1)}_{r-1,n}u_{\tilde{l}_{r-2}+1,\tilde{l}_{n}+1},$
\end{center}
thus
\begin{center}
$[X,AX]=(\lambda_1^{(r-1,n+1)}-\lambda_1^{(r-1,n)})w_{\tilde{l}_{n}+1,\tilde{l}_{n-1}+1}+(b^{(1)}_{r-1,n+1}-b^{(1)}_{r-1,n})u_{\tilde{l}_{n}+1,\tilde{l}_{n-1}+1}\in M_{n+1,n}.$
\end{center}
There exists $Z\in\mathfrak{k}_\Theta$ such that $[Z+X,AX]=0$, but $AX\in S_{n}\oplus S_{n+1}\Longrightarrow [Z,AX]\in S_{n}\oplus S_{n+1},$ then, $[X,AX]=0$ $\Longrightarrow$ $\lambda_1^{(r-1,n+1)}=\lambda_1^{(r-1,n)}$ and $b^{(1)}_{r-1,n+1}=b^{(1)}_{r-1,n}.$ By using the same argument for $X=u_{\tilde{l}_{r-2}+1,\tilde{l}_{n-1}+1}+u_{\tilde{l}_{r-2}+1,\tilde{l}_{n}+1}$ and $w_{l,\tilde{l}_{n-1}+1}+w_{l,\tilde{l}_{n}+1}$ we obtain that $\lambda_2^{(r-1,n+1)}=\lambda_2^{(r-1,n)}$ and $b^{(2)}_{r-1,n+1}=b^{(2)}_{r-1,n},$ respectively. Thus
\begin{center}
$\left\{\begin{array}{l}
\lambda_i^{(r-1,1)}=\dots=\lambda_i^{(r-1,r-2)}=:\lambda_i^{(r-1)}\\
\\
b^{(i)}_{r-1,1}=\dots=b^{(i)}_{r-1,r-2}=:b^{(i)} 
\end{array},\right.$ \ \ \ $i=1,2.$
\end{center}
For $X=w_{\tilde{l}_{r-3}+1,1}+w_{\tilde{l}_{r-2}+1,1},$ $w_{\tilde{l}_{r-3}+1,1}+w_{l1}$ and $u_{\tilde{l}_{r-3}+1,1}+w_{l1}$, we have $\lambda_1^{(r-1)}=\lambda_1$, $b^{(1)}=b^{(2)}=b$ and $\lambda_2^{(r-1)}=\lambda_2.$ Now we will show that $\gamma^{(r-1)}=\frac{\lambda_1^2-b^2}{\lambda_1}=\frac{\lambda_2^2-b^2}{\lambda_2}$, in fact, if $X=w_{l1}+w_{l,\tilde{l}_{r-2}+1}$, then 
\begin{center}
$AX=b\ u_{l1}+\lambda_2w_{l1}+\gamma^{(r-1)}w_{l,\tilde{l}_{r-2}+1}$
\end{center}
and
\begin{center}
$[X,AX]=(\gamma^{(r-1)}-\lambda_2)w_{\tilde{l}_{r-2}+1,1}-b\ u_{\tilde{l}_{r-2}+1,1}.$
\end{center}
Let $Z\in\mathfrak{k}_\Theta$ such that $[Z+X,AX]=0.$ In this case we write $Z$ as
\begin{center}
$Z=\displaystyle\sum\limits_{\begin{subarray}{c}j\leq r-1\\l_j>1\end{subarray}}\left(\sum\limits_{1\leq t<s\leq l_j}z^{(j)}_{st}w_{\tilde{l}_{j-1}+s,\tilde{l}_{j-1}+t}\right)+\sum\limits_{t=1}^{l_{r-1}}z_tu_{l,\tilde{l}_{r-2}+t},$\ \ \ $z^{(j)}_{st},z_t\in\mathbb{R},$
\end{center}
therefore 
\begin{center}
$[Z,AX]=-z_1b\ w_{\tilde{l}_{r-2}+1,1}-z_1\lambda_2u_{\tilde{l}_{r-2}+1,1}+Z',$
\end{center}
where $\{w_{\tilde{l}_{r-2}+1,1},u_{\tilde{l}_{r-2}+1,1},Z'\}$ is linear independent. Since $[Z+X,AX]=0$, then 
\begin{center}
$\left\{\begin{array}{l}
\lambda_2-\gamma^{(r-1)}+z_1b=0\\
\\
b+z_1\lambda_2=0,
\end{array}\right.$
\end{center}
where we have $\gamma^{(r-1)}=\frac{\lambda_2^2-b^2}{\lambda_2}.$ When $X=u_{l1}+w_{l,\tilde{l}_{r-2}+1},$ we have $\gamma^{(r-1)}=\frac{\lambda_1^2-b^2}{\lambda_1}$. Conversely, if $A$ satisfies \eqref{48} every $X\in\mathfrak{m}_\Theta$ can be written as 
\begin{center}
$X=\displaystyle\sum\limits_{\begin{subarray}{c}i=1\\ l_i>1\end{subarray}}^{r-1}Y_i+\sum\limits_{1\leq n<m\leq r-1}X_{mn}+\sum\limits_{1\leq n<m\leq r-1}Y_{mn},$
\end{center}
where $Y_i\in U_i$ if $1\leq i \leq r-2$, $Y_{r-1}\in V_{r-1}$, $X_{mn}\in W_{mn}$, $Y_{mn}\in U_{mn}$ if $m\leq r-2,$ $X_{r-1,n}\in M_n$ and $Y_{r-1,n}\in N_n.$ For $i,m\leq r-2$ we consider $\tilde{X}_{mn},$ $\tilde{Y}_{mn}$ and $\tilde{Y}_i$ as in item $b)$, if  
\begin{center}
$X_{r-1,n}=\left(\begin{array}{cc}
A_{r-1,n} & B_{r-1,n}\\
B_{r-1,n} & A_{r-1,n}\\
\end{array}\right),$ $Y_{r-1,n}=\left(\begin{array}{cc}
C_{r-1,n} & D_{r-1,n}\\
D_{r-1,n} & C_{r-1,n}\\
\end{array}\right)$ and \ $Y_{r-1}=\left(\begin{array}{cc}
C_{r-1} & D_{r-1}\\
D_{r-1} & C_{r-1}\\
\end{array}\right)$
\end{center}
we set
\begin{center}
$\tilde{X}_{r-1,n}=\left(\begin{array}{cc}
B_{r-1,n} & A_{r-1,n}\\
A_{r-1,n} & B_{r-1,n}\\
\end{array}\right),$ $\tilde{Y}_{r-1,n}=\left(\begin{array}{cc}
D_{r-1,n} & C_{r-1,n}\\
C_{r-1,n} & D_{r-1,n}\\
\end{array}\right)$ and \ $\tilde{Y}_{r-1}=\left(\begin{array}{cc}
D_{r-1} & C_{r-1}\\
C_{r-1} & D_{r-1}\\
\end{array}\right).$
\end{center} 
Thus,
\begin{center}
$AX=\displaystyle\gamma\sum\limits_{\begin{subarray}{c}i=1\\ l_i>1\end{subarray}}^{r-1}Y_i+\lambda\sum\limits_{1\leq n<m\leq r-1}X_{mn}+b\sum\limits_{1\leq n<m\leq r-1}\tilde{X}_{mn}+b\sum\limits_{1\leq n<m\leq r-1}\tilde{Y}_{mn}+\lambda\sum\limits_{1\leq n<m\leq r-1}Y_{mn}$,
\end{center}
we can proceed exactly as in the proof of item $a)$ to conclude that $Z=\displaystyle-\frac{b}{\lambda}\sum\limits_{\begin{subarray}{c}i=1\\ l_i>1\end{subarray}}^{r-1}\tilde{Y}_i$ implies
\begin{center}
$\begin{array}{ccl}
\displaystyle [Z+X,AX] & = & \displaystyle\left(\lambda-\gamma-\frac{b^2}{\lambda}\right)\left[\sum\limits_{\begin{subarray}{c}i=1\\l_i>1\end{subarray}}^{r-1}Y_i,\sum\limits_{1\leq n<m\leq r-1}X_{mn}\right]+\left(\lambda-\gamma-\frac{b^2}{\lambda}\right)\left[\sum\limits_{\begin{subarray}{c}i=1\\l_i>1\end{subarray}}^{r-1}Y_i,\sum\limits_{1\leq n<m\leq r-1}Y_{mn}\right]\\
 \\
 & = & 0.
\end{array}$
\end{center}
Hence $A$ is a g.o. metric.
\end{proof}

\bibliographystyle{abbvr}

\renewcommand\refname{}

\end{document}